\documentclass[reqno,11pt,a4paper]{amsart}
\usepackage[utf8]{inputenc}
\usepackage{pifont}
\usepackage[T1]{fontenc}
\usepackage{enumerate}
\usepackage{textcomp}
\usepackage{xcolor}
\usepackage{pgffor}
\usepackage{enumitem}
\usepackage{yfonts}
\usepackage{calligra}
\usepackage{esint}
\usepackage{comment}
\usepackage{pdfsync}
\usepackage[colorlinks=true,linkcolor=blue]{hyperref}
\usepackage[T1]{fontenc}
\usepackage{pifont}
\usepackage[english]{babel}

\usepackage{ulem}
\usepackage{amsmath,esint,color}
\usepackage{amsthm,latexsym,epsfig,graphicx,marvosym, mathrsfs,bigints,stmaryrd,textgreek,upgreek}
\usepackage{amsmath,stmaryrd}
\usepackage{relsize}
\usepackage{amsfonts}
\usepackage{amssymb}
\usepackage[headings]{fullpage}
\numberwithin{equation}{section}
\allowdisplaybreaks

\usepackage{amsmath}
\usepackage{amsfonts}
\usepackage{amssymb}
\usepackage{amsthm}

\newcommand{\NN}{\mathbb{N}}
\newcommand{\TT}{\mathbb{T}}

\newcommand{\RR}{\mathbb{R}}

\newtheorem{theorem}{Theorem}[section]
\newtheorem{proposition}{Proposition}[section]
\newtheorem{lemma}{Lemma}[section]
\newtheorem{remark}{Remark}[section]
\newtheorem{coro}{Corollary}[section]
\newtheorem{definition}{Definition}[section]

\newcommand{\lr}[1]{\left( #1 \right)}
\DeclareMathOperator{\ebdiv}{div}

\renewcommand{\div}{{\rm div}\,}
\newcommand{\osc}{\operatorname{osc}}

\newcommand{\tcrb}{\textcolor{black!20!red}}

\begin{document}

\title[Regularity aspects of temperature in Heat-Conducting Fluid]{Regularity Aspects of Temperature \\ in Heat-Conducting Fluid Systems via De Giorgi Methods
}
\author[N.chaudhuri, P. B. Mucha, T. Piasecki, J. Yang]{
Nilasis Chaudhuri$^{1}$,
Piotr B. Mucha$^{1,*}$,
Tomasz Piasecki$^{1}$,
Jiakun Yang$^{2}$
}

\thanks{*\ \  Corresponding author}
\thanks{$^{1}$ Institute of Applied Mathematics and Mechanics, Faculty of Mathematics, Informatics and Mechanics,
University of Warsaw, ul. Banacha 2, Warsaw, 02-097, Poland.
Emails:
\texttt{nchaudhuri@mimuw.edu.pl},
\texttt{pbmucha@mimuw.edu.pl},
\texttt{tpiasecki@mimuw.edu.pl}.}

\thanks{$^{2}$School of Mathematical Sciences, Laboratory of Mathematics and Complex Systems (MOE), 
Beijing Normal University, Beijing 100875, P.R. China.
Email:
\texttt{yangjk@mail.bnu.edu.cn}}
\date{}
\maketitle
\tableofcontents

\begin{abstract}
This paper is concerned with the conditional Hölder regularity of the temperature for the heat-conducting compressible Navier–Stokes system, where the pressure takes a complete form that includes the Stefan–radiation correction. Under minimal assumptions on the density and velocity, we prove that any weak solution automatically exhibits Hölder continuity in the temperature component. This constitutes the main result of the work and is achieved by an application of the De Giorgi method.
As a consequence of this conditional regularity, we establish the global-in-time existence of a regular, unique solution to a simplified system for the compressible potential flow of a heat-conducting fluid, valid for arbitrarily large initial data.
\end{abstract}
\medskip
%
%
%
%

\section{Introduction}

The subject of this paper is the analysis of regular solutions to the equations governing the motion of a compressible viscous heat-conducting fluid. From a physical perspective, this system constitutes one of the simplest models of a viscous fluid that consistently incorporates the principle of energy conservation. In particular, the dissipation of kinetic energy caused by viscous friction is compensated by a corresponding increase in internal energy through the temperature equation.\par 
This mechanism gives rise to a rich mathematical structure and makes the system a natural object of analytical investigation. In contrast to the barotropic model, where the pressure depends solely on the density and no heat transfer is taken into account, the full compressible Navier--Stokes--Fourier system involves an additional evolution equation for the temperature. As a consequence, the density, velocity, and temperature are strongly coupled through both the constitutive relations and the energy balance. At the level of sufficiently regular solutions, the total energy of the system is conserved, providing a fundamental a priori estimate and an essential tool in the mathematical analysis.\par 
The presence of temperature significantly alters the qualitative behavior of solutions and introduces substantial analytical difficulties. The thermal effects generate additional nonlinearities and coupling mechanisms that are absent in the barotropic setting. Understanding how these features influence the regularity and long-time behavior of solutions is one of the central challenges in the mathematical theory of compressible viscous fluids.

The system describes the motion of a compressible, heat-conducting fluid
-- known often as the Navier--Stokes--Fourier system -- in the absence of external forces, formulated through the conservation of mass, momentum, and internal energy:
\begin{align}
    &\partial_t \rho + \operatorname{div}(\rho u) = 0, \label{nsf_mass}\\
    &\partial_t (\rho u) + \operatorname{div}(\rho u \otimes u) + \nabla p(\rho, \theta) = \operatorname{div} \mathbb{S}, \label{nsf_mom}\\
    &\partial_t (\rho e(\rho, \theta)) + \operatorname{div}(\rho e(\rho, \theta) u) + \operatorname{div} q = \mathbb{S} \colon {D} u - p \operatorname{div} u, \label{nsf_energy}
\end{align}
where $\rho = \rho(t,x)$ is the fluid density, $u = u(t,x)$ is the velocity vector field, $p = p(\rho, \theta)$ denotes the pressure, and $e = e(\rho, \theta)$ represents the specific internal energy. The absolute temperature is denoted by $\theta = \theta(t,x)$. The symmetric velocity gradient tensor is given by ${D} u = \frac{1}{2}(\nabla u + \nabla^T u)$, while $\mathbb{S}$ and $q$ stand for the viscous stress tensor and the heat flux, respectively.

The system is supplemented with the initial conditions
\begin{align}\label{initial_cond}
    \rho(0,\cdot) = \rho_0, \quad (\rho u)(0,\cdot) = (\rho u)_0, \quad \theta(0,\cdot) = \theta_0 \quad \text{in } \Omega,
\end{align}
and subject to no-slip boundary conditions for the velocity field alongside perfectly insulating (adiabatic) boundary conditions for the temperature:
\begin{align}\label{boundary_cond}
    u(t,\cdot)\vert_{\partial\Omega} = 0, \quad \nabla \theta(t,\cdot) \cdot n(\cdot)\vert_{\partial\Omega} = 0,
\end{align}
where $n$ denotes the outward unit normal vector to $\partial\Omega$. The boundary condition \eqref{boundary_cond} implies that at least for sufficiently smooth solutions the total energy is preserved in time i.e. for a.e. $t \in [0,T)$ it holds
\begin{equation}
    \int_\Omega \lr{\frac 12 \rho u^2 + \rho e(\rho,\theta)}dx (t) = \int_\Omega \lr{\frac 12 \rho_0 u_0^2 + \rho_0 e(\rho_0,\theta_0)}dx.
\end{equation}
Since our interest is solely in the local regularity properties of the temperature for the system \eqref{nsf_mass}--\eqref{nsf_energy}, we work either on a smooth domain or on the flat torus \(\Omega=\mathbb{T}^d\).

In order to avoid issues with physical interpretation of the constitutive relations we study just explicit forms of pressure, internal energy and so on. To keep the mathematical needs we specify the growth conditions of coefficients defining our thermodynamical quantities.

\begin{itemize}[leftmargin=*]

\item \textbf{Pressure and Internal Energy.}
We consider the constitutive laws 
\begin{align}\label{pr:idR1}
    p(\rho,\theta)
    &= R\rho\theta + C\rho^\gamma + \frac{a}{3}\theta^4 \text{ and }
    \rho e(\rho,\theta)
    = \frac{R}{\gamma-1}\rho\theta
    + \frac{C}{\gamma-1}\rho^\gamma
    + a\theta^4,
\end{align}
where $\gamma>1$ and $R,C,a>0$. Throughout the paper, for simplicity of notation, we assume 
\begin{align}\label{pr:norm}
    R=C=a=1 \text{ and }c_v=\frac{1}{\gamma-1}.
\end{align}

\item \textbf{Viscosity Coefficients.}
The viscous stress tensor \(\mathbb{S}\) is given by Newton's rheological law:
\begin{align}\label{NR}
    \mathbb{S}(\theta,\nabla u)
    =
    \mu(\theta)
    \left(
    \frac{\nabla u+\nabla^T u}{2}
    -\frac{1}{d}(\operatorname{div}u)\mathbb{I}
    \right)
    +
    \eta(\theta)(\operatorname{div}u)\mathbb{I},
\end{align}
where \(d=3\) denotes the spatial dimension and \(\mathbb{I}\) is the identity matrix. The temperature-dependent shear viscosity \(\mu(\theta)\) and bulk viscosity \(\eta(\theta)\) satisfy
\begin{align}
    &(1+\theta^\alpha)\underline{\mu}
    \leq
    \mu(\theta)
    \leq
    (1+\theta^\alpha)\overline{\mu},
    \qquad
    \sup_{\theta\ge0}|\mu'(\theta)|
    \leq C_\mu,
    \label{c:mu}\\
    &0
    \leq
    \underline{\eta}+\underline{\eta}_1\theta
    \leq
    \eta(\theta)
    \leq
    \overline{\eta}(1+\theta),
    \label{c:eta}
\end{align}
for some \(\alpha\in(\frac25,1]\) and non-negative constants
\(\underline{\mu},\overline{\mu},
\underline{\eta},\underline{\eta}_1,
\overline{\eta},C_\mu\).

\item \textbf{Heat Conductivity.}
The heat flux \(q\) is governed by Fourier's law:
\begin{align}\label{FL}
    q=-\kappa(\theta)\nabla\theta,
\end{align}
where the total heat conductivity is decomposed as
\begin{align}\label{c:kappa}
    \kappa(\theta)
    =
    \kappa_M(\theta)
    +
    \kappa_R(\theta).
\end{align}
We assume that
\begin{align}\label{c:kappa:2}
    \underline{\kappa}
    \leq
    \kappa_M(\theta)
    \leq
    \overline{\kappa}(1+\theta^\alpha),
    \qquad
    \kappa_R(\theta)=\theta^m,
\end{align}
for some structural constants \(\underline{\kappa},\overline{\kappa}>0\), the same exponent \(\alpha\in(\frac25,1]\), and \(m\ge3\).

\end{itemize}


The main goal of the paper is to establish a conditional regularity result for the temperature. We want to analyze the weak solutions to the whole system  (\ref{nsf_mass})-(\ref{nsf_energy}) aiming at showing the H\"older regularity of the temperature under conditional information that the density and velocity of the fluid are in sufficiently high regularity class. 
To this end, we consider the system \eqref{nsf_mass}--\eqref{nsf_energy}, supplemented with the constitutive relations \eqref{pr:idR1}, \eqref{NR}, \eqref{FL}, and the assumptions \eqref{c:mu}, \eqref{c:eta}, \eqref{c:kappa:2}. This issue can be regarded as an analogue of the classical Lady\v{z}enskaja–Prodi–Serrin regularity criteria, which establish sufficient conditions under which weak solutions attain regularity \cite{Kozo-Sohr,FSV}.

Recall that the continuity equation admits the renormalized form
\[
\partial_t b(\rho)+\operatorname{div}(b(\rho)u)
+\bigl(\rho b'(\rho)-b(\rho)\bigr)\operatorname{div}u=0,
\]
for every admissible renormalization
$b\in C^1((0,\infty))\cap C([0,\infty)).
$
Using the renormalized continuity equation with $\displaystyle b(\rho)= \frac{1}{\gamma-1}\rho^\gamma$, the internal energy equation \eqref{nsf_energy} can be rewritten in the form
\begin{equation}\label{eq:theta}
\begin{aligned}
\partial_t(c_v\rho\theta+\theta^4)
&+\operatorname{div}\!\bigl((c_v\rho\theta+a\theta^4)u\bigr)
+(\rho\theta+\tfrac{1}{3}\theta^4)\operatorname{div}u
\\
&=
\mu(\theta)\left|D u-\frac13(\operatorname{div}u)\mathbb I\right|^2
+\eta(\theta)|\operatorname{div}u|^2
+\operatorname{div}\bigl(\kappa(\theta)\nabla\theta\bigr),
\end{aligned}
\end{equation}
where we recall $c_v= \frac{1}{\gamma-1}$ and $ D u=\frac{\nabla^Tu+\nabla u}{2}$.\par 
The main idea is to establish a De Giorgi type regularity theory for \eqref{eq:theta}. Unlike the classical uniformly parabolic setting, the equation contains transport terms, nonlinear diffusion, and lower-order source terms depending on the velocity field. Therefore, the first step is to obtain a local $L^\infty$ bound for the temperature through a De Giorgi iteration scheme. Once boundedness is available, we derive an oscillation decay estimate and subsequently establish the local Hölder continuity of $\theta$.

Our analysis starts from an admissible weak solution of \eqref{eq:theta}. We emphasize that this notion of weak solution requires a certain amount of additional regularity. In particular, the existence of such solutions within the framework of the currently available weak solution theory remains an open problem due to the regularity requirements. Therefore, throughout this section, we work with the following notion of admissible weak solution. We assume the initial data \eqref{initial_cond} satisfy
\begin{align}\label{cond:initial}
\quad \rho_0\ge0,\ \theta_0\ge0\ \text{a.e.},\qquad
\rho_0\in L^\gamma(\Omega),\quad \theta_0\in L^5(\Omega),\qquad
\rho_0|u_0|^2\in L^1(\Omega),
\end{align}
with finite initial energy $\displaystyle E_0:=\int_\Omega\big(\tfrac12\rho_0|u_0|^2
+c_v\rho_0^\gamma+c_v\rho_0\theta_0+\theta_0^4\big)\,dx<\infty$.

\begin{definition}[Definition of admissible weak solution]\label{def:weak1}
Let $\mu(\theta),\,\eta(\theta),\,\kappa(\theta)$ satisfy
\eqref{c:mu},\,\eqref{c:eta},\,\; \;\eqref{c:kappa}-\eqref{c:kappa:2} with $m=3$, respectively. We say that a triple
of measurable functions $(\rho,\,u,\,\theta)$ is an \emph{admissible weak
solution} of \eqref{nsf_mass}, \eqref{nsf_mom}, \eqref{eq:theta} with initial data $(\rho_0, u_0,\theta_0)$ satisfying \eqref{cond:initial} in
$Q:=(0,T)\times\Omega$ if the following hold:
\begin{align}
    &\textup{(i)   } \rho \ge 0, \quad \theta \ge 0 \quad \text{a.e. on } Q, \label{cond:positivity} \\
    &\textup{(ii)  }\theta \in L^\infty(0,T;L^5(\Omega)) \cap L^2(0,T; H^1(\Omega)), \quad \theta^{5/2} \in L^2(0,T;\dot H^1(\Omega)),  \label{cond:theta_reg} \\
    &\textup{(iii) }\rho\in L^q(Q),\;
u\in L^\infty\!\bigl(0,T;W^{1,q}(\Omega)\bigr),\; q>d,\quad
|\nabla u|^2\in L^p(Q),\quad
p>\frac{d+2}{2}, \label{cond:spaces}
\end{align}

and \eqref{nsf_mass}, \eqref{nsf_mom}, \eqref{eq:theta} are satisfied in the
following sense, for a.e. $\tau\in[0,T]$:

\begin{itemize}[leftmargin=*]
\item \textbf{(Continuity equation)} for every
$\phi\in C^1([0,T]\times\overline{\Omega})$,
\begin{equation}
\int_{\Omega} \rho(\tau,\cdot)\phi(\tau,\cdot)\,dx
-\int_{\Omega}\rho_0\,\phi(0,\cdot)\,dx
=\int_0^\tau\!\!\int_{\Omega}\big(\rho\,\partial_t\phi+\rho u\cdot\nabla\phi\big)\,dx\,dt;
\end{equation}

\item \textbf{(Momentum equation)} for every
$\varphi\in C^1([0,T]\times\overline{\Omega};\mathbb R^3)$ with
$\varphi|_{\partial\Omega}=0$,
\begin{align}
&\int_\Omega \rho u(\tau,\cdot)\cdot\varphi(\tau,\cdot)\,dx
-\int_\Omega \rho_0u_0\cdot\varphi(0,\cdot)\,dx
\nonumber\\
&\qquad=\int_0^\tau\!\!\int_\Omega
\Big(\rho u\cdot\partial_t\varphi
+\rho u\otimes u:\nabla\varphi
+p(\rho,\theta)\operatorname{div}\varphi
-\mathbb S:\nabla\varphi\Big)\,dx\,dt;
\end{align}

\item \textbf{(Internal energy equation)} for every
$\phi\in C^1([0,T]\times\overline{\Omega})$ with $\phi|_{\partial\Omega}=0$,
\begin{align}
&\int_{\Omega}(c_v\rho\theta+\theta^4)(\tau,\cdot)\,\phi(\tau,\cdot)\,dx
-\int_{\Omega}(c_v\rho_0\theta_0+\theta_0^4)\,\phi(0,\cdot)\,dx
\nonumber\\
&\quad+\int_0^\tau\!\!\int_{\Omega}\Big((\rho\theta+\tfrac{1}{3}\theta^4)\operatorname{div}u\,\phi
+\kappa(\theta)\nabla\theta\cdot\nabla\phi\Big)\,dx\,dt
\nonumber\\
&\quad=\int_0^\tau\!\!\int_{\Omega}\bigg((c_v\rho\theta+\theta^4)(\partial_t\phi+u\cdot\nabla\phi)
+\mu(\theta)\Big|Du-\tfrac13(\operatorname{div}u)\mathbb I\Big|^2\phi
+\eta(\theta)|\operatorname{div}u|^2\phi\bigg)\,dx\,dt.
\end{align}
\end{itemize}

\end{definition}
Our next goal is to state the main result concerning the local H\"older continuity of the temperature under additional regularity assumptions on the density and velocity. We emphasize the following point: the result relies on a De~Giorgi iteration driven by a \emph{renormalized} temperature equation, rather than solely on Definition~\ref{def:weak1}. More precisely, it uses a family of inequalities obtained by testing \eqref{eq:theta} against nonlinear functions \(b'(\theta)\) \textup{(}for suitable smooth functions \(b\)\textup{)}, which are satisfied in a weak sense. We postpone the details to the beginning of Section~\ref{sec:UB}; the precise notion is introduced in Definition~\ref{def:weakrenorm1} \textup{(}Subsection~\ref{ssec:renorm}\textup{)}. Our main result is the following.
\begin{theorem}[Local H\"older regularity of the temperature]
\label{thm:main}
Let $(\rho,u,\theta)$ be an admissible renormalized weak solution following Definition \ref{def:weakrenorm1} of \eqref{nsf_mass}-\eqref{eq:theta} in $Q=(0,T)\times\Omega $ with $\Omega \subset \mathbb{R}^3$.
Assume further that for some $p>\frac52$,
\begin{align}
|u|^2
\in L^p_{\mathrm{loc}}((0,T]\times \Omega),\,\,\rho\in L_{\mathrm{loc}}^\infty((0,T]\times \Omega),
\label{hyp:main}
\end{align}
and for some $\eta>0$,
\begin{align}\label{hyp-theta-lb}
    \theta>\eta\,\,\,\text{in}\,\,\,Q.
\end{align}
Then there exists
$\alpha\in(0,1),$
depending only on the structural constants appearing in \eqref{c:mu}, \eqref{c:eta}, \eqref{c:kappa:2} and the quantities appearing in \eqref{hyp:main}-\eqref{hyp-theta-lb}, such that
\[
\theta\in C^{\frac{\alpha}{2},\alpha}_{\mathrm{loc}}((0,T]\times \Omega).
\]
More precisely, for every compact set
$\mathcal K\Subset (0,T]\times \Omega,$
there exists a constant $C_{\mathcal K}>0$ such that
\begin{align*}
|\theta(t,x)-\theta(s,y)|
\le
C_\mathcal K\Bigl(
|x-y|^\alpha
+
|t-s|^{\alpha/2}
\Bigr)
\end{align*}
for all $(t,x),(s,y)\in \mathcal K$.

\end{theorem}

\begin{remark}
The proof of Theorem~\ref{thm:main} rests on a De~Giorgi iteration for the temperature
equation. Its two ingredients are a local $L^\infty$ bound on $\theta$ and a
scale-invariant oscillation-reduction step; iterating the latter over dyadic scales
converts boundedness into the H\"older modulus. For linear or semilinear parabolic
equations the $L^\infty$ bound follows directly from the maximum principle or a
Stampacchia truncation, but here the coupling with the density and velocity makes it
more delicate, and we establish it separately in Section~\ref{sec:UB} by an $L^\infty$
De~Giorgi energy iteration. That iteration is driven by the renormalized temperature
balance \eqref{eq:renorm}, which supplies the required family of truncated energy
inequalities, and it needs only weaker hypotheses on $\rho$ and $u$ than
Theorem~\ref{thm:main} (the precise statement being Theorem~\ref{thm:max}). Throughout,
we follow the nonhomogeneous De Giorgi framework developed by Vasseur~\cite[Section~3.2]{Vasseur16}, while adopting the quantitative parabolic approach of Guerand~\cite[Section~3]{JG20}, whose lemmas we recall, for completeness, in the
appendices.
\end{remark}
\begin{remark}
Although Definition~\ref{def:weak1} is stated globally on $Q$, the conclusion of
Theorem~\ref{thm:main} is local. Accordingly, only a local version of the solution
concept is needed: it suffices that the weak formulation and the renormalized
inequality \eqref{eq:renorm} hold for test functions compactly supported in $Q$, and
that $(\rho,u,\theta)$ enjoy the regularity \eqref{cond:theta_reg}--\eqref{cond:spaces}
and \eqref{hyp:main} only locally, i.e.\ on every compact subset $K\Subset Q$. Under
these local hypotheses the proof goes through unchanged, since every estimate is
carried out on parabolic cylinders contained in $Q$.
\end{remark}
\begin{remark}
The hypothesis on the lower bound of the temperature \eqref{hyp-theta-lb} in Theorem~\ref{thm:main} is a somewhat strong assumption. The lower bound for the temperature has been established in the works of
Mellet and Vasseur~\cite{MV09} and Baer and Vasseur~\cite{BV13}, though without the
radiation pressure component. Another known result in this direction, which includes it,
is due to Feireisl~\cite{Feireisl25}, who considers a general pressure law together with
the hypothesis of the third law of thermodynamics. We plan to prove the corresponding
result for the pressure \eqref{pr:idR1} in upcoming work.
\end{remark}
We conclude  by discussing the significance of Theorem \ref{thm:main}. Broadly speaking, the result can be viewed from two different perspectives.

The first concerns the direct construction of regular solutions to the compressible Navier--Stokes--Fourier system. The existence theory for global regular solutions corresponding to small perturbations of equilibrium states is by now well established; we refer, for example, to the classical works of Matsumura and Nishida \cite{MANI}, Valli \cite{Vall1,Vall2}, Valli and Zaj\k{a}czkowski \cite{VAZA}, Danchin \cite{Danchin2001,DanchCPDE}, Cho and Kim \cite{ChoKim1}, the third author and Pokorn\'y \cite{PP} and the references therein. In this context, Theorem \ref{thm:main} provides a new regularity mechanism based on the thermal component of the system.

The second perspective, particularly relevant to the present paper, arises in the construction of weak solutions. The existence theory for weak solutions of the compressible Navier--Stokes system was initiated by P.-L.~Lions~\cite{Lions1998} and subsequently developed in depth by Feireisl~\cite{Feireisl2004}, Feireisl--Novotn\'y~\cite{FN2001}, Kreml et al. \cite{KMNW1, KMNW2}, Pokorn\'y and 
Sk\v ri\v sovsk\'y \cite{PS21}. The problem with inflow-outflow boundary conditions was addressed by Feiresl and Novotn\'y in \cite{FN21} and the result was further extended to cover the case of Dirichlet boundary conditions for the temperature by the first author and Feireisl in \cite{FC22}. The construction of such weak solutions relies on a multilevel approximation scheme (see Feireisl \cite{F2004} and Feiresl and Novotn\'y \cite{FN09}), in which the regularity properties of the approximate solutions constitute a crucial intermediate step. It is here that Theorem~\ref{thm:main} simplifies the analysis substantially: the
information already available at the level of weak solutions is enough, through the parabolic structure of the temperature equation, to yield H\"older continuity of $\theta$. Once this initial gain of regularity is in hand, standard parabolic estimates bootstrap the solution to full regularity. In this way the theorem exploits the parabolic character of the temperature equation in a particularly efficient manner, bypassing the more delicate compactness arguments otherwise required at this stage.

At first sight, the statement of Theorem \ref{thm:main} appears to be of a local-in-time nature, since each oscillation-reduction step passes from a larger parabolic cylinder to a strictly smaller one. Consequently, the iteration does not reach the initial time. However, the behavior of solutions near the initial time can be treated by means of the local-in-time regularity theory, which is substantially simpler than the analysis leading to Theorem \ref{thm:main}. Consequently, the result may be viewed as providing the essential regularity mechanism away from the initial time singularities.


From the mathematical point of view, the result is, in a sense, expected, as it is based on the De Giorgi iteration technique, which plays a fundamental role in the regularity theory of parabolic equations. However, several obstacles prevent a direct application of the classical theory (see, e.g., DiBenedetto \cite{Di1993} and related works on De Giorgi methods \cite{DeG1957}), in particular due to the special structure of the compressible heat-conducting system and the presence of additional degenerated nonlinear terms.
From this perspective, the proof of Theorem \ref{thm:main} should be viewed as an illustration of the strength of classical regularity methods in the context of modern analysis of compressible heat-conducting fluids. The argument highlights how the De Giorgi-type machinery can still be effectively implemented once the correct structural properties of the system are identified.
To keep the paper self-contained, we also include several proofs of auxiliary lemmas that constitute the basic elements of the De Giorgi regularity theory adapted to our setting.

De Giorgi methods have previously been applied to compressible heat-conducting fluids, primarily to establish positive lower bounds for the temperature. Mellet and Vasseur~\cite{MV09} treated the case in which the pressure is affine in the temperature, while Baer and Vasseur~\cite{BV13} extended this approach to more general constitutive relations by means of a localized entropy inequality; related arguments were also employed by Wang and Zuo~\cite{WZ22}. These works are based on the entropy formulation of the thermal energy balance and rely on logarithmic truncations of the temperature. Consequently, the singular weights of order $1/\theta$ arising from the entropy production lead to substantial technical difficulties in the low-temperature regime, particularly in the presence of vacuum.

In contrast, in the solution framework adopted in the present work, the internal energy equation is formulated through a renormalized inequality \eqref{eq:renorm}, allowing us to work directly at the renormalized level. This enables us to employ the upper and lower quadratic truncations $\frac12(\theta-C)_+^2$ and $\frac12(C-\theta)_+^2$, respectively (see \eqref{eq:trunc} and \eqref{eq:trunc-minus}), and to obtain not only local boundedness of the temperature, but also the oscillation-reduction mechanism leading to its local Hölder continuity.
\medskip

The second part of the paper is devoted to an application of Theorem~\ref{thm:main}. We consider a simplified version of the full Navier--Stokes--Fourier system by replacing the momentum equation with its potential-flow counterpart. To this end, we first assume that the viscosity coefficients in \eqref{NR} are positive constants independent of the temperature and that the velocity is irrotational, namely
\[
u=\nabla\phi.
\]
Then for simplicity we put $\operatorname{div}\mathbb S = \Delta u=\nabla \div u$ as $u $ is potential. And the momentum equation reduces to
$\Delta u=\nabla p.$
This way we obtain a version of the main system for the compressible potential viscous flow for heat-conducting fluids. The system is given by
\begin{align}
\label{nsf_toy}
\left\{
\begin{array}{rl}
\partial_t\rho+\ebdiv (\rho u)&=0,\\[2mm]
\Delta u &=\nabla p,
\\[2mm]
\partial_t(c_v\rho\theta+\theta^4)
+\operatorname{div}\!\left((c_v\rho\theta+\theta^4)u\right)
+
(\rho\theta+\tfrac13\theta^4)\operatorname{div}u
&=
|\operatorname{div}u|^2
+\operatorname{div}(\kappa(\theta)\nabla\theta),
\end{array}
\right.
\end{align}
with pressure law 
\begin{equation} \label{pres:toy}
p(\rho,\theta)
=
\rho^\gamma
+
\rho\theta
+
\frac13\theta^4.
\end{equation}
For this problem, we work on the torus $\mathbb{T}^d$ for arbitrary dimension $d\geq 1$.

\begin{theorem}[Global unique regular solution]
\label{thm:global_toy}
Let $\underline{\theta}_0>0$ and
\begin{align}\label{hyp:toy}
    0\leq \rho_0\in W^1_p(\TT^d),
\qquad
0<\underline{\theta}_0\leq \theta_0\in W^{2-\frac{2}{p}}_p(\TT^d),
\end{align}
for some \(p>d+2\).
Then the system \eqref{nsf_toy} admits a unique global solution
$(\rho,\theta,u)$ on \([0,\infty)\times\TT^d\)
such that for any $T>0$
\begin{equation}
  \rho \in L^\infty(0,T;W^1_p(\TT^d))\cap
  W^1_p(0,T;L^p(\TT^d)), \quad \theta\in W^{2,1}_p(\TT^d\times(0,T)), \quad
  u \in L^\infty(0,T;W^2_p(\TT^d)).
\end{equation}
Furthermore,
\[
0\leq \rho(t,x) \quad\text{and}\quad
0<\underline{\theta}_0 e^{-Ct}\le\theta(t,x)\le\theta_{\mathrm{MAX}}
\qquad\text{for all }(t,x)\in [0,T]\times\TT^d,
\]
where \(C\) and \(\theta_{\mathrm{MAX}}\) depend only on the initial data and the
parameters of the system.
\end{theorem}

The system \eqref{nsf_toy} belongs to a broad class of modifications of system \eqref{nsf_mass},\eqref{nsf_mom} and \eqref{eq:theta}, which arise naturally both from mathematical and physical considerations. From the analytical point of view, such systems provide a testing ground for modern methods in the theory of nonlinear partial differential equations, while from the physical perspective they can be regarded as admissible models in certain regimes where additional structural effects are taken into account.

In the present paper, we provide an application of Theorem \ref{thm:main}. The main consequence is the possibility of extending a local-in-time solution to a global one under suitable structural assumptions. The argument proceeds as follows.

\smallskip 
\begin{itemize}[leftmargin=*]
    \item[-] First, we establish the local-in-time existence of a strong solution, assuming that the initial data belong to an appropriate regularity class, in the spirit of Solonnikov~\cite{SoloI} and Lady\v{z}enskaja, Solonnikov, and Ural'tseva~\cite{LSU}. In contrast to many classical small-data results, the construction does not rely on smallness assumptions on the data; the only essential requirement is the positivity of the initial temperature, while strict positivity of the initial density is not imposed. \vspace{2mm}
    \item[-] Next, we derive a maximum principle for the density and, subsequently, for the temperature. As a byproduct, we obtain uniform-in-time bounds on the divergence of the velocity field. These estimates are independent of time and play a crucial role in the continuation argument. \vspace{2mm}
    \item[-] Combining these bounds, we obtain a positive lower bound for the temperature which, owing to the structure of the system, may decay at most exponentially in time. This information is sufficient to apply Theorem~\ref{thm:main}, which yields H\"older regularity of the temperature. \vspace{2mm}
    \item[-] A careful analysis of the continuity equation in Slobodeckij spaces then implies H\"older continuity of the density as well. In this step, we also make use of logarithmic estimates of $L^\infty$--BMO type, which allow one to control the nonlinear coupling terms. \vspace{2mm}
    \item[-] Once these regularity properties are established, we show that the temperature in fact belongs to a parabolic regularity class of the form $W^{2,1}_p$ for sufficiently large $p$. This higher regularity follows from standard parabolic estimates applied to the heat equation with controlled coefficients. \vspace{2mm}
    \item[-] Finally, the above bounds provide a continuation criterion for strong solutions: the solution can be extended beyond any finite time interval as long as the temperature remains bounded from below and the density and velocity satisfy the a priori estimates derived above. Since these conditions are preserved globally in time, no blow-up can occur.
\end{itemize}
As a consequence, we obtain the existence of a global-in-time unique regular solution to system \eqref{nsf_toy}.

\smallskip

\textbf{The paper is organized as follows:} Section~\ref{sec:UB} is devoted to the derivation of a local upper bound for the temperature. In Section~\ref{sec:holder}, we establish the local H\"older continuity of the temperature and prove Theorem \ref{thm:main}. Section~\ref{sec:appl} contains an application of the theory to a heat-conducting potential flow model \eqref{nsf_toy}. Finally, the appendices (Appendices~\ref{sec:itlem}, \ref{sec:degiorgi-proof}, and \ref{proof:ineqTr}) collect several technical results and auxiliary proofs used throughout the paper, largely based on the expositions of Vasseur~\cite{Vasseur16} and Guerand~\cite{JG20}.


\section{Upper bound of the temperature}\label{sec:UB}
The goal of this section is to bound the temperature from above. To this end, we derive a renormalized internal energy equation, which provides the energy inequalities on level sets needed to run a De Giorgi iteration.
\subsection{The renormalized internal energy equation}\label{ssec:renorm}

In this section we consider the case $p(\rho, \theta) = \rho \theta + \rho^{\gamma} + \frac{1}{3}\theta^4,$ then the equation \eqref{nsf_energy} becomes
\begin{equation}\label{eq:theta1}
	\begin{aligned}
\partial_t(c_v\rho\theta+\theta^4)&+\text{div}\,\left((c_v\rho\theta+\theta^4)u\right)+ (\rho\theta+\tfrac{1}3\theta^4)\text{div}\,u \\
		&=
		\mu(\theta)\left|D u-\frac13(\operatorname{div}u)\mathbb I\right|^2
		+\eta(\theta)|\operatorname{div}u|^2+\text{div} (\kappa(\theta)\nabla \theta),
	\end{aligned}
\end{equation}
where $\mu(\theta),\,\eta(\theta),\,\kappa(\theta)$ satisfy \eqref{c:mu},\,\eqref{c:eta},\,\eqref{c:kappa}-\eqref{c:kappa:2} with $m=3$. 

For a suitable $b: [0,\infty)\rightarrow \mathbb{R}$ such that $b\in C^2_b$, we consider the following renormalized version of the thermal equation. Using the continuity equation, it reads as:

\begin{align}\label{renor-thermal-1}
	\begin{split}
		&	\partial_t \left(c_v \rho b(\theta) + \eta_b(\theta) \right)+ \text{div} \left(\left(c_v\rho b(\theta) + \eta_b(\theta) \right) u\right) - \Delta K_b(\theta) \\
		&=- b^{\prime \prime}(\theta) \kappa(\theta) |\nabla \theta|^2+ b^\prime(\theta)\left(\mu(\theta)\left|D u-\frac13(\operatorname{div}u)\mathbb I\right|^2+\eta(\theta)|\operatorname{div}u|^2\right) \\
		&\quad - \left( b^\prime(\theta) \rho\theta + \left( \frac{4}{3} \theta^4 b^\prime(\theta)-\eta_b(\theta) \right)\right) \text{div} u,
	\end{split}
\end{align}
where 
\begin{align*}
	\eta_b^\prime(\theta)=4 b^\prime(\theta) \theta^3\text{ and }K^{\prime}_b(\theta)= b^\prime(\theta) \kappa(\theta) .
\end{align*}

\begin{definition}\label{def:weakrenorm1}
Suppose all the hypotheses of Definition~\ref{def:weak1} hold. We say that
$(\rho,u,\theta)$ is an \emph{admissible renormalized weak solution} if it is a
weak solution in the sense of Definition~\ref{def:weak1} and, in addition,
satisfies the following inequality for every nonnegative $\psi\in C^1_c(\Omega)$
and a.e.\ $0\le\tau<s\le T$:
\begin{align}\label{eq:renorm}
\Big[\int_\Omega E_b\,\psi\,dx\Big]_{\tau}^{s}
&+\int_\tau^s\!\!\int_\Omega b''(\theta)\,\kappa(\theta)\,|\nabla\theta|^2\,\psi
\;\le\;
\int_\tau^s\!\!\int_\Omega E_b\,u\cdot\nabla\psi
\;-\;\int_\tau^s\!\!\int_\Omega b'(\theta)\,\kappa(\theta)\nabla\theta\cdot\nabla\psi
\nonumber\\
&+\int_\tau^s\!\!\int_\Omega b'(\theta)\Big(\mu(\theta)\big|Du-\tfrac13(\operatorname{div}u)\mathbb I\big|^2
+\eta(\theta)|\operatorname{div}u|^2\Big)\psi
\nonumber\\
&-\int_\tau^s\!\!\int_\Omega \Big(b'(\theta)\rho\theta
+\big(\tfrac43\theta^4 b'(\theta)-\eta_b(\theta)\big)\Big)\operatorname{div}u\,\psi\;dx\,dt,
\end{align}
for every renormalizer
\[
b\in C^2([0,\infty)),\qquad b'\ \text{and}\ b''\ \text{bounded},\qquad
b\ \text{constant on}\ [M,\infty)\ \text{for some }M<\infty,
\]
where $\eta_b$ is defined by $\eta_b'(\theta)=4b'(\theta)\theta^3$,
$\eta_b(0)=0$, and $E_b=c_v \rho\,b(\theta)+\eta_b(\theta)$.
\end{definition}
\begin{remark}
It is not clear a priori whether a weak solution in the sense of
Definition~\ref{def:weak1} satisfies the renormalized inequality
\eqref{eq:renorm}. Since the justification of such a statement appears to require additional arguments and regularity, we take \eqref{eq:renorm} as part of our assumptions and focus on its consequences.

A rigorous derivation of the renormalized inequality from the weak formulation under minimal additional hypotheses is an interesting problem in its own right. We hope to address this question in future work, in a framework reminiscent of the DiPerna--Lions renormalization theory for the continuity equation.
\end{remark}

\begin{coro}[Truncated renormalized energy inequality]\label{lem:renormTr}
Let $(\rho,u,\theta)$ be an admissible renormalized weak solution on
$Q=(0,T)\times\Omega$. Fix a level $C>0$ and write $\theta_{C}:=(\theta-C)_+$, so
that $\nabla\theta_{C}=\nabla\theta\,\mathbf 1_{\{\theta>C\}}$. With
\[
\eta_b(\theta)=\tfrac45\theta_{C}^5+3C\theta_{C}^4+4C^2\theta_{C}^3+2C^3\theta_{C}^2,
\qquad
E_b=\tfrac12 c_v  \rho\,\theta_{C}^2+\eta_b(\theta)\ge0,
\]
for every nonnegative $\psi\in C^1_c(\Omega)$ and a.e.\ $0\le\tau<s\le T$,
\begin{align}\label{eq:trunc}
\Big[\int_\Omega E_b\,\psi\Big]_{\tau}^{s}
&+\int_\tau^s\!\!\int_{\{\theta>C\}}\kappa(\theta)\,|\nabla\theta_{C}|^2\,\psi
\;\le\;
\int_\tau^s\!\!\int_\Omega E_b\,u\cdot\nabla\psi
\nonumber\\
&-\int_\tau^s\!\!\int_\Omega \theta_{C}\,\kappa(\theta)\nabla\theta\cdot\nabla\psi
+\int_\tau^s\!\!\int_\Omega \theta_{C}\Big(\mu(\theta)\big|Du-\tfrac13(\operatorname{div}u)\mathbb I\big|^2
+\eta(\theta)|\operatorname{div}u|^2\Big)\psi
\nonumber\\
&-\int_\tau^s\!\!\int_\Omega \Big(\,\theta_{C}\,\rho\theta
+\big(\tfrac43\theta^4\theta_{C}-\eta_b(\theta)\big)\Big)\operatorname{div}u\,\psi .
\end{align}
Moreover, let $b_C^-(\theta):=\frac12(C-\theta)_+^2, \theta_C^-:=(C-\theta)_+,$ so that $\nabla\theta_C^- = -\nabla\theta\,\mathbf{1}_{\{\theta<C\}}$. After adding the harmless constant $\frac15C^5$ to the corresponding primitive, we define
\begin{align*}
\eta_b^-(\theta)
:=
\mathbf{1}_{\{\theta<C\}}
\int_\theta^C4(C-\zeta)\zeta^3\,d\zeta=
2C^3(\theta_C^-)^2
-4C^2(\theta_C^-)^3
+3C(\theta_C^-)^4
-\frac45(\theta_C^-)^5,
\end{align*}
and
\begin{align*}
E_b^-:=\frac12 c_v  \rho(\theta_C^-)^2+\eta_b^-(\theta)\geq0.
\end{align*}
Then, for every nonnegative $\psi\in C_c^1(\Omega)$ and for a.e.
$0\leq\tau<s\leq T$,
\begin{align}
\left[\int_\Omega E_b^-\psi\right]_\tau^s
&+\int_\tau^s\int_{\{\theta<C\}}
\kappa(\theta)|\nabla\theta_C^-|^2\psi
\leq
\int_\tau^s\int_\Omega E_b^-u\cdot\nabla\psi\nonumber\\
&\quad
-\int_\tau^s\int_\Omega
\theta_C^-\kappa(\theta)\nabla\theta_C^-\cdot\nabla\psi
-\int_\tau^s\int_\Omega
\theta_C^-
\left(
\mu(\theta)
\left|Du-\frac13(\operatorname{div}u)\mathbb I\right|^2
+\eta(\theta)|\operatorname{div}u|^2
\right)\psi
\nonumber\\
&\qquad
+\int_\tau^s\int_\Omega
\left(
\theta_C^-\rho\theta
+\frac43\theta^4\theta_C^-
+\eta_b^-(\theta)
\right)
\operatorname{div}u\,\psi.
\label{eq:trunc-minus}
\end{align}
\end{coro}
The proof is given in the Appendix \ref{proof:ineqTr}.

\subsection{Statement and proof of the upper-bound theorem}
Using the truncated renormalized energy inequality \eqref{eq:trunc} together with the De Giorgi iteration, we obtain the following main result.
\begin{theorem} \label{thm:max}
 Let $(x_0,t_0)\in\Omega\times(0,T],$ and set  $\delta:=dist(x_0,\partial \Omega)$, for any $0<R<\min\{\delta, \sqrt{t_0}\}$, if $(\rho,\,u,\,\theta)$ is an admissible renormalized weak solution to \eqref{nsf_mass}-\eqref{eq:theta} (in the sense of Definition \ref{def:weakrenorm1}) in $(t_0-R^2,t_0)\times B(x_0,R)$, assume further that $\rho|\operatorname{div}u|
\in L^p((t_0-R^2,t_0)\times B(x_0,R))$ for some $p>\frac52$, then there exists $M= M(R)$ such that $\theta(t,x)\leq M$ a.e. in $(t_0-\frac{R^2}{4},t_0)\times B(x_0,\frac{R}{2})$.
\end{theorem}

\begin{proof}[Proof of the Theorem \ref{thm:max}]
Let us define 
\begin{equation} \label{def:rQB}
r_k=\frac{R}{2}(1+2^{-k}), \;\; B_{k}=B(x_0,r_k), \;\; Q_k=(t_0-r_k^2,t_0)\times B_k,
\end{equation}
and
\begin{equation} \label{def:Ck1}
\Theta_k:=\int_{Q_k}\theta_k^{\frac{25}{3}} dxds, \;\; \theta_k=(\theta- C_k)_+, \;\;  C_k=  M(1-2^{-k-1}), 
\end{equation}
where $M$ is to be determined later. Then for all $k\geq 0,\,\frac{ M}{2}\leq  C_k \leq  M$. In order to prove the desired result, it is enough to show that $\lim_{k \to \infty} \Theta_k =0.$

We also consider $\{\phi_k\}_{k=1}^\infty$ be a family of smooth nonnegative cutoff functions, compactly supported in $B_{k-1}$ and equal to $1$ in $B_k$ such that $|\nabla\phi_k|\leq \frac{C}{R}2^k$.

\smallskip 

In \eqref{eq:trunc}, let $t_0-r_k^2<\tau<t_0-r_{k+1}^2<s<t_0, C=C_{k+1}, \psi=\phi_{k+1}^2$. We get
\begin{align}
		\int_{B_k} E_b\,\phi_{k+1}^2(s)
		&+\int_\tau^s\!\!\int_{B_k}\kappa(\theta)\,|\nabla\theta_{k+1}|^2\,\phi_{k+1}^2
		\;\le\;
		\int_{B_k} E_b\,\phi_{k+1}^2(\tau)+\int_\tau^s\!\!\int_{B_k} E_b\,u\cdot\nabla(\phi_{k+1}^2)
		\nonumber\\
		&-\int_\tau^s\!\!\int_{B_k} \theta_{k+1}\,\kappa(\theta)\nabla\theta\cdot\nabla(\phi_{k+1}^2)
		+\int_\tau^s\!\!\int_{B_k} \theta_{k+1}\Big(\mu(\theta)\big|Du-\tfrac13(\operatorname{div}u)\mathbb I\big|^2
		+\eta(\theta)|\operatorname{div}u|^2\Big)\phi_{k+1}^2
		\nonumber\\
		&-\int_\tau^s\!\!\int_{B_k} \Big(\,\theta_{k+1}\,\rho\theta
		+\big(\tfrac43\theta^4\theta_{k+1}-\eta_b(\theta)\big)\Big)\operatorname{div}u\,\phi_{k+1}^2,\label{ene100}
	\end{align}
where
	\begin{align}\label{eq:eta-E}
    \eta_b(\theta)=\tfrac45\theta_{k+1}^5+3C_{k+1}\theta_{k+1}^4+4C_{k+1}^2\theta_{k+1}^3+2C_{k+1}^3\theta_{k+1}^2,
	\qquad E_b=\tfrac12 c_v \rho\,\theta_{k+1}^2+\eta_b(\theta)\ge0.
	\end{align}
	Note that
    \begin{align*}
        \left|\int_\tau^s\!\!\int_{B_k} \theta_{k+1}\,\kappa(\theta)\nabla\theta\cdot\nabla(\phi_{k+1}^2)\right|=&2\left|\int_\tau^s\!\!\int_{B_k} \phi_{k+1}\,\theta_{k+1}\,\kappa(\theta)\nabla\theta_{k+1}\cdot\nabla\phi_{k+1}\right|\\
        &\leq \frac12\int_\tau^s\!\!\int_{B_k} \,\kappa(\theta)|\nabla\theta_{k+1}|^2\phi_{k+1}^2 + C \int_\tau^s\int_{B_k} \kappa(\theta)\theta_{k+1}^2|\nabla\phi_{k+1}|^2.
    \end{align*}
    The first term can be absorbed by the LHS of \eqref{ene100}, which gives 
    \begin{align*}
        \int_{B_k} E_b\,\phi_{k+1}^2(s)
		&+\frac12\int_\tau^s\!\!\int_{B_k}\kappa(\theta)\,|\nabla\theta_{k+1}|^2\,\phi_{k+1}^2
		\;\le\;
		\int_{B_k} E_b\,\phi_{k+1}^2(\tau)+\int_\tau^s\!\!\int_{B_k} E_b\,u\cdot\nabla(\phi_{k+1}^2)
		\nonumber\\
		&+ C \int_\tau^s\int_{B_k} \kappa(\theta)\theta_{k+1}^2|\nabla\phi_{k+1}|^2
		+\int_\tau^s\!\!\int_{B_k} \theta_{k+1}\Big(\mu(\theta)\big|Du-\tfrac13(\operatorname{div}u)\mathbb I\big|^2
		+\eta(\theta)|\operatorname{div}u|^2\Big)\phi_{k+1}^2
		\nonumber\\
		&-\int_\tau^s\!\!\int_{B_k} \Big(\,\theta_{k+1}\,\rho\theta
		+\big(\tfrac43\theta^4\theta_{k+1}-\eta_b(\theta)\big)\Big)\operatorname{div}u\,\phi_{k+1}^2.
    \end{align*}
    Since 
    \begin{align*}
        \frac12\int_\tau^s\!\!\int_{B_k}\kappa(\theta)\,|\nabla\theta_{k+1}|^2\,\phi_{k+1}^2\geq &\frac12\int_\tau^s\!\!\int_{B_k}(\underline{\kappa}+\theta^3)\,|\nabla\theta_{k+1}|^2\,\phi_{k+1}^2\\
        &\geq \frac12\int_\tau^s\!\!\int_{B_k}(\underline{\kappa}+\theta_{k+1}^3)\,|\nabla\theta_{k+1}|^2\,\phi_{k+1}^2\\
        &=\frac12\int_\tau^s\!\!\int_{B_k}\left(\underline{\kappa}|\nabla\theta_{k+1}|^2+\tfrac{4}{25}|\nabla\theta_{k+1}^{\frac52}|^2\right)\,\phi_{k+1}^2,
    \end{align*}
    we get
    \begin{align*}
         &\int_{B_k} E_b\,\phi_{k+1}^2(s) +\frac{\underline{\kappa}}2\int_\tau^s\int_{B_{k}}|\nabla\theta_{k+1}|^2\phi_{k+1}^2 dx d s'+\tfrac{2}{25}\int_\tau^s\int_{B_{k}}|\nabla\theta_{k+1}^{\frac52}|^2\phi_{k+1}^2 dx d s'\\
		&\;\le\;
		\int_{B_k} E_b\,\phi_{k+1}^2(\tau)+\int_\tau^s\!\!\int_{B_k} E_b\,u\cdot\nabla(\phi_{k+1}^2)
		+ C \int_\tau^s\int_{B_k} \kappa(\theta)\theta_{k+1}^2|\nabla\phi_{k+1}|^2\\
		&
		+\int_\tau^s\!\!\int_{B_k} \theta_{k+1}\Big(\mu(\theta)\big|Du-\tfrac13(\operatorname{div}u)\mathbb I\big|^2
		+\eta(\theta)|\operatorname{div}u|^2\Big)\phi_{k+1}^2
		\nonumber\\
		&-\int_\tau^s\!\!\int_{B_k} \Big(\,\theta_{k+1}\,\rho\theta
		+\big(\tfrac43\theta^4\theta_{k+1}-\eta_b(\theta)\big)\Big)\operatorname{div}u\,\phi_{k+1}^2.
    \end{align*}
    By using the range of $s$ and $\tau$,  $t_0-r_k^2<\tau<t_0-r_{k+1}^2<s<t_0$, we obtain
\begin{align*}&\int_{B_k} E_b(s)\,\phi_{k+1}^2+\frac{\underline{\kappa}}2\int_{t_0-r_{k+1}^2}^s\int_{B_{k}}|\nabla\theta_{k+1}|^2\phi_{k+1}^2 dx d s'+\tfrac{2}{25}\int_{t_0-r_{k+1}^2}^s\int_{B_{k}}|\nabla\theta_{k+1}^{\frac52}|^2\phi_{k+1}^2 dx d s'\\
   &\leq \int_{B_{k}}E_b(\tau) \phi_{k+1}^2 dx+\int_{t_0-r_{k}^2}^{t_0}\!\!\int_{B_k} E_b\,\left|u\cdot\nabla(\phi_{k+1}^2)\right|
		+ C \int_{t_0-r_{k}^2}^{t_0}\int_{B_k} \kappa(\theta)\theta_{k+1}^2|\nabla\phi_{k+1}|^2\\
		&
		+\int_{t_0-r_{k}^2}^{t_0}\!\!\int_{B_k} \theta_{k+1}\Big(\mu(\theta)\big|Du-\tfrac13(\operatorname{div}u)\mathbb I\big|^2
		+\eta(\theta)|\operatorname{div}u|^2\Big)\phi_{k+1}^2
		\nonumber\\
		&+\int_{t_0-r_{k}^2}^{t_0}\!\!\int_{B_k} \left|\Big(\,\theta_{k+1}\,\rho\theta
		+\big(\tfrac43\theta^4\theta_{k+1}-\eta_b(\theta)\big)\Big)\operatorname{div}u\right|\,\phi_{k+1}^2.
\end{align*}

Taking the mean value in $\tau$ between $t_0-r_k^2$ and $t_0-r_{k+1}^2$ we arrive at
\begin{align*}
&\int_{B_k} E_b(s)\,\phi_{k+1}^2+\frac{\underline{\kappa}}2\int_{t_0-r_{k+1}^2}^s\int_{B_{k}}|\nabla\theta_{k+1}|^2\phi_{k+1}^2 dx d s'+\tfrac{2}{25}\int_{t_0-r_{k+1}^2}^s\int_{B_{k}}|\nabla\theta_{k+1}^{\frac52}|^2\phi_{k+1}^2 dx d s'\\
   &\leq \frac{1}{r_k^2-r_{k+1}^2} \int_{t_0-r_k^2}^{t_0-r_{k+1}^2}\int_{B_{k}}E_b \phi_{k+1}^2 dx+\int_{t_0-r_{k}^2}^{t_0}\!\!\int_{B_k} E_b\,\left|u\cdot\nabla(\phi_{k+1}^2)\right|
		+ C \int_{t_0-r_{k}^2}^{t_0}\int_{B_k} \kappa(\theta)\theta_{k+1}^2|\nabla\phi_{k+1}|^2\\
		&
		+\int_{t_0-r_{k}^2}^{t_0}\!\!\int_{B_k} \theta_{k+1}\Big(\mu(\theta)\big|Du-\tfrac13(\operatorname{div}u)\mathbb I\big|^2
		+\eta(\theta)|\operatorname{div}u|^2\Big)\phi_{k+1}^2
		\nonumber\\
		&+\int_{t_0-r_{k}^2}^{t_0}\!\!\int_{B_k} \left|\Big(\,\theta_{k+1}\,\rho\theta
		+\big(\tfrac43\theta^4\theta_{k+1}-\eta_b(\theta)\big)\Big)\operatorname{div}u\right|\,\phi_{k+1}^2.
\end{align*}

Taking the supremum in $s$, one can get
\begin{align}\label{es:Eb}
&\sup\limits_{s\in[t_0-r_{k+1}^2,t_0]}\int_{B_k} E_b(s)\,\phi_{k+1}^2+\frac{\underline{\kappa}}2\int_{t_0-r_{k+1}^2}^{t_0}\int_{B_{k}}|\nabla\theta_{k+1}|^2\phi_{k+1}^2 dx d s'+\tfrac{2}{25}\int_{t_0-r_{k+1}^2}^{t_0}\int_{B_{k}}|\nabla\theta_{k+1}^{\frac52}|^2\phi_{k+1}^2 dx d s'\nonumber\\
   &\leq \frac{1}{r_k^2-r_{k+1}^2} \int_{t_0-r_k^2}^{t_0-r_{k+1}^2}\int_{B_{k}}E_b \phi_{k+1}^2 dx+\int_{t_0-r_{k}^2}^{t_0}\!\!\int_{B_k} E_b\,\left|u\cdot\nabla(\phi_{k+1}^2)\right|
		+ C \int_{t_0-r_{k}^2}^{t_0}\int_{B_k} \kappa(\theta)\theta_{k+1}^2|\nabla\phi_{k+1}|^2\nonumber\\
		&
		+\int_{t_0-r_{k}^2}^{t_0}\!\!\int_{B_k} \theta_{k+1}\Big(\mu(\theta)\big|Du-\tfrac13(\operatorname{div}u)\mathbb I\big|^2
		+\eta(\theta)|\operatorname{div}u|^2\Big)\phi_{k+1}^2
		\nonumber\\
		&+\int_{t_0-r_{k}^2}^{t_0}\!\!\int_{B_k} \left|\Big(\,\theta_{k+1}\,\rho\theta
		+\big(\tfrac43\theta^4\theta_{k+1}-\eta_b(\theta)\big)\Big)\operatorname{div}u\right|\,\phi_{k+1}^2.
\end{align}

By using the interpolation inequlity $\|f\|_{L^{\frac{25}{3}}}\leq \|f\|_{L^5}^{\frac25}\|f\|_{L^{15}}^{\frac35}$ we have 
\begin{align*}
  \|f\|_{L_T^{\frac{25}{3}}L^{\frac{25}{3}}}&\leq C \|f\|_{L_T^\infty L^5}^{\frac25}\|f\|_{L_T^5 L^{15}}^{\frac35}\\
  &\leq C \|f\|_{L_T^\infty L^5}^{\frac25}\|f^{\frac52}\|_{L_T^2 L^6}^{\frac{6}{25}}
  \leq C \|f\|_{L_T^\infty L^5}^{\frac25}\|\nabla(f^{\frac52})\|_{L_T^2 L^2}^{\frac{6}{25}}.
\end{align*}
Therefore one can get
\begin{align} \label{es:Theta}
  \Theta_{k+1}=\int_{Q_{k+1}}\theta_{k+1}^{\frac{25}{3}} dxds&\leq C \|\theta_{k+1}\|_{L_T^\infty L^5(Q_{k+1})}^{\frac{10}{3}}\|\nabla(\theta_{k+1}^{\frac52})\|_{L_T^2 L^2(Q_{k+1})}^{2}.
\end{align}
Recalling the definition \eqref{def:rQB} we have $r_k^2-r_{k+1}^2\geq \frac{R^2}{2^{k+2}}$. Therefore, by \eqref{eq:eta-E}, \eqref{es:Eb} and \eqref{es:Theta} and Young's inequality, one can get
\begin{align*}
  \Theta_{k+1}^{\frac35} &\leq \|\theta_{k+1}\|_{L_T^\infty L^5(Q_{k+1})}^2\|\nabla(\theta_{k+1}^{\frac52})\|_{L_T^2 L^2(Q_{k+1})}^{\frac65}\\
   &\leq \frac{C 2^{k+2}}{R^2} \int_{t_0-r_k^2}^{t_0-r_{k+1}^2}\int_{B_{k}}E_b \phi_{k+1}^2 + C\int_{t_0-r_{k}^2}^{t_0}\!\!\int_{B_k} E_b\,\left|u\cdot\nabla(\phi_{k+1}^2)\right|
		+ C \int_{t_0-r_{k}^2}^{t_0}\int_{B_k} \kappa(\theta)\theta_{k+1}^2|\nabla\phi_{k+1}|^2\nonumber\\
		&
		+C\int_{t_0-r_{k}^2}^{t_0}\!\!\int_{B_k} \theta_{k+1}\Big(\mu(\theta)\big|Du-\tfrac13(\operatorname{div}u)\mathbb I\big|^2
		+\eta(\theta)|\operatorname{div}u|^2\Big)\phi_{k+1}^2
		\nonumber\\
		&+C\int_{t_0-r_{k}^2}^{t_0}\!\!\int_{B_k} \left|\Big(\,\theta_{k+1}\,\rho\theta
		+\big(\tfrac43\theta^4\theta_{k+1}-\eta_b(\theta)\big)\Big)\operatorname{div}u\right|\,\phi_{k+1}^2=: \sum_{i=1}^5 J_i.
\end{align*}
Noticing that on the set $\{\theta>C_{k+1}\},$ by \eqref{def:Ck1} we have
\begin{align*}
    \theta- C_k \geq  C_{k+1}- C_k =  M 2^{-k-2}\geq  C_{k+1}2^{-k-2}\geq  C_{k}2^{-k-2},
\end{align*}
which implies 
\begin{align}\label{es:1-Ck-up}
 1\leq \frac{2^{k+2}\,\theta_k }{M},\,\,C_{k}\leq   C_{k+1}\leq 2^{k+2}\,\theta_k, 
\end{align}
and
\begin{align}\label{es:theta-power}
    \theta^\beta=\left((\theta-C_k)+C_k\right)^\beta\leq C (\theta_k^\beta +C_k^\beta)\leq C^k\, \theta_k^\beta,
\end{align}
for all $\beta\in(0,\infty)$.
Therefore \eqref{eq:eta-E}, \eqref{es:1-Ck-up} yields
\begin{align}\label{es:J1}
  J_1\leq & \frac{C^k}{R^2}\int_{Q_k}\tfrac12\rho\,\theta_{k+1}^2+\tfrac45\theta_{k+1}^5+3C_{k+1}\theta_{k+1}^4+4C_{k+1}^2\theta_{k+1}^3+2C_{k+1}^3\theta_{k+1}^2\nonumber\\
  & =\frac{C^k}{R^2}\int_{Q_k\cap \{\theta>C_{k+1}\}}\tfrac12\rho\,\theta_{k+1}^2+\tfrac45\theta_{k+1}^5+3C_{k+1}\theta_{k+1}^4+4C_{k+1}^2\theta_{k+1}^3+2C_{k+1}^3\theta_{k+1}^2\nonumber\\
  &\leq \frac{C^k}{R^2}\int_{Q_k\cap \{\theta>C_{k+1}\}}(\rho\,\theta_{k}^2+\theta_{k}^5),
\end{align}
where we use the fact that $\theta_{k+1}\leq \theta_k$.

For the first term, applying \eqref{es:1-Ck-up} we get
\begin{align*}
    &\frac{C^k}{R^2}\int_{Q_k\cap \{\theta>C_{k+1}\}}\rho\theta_{k}^2 dx ds'\\
    &\leq \frac{C^k}{R^2}\left(\int_{Q_k\cap \{\theta> C_{k+1}\}} dx ds'\right)^{\frac{19}{25}-\frac1{q_1}}\|\rho\|_{L^{q_1}}\left(\int_{Q_k}\theta_k^{\frac{25}{3}} dx d s'\right)^{\frac{6}{25}}\\
     &\leq \frac{C^k}{R^2 M^{\frac{19}{3}-\frac{25}{3q_1}}}\|\rho\|_{L^{q_1}}\left(\int_{Q_k}\theta_k^{\frac{25}{3}} dx d s'\right)^{1-\frac1{q_1}}=\frac{C^k}{R^2 M^{\frac{19}{3}-\frac{25}{3q_1}}}\|\rho\|_{L^{q_1}}\Theta_k^{1-\frac1{q_1}}, 
\end{align*} 
similarly for the second term, we have
\begin{align*}
    &\frac{C^k}{R^2}\int_{Q_k\cap \{\theta>C_{k+1}\}}\theta_{k}^5 d x d s' \\
     &\leq \frac{C^k}{R^2}\left(\int_{Q_k\cap \{\theta>C_{k+1}\}}\theta_{k}^{\frac{25}{3}} d x d s'\right)^{\frac{3}{5}}\left(\int_{Q_k\cap \{\theta> C_{k+1}\}} d x d s'\right)^{\frac25}\\
   &\leq \frac{C^k }{R^2 M^{\frac{10}{3}}}\int_{Q_k}\theta_{k}^{\frac{25}{3}} d x d s'= \frac{ C^k }{R^2 M^{\frac{10}{3}}}\Theta_k,
\end{align*}
which implies
\begin{align}\label{es:J1-final}
  J_1 \leq \frac{C^k}{R^2 M^{\frac{19}{3}-\frac{25}{3q_1}}}\|\rho\|_{L^{q_1}}\Theta_k^{1-\frac1{q_1}} + \frac{ C^k }{R^2 M^{\frac{10}{3}}}\Theta_k.
\end{align}
Similarly as \eqref{es:J1}, one can get
\begin{align*}
 J_2\leq \frac{C^k}{R}\int_{Q_k\cap \{\theta>C_{k+1}\}}(\rho u \,\theta_{k}^2+ u \theta_{k}^5),
\end{align*}
where we use the fact that $|\nabla\phi_{k+1}|\leq \frac{C}{R}2^{k+1}$.
For the first term, similarly as above, one can get
\begin{align*}
  &\frac{C^k}{R}\int_{Q_k\cap \{\theta>C_{k+1}\}}\rho u \theta_{k}^2 dx ds'\\
    &\leq \frac{C^k}{R}\left(\int_{Q_k\cap \{\theta> C_{k+1}\}} dx ds'\right)^{\frac{19}{25}-\frac1{q_2}}\|\rho u\|_{L^{q_2}}\left(\int_{Q_k}\theta_k^{\frac{25}{3}} dx d s'\right)^{\frac{6}{25}}\\
     &\leq \frac{C^k}{R M^{\frac{19}{3}-\frac{25}{3q_2}}}\|\rho u\|_{L^{q_2}}\left(\int_{Q_k}\theta_k^{\frac{25}{3}} dx d s'\right)^{1-\frac1{q_2}}=\frac{C^k}{R M^{\frac{19}{3}-\frac{25}{3q_2}}}\|\rho u\|_{L^{q_2}}\Theta_k^{1-\frac1{q_2}}, 
\end{align*}
for the second term, 
\begin{align*}
  &\frac{C^k}{R}\int_{Q_k\cap \{\theta>C_{k+1}\}}u \theta_{k}^5 d x d s' \\
     &\leq \frac{C^k}{R}\|u\|_{L^{q_3}}\left(\int_{Q_k\cap \{\theta>C_{k+1}\}}\theta_{k}^{\frac{25}{3}} d x d s'\right)^{\frac{3}{5}}\left(\int_{Q_k\cap \{\theta> C_{k+1}\}} d x d s'\right)^{\frac25-\frac1{q_3}}\\
   &\leq \frac{C^k }{R M^{\frac{10}{3}-\frac{25}{3q_3}}}\|u\|_{L^{q_3}}\left(\int_{Q_k}\theta_{k}^{\frac{25}{3}} d x d s'\right)^{1-\frac1{q_3}}= \frac{ C^k }{R M^{\frac{10}{3}-\frac{25}{3q_3}}}\|u\|_{L^{q_3}}\Theta_k^{1-\frac1{q_3}},
\end{align*}
which implies
\begin{align}\label{es:J2-final}
  J_2 \leq \frac{C^k}{R M^{\frac{19}{3}-\frac{25}{3q_2}}}\|\rho u\|_{L^{q_2}}\Theta_k^{1-\frac1{q_2}} + \frac{ C^k }{R M^{\frac{10}{3}-\frac{25}{3q_3}}}\|u\|_{L^{q_3}}\Theta_k^{1-\frac1{q_3}}.
\end{align}
For the estimation of $J_3$, by using \eqref{es:1-Ck-up}, \eqref{es:theta-power} and the fact that $|\nabla\phi_{k+1}|\leq \frac{C}{R}2^{k+1}$, similarly as above, one can get
\begin{align}\label{es:J3-final}
  J_3 \leq & C \int_{Q_k\cap\{\theta>C_{k+1}\}} \kappa(\theta)\theta_{k+1}^2|\nabla\phi_{k+1}|^2\nonumber\\
  &\leq C \int_{Q_k\cap\{\theta>C_{k+1}\}} (1+\theta^3)\theta_{k+1}^2|\nabla\phi_{k+1}|^2\nonumber\\
  &\leq \frac{C^k}{R^2}\int_{Q_k\cap\{\theta>C_{k+1}\}} \theta_{k}^2+\theta_k^5\leq  \frac{C^k}{R^2 M^{\frac{19}{3}}}\Theta_k + \frac{ C^k }{R^2 M^{\frac{10}{3}}}\Theta_k.
\end{align}
For $J_4$, in light of \eqref{c:mu}, \eqref{c:eta}, similarly as above, we have
\begin{align*}
  J_4\leq& C\int_{Q_k\cap\{\theta>C_{k+1}\}}\theta_{k+1}\Big(\mu(\theta)\big|Du-\tfrac13(\operatorname{div}u)\mathbb I\big|^2
		+\eta(\theta)|\operatorname{div}u|^2\Big)\phi_{k+1}^2\\
&\leq C\int_{Q_k\cap\{\theta>C_{k+1}\}}\theta_{k+1}(1+\theta)(|D u|^2+|\operatorname{div}u|^2)\\
&\leq C\int_{Q_k\cap\{\theta>C_{k+1}\}}(\theta_k+\theta_k^2)(|D u|^2+|\operatorname{div}u|^2)\\
&\leq \left(\frac{C^k}{M^{\frac{19}{3}-\frac{25}{3q_4}}} + \frac{C^k}{M^{\frac{22}{3}-\frac{25}{3q_4}}}\right)\left(\||D u|^2\|_{L^{q_4}}+\||\operatorname{div}u|^2\|_{L^{q_4}}\right)\Theta_k^{1-\frac1{q_4}}.
\end{align*}
For $J_5$, 
\begin{align*}
  J_5\leq & C\int_{Q_k\cap\{\theta>C_{k+1}\}}\left|\Big(\theta_{k+1}\,\rho\theta
		+\big(\tfrac43\theta^4\theta_{k+1}-\eta_b(\theta)\big)\Big)\operatorname{div}u\right|\,\phi_{k+1}^2\\
& \leq C^k \int_{Q_k\cap\{\theta>C_{k+1}\}}(\rho\theta_k^2 + \theta_k^5)|\operatorname{div}u|\\
&\leq \frac{C^k}{M^{\frac{19}{3}-\frac{25}{3q_5}}}\|\rho \operatorname{div}u\|_{L^{q_5}}\Theta_k^{1-\frac1{q_5}} + \frac{C^k}{M^{\frac{10}{3}-\frac{25}{3q_6}}}\|\operatorname{div}u\|_{L^{q_6}}\Theta_k^{1-\frac1{q_6}} .
\end{align*}

Summing up the above estimates we get 
\begin{align*}
    \Theta_{k+1}^{\frac35}\leq &\frac{C^k}{R^2 M^{\frac{19}{3}-\frac{25}{3q_1}}}\|\rho\|_{L^{q_1}}\Theta_k^{1-\frac1{q_1}} + \frac{ C^k }{R^2 M^{\frac{10}{3}}}\Theta_k \\
&+ \frac{C^k}{R M^{\frac{19}{3}-\frac{25}{3q_2}}}\|\rho u\|_{L^{q_2}}\Theta_k^{1-\frac1{q_2}} + \frac{ C^k }{R M^{\frac{10}{3}-\frac{25}{3q_3}}}\|u\|_{L^{q_3}}\Theta_k^{1-\frac1{q_3}}\\
& + \frac{C^k}{R^2 M^{\frac{19}{3}}}\Theta_k + \frac{ C^k }{R^2 M^{\frac{10}{3}}}\Theta_k \\
&+ \left(\frac{C^k}{M^{\frac{19}{3}-\frac{25}{3q_4}}} + \frac{C^k}{M^{\frac{22}{3}-\frac{25}{3q_4}}}\right)\left(\||D u|^2\|_{L^{q_4}}+\||\operatorname{div}u|^2\|_{L^{q_4}}\right)\Theta_k^{1-\frac1{q_4}} \\
&+ \frac{C^k}{M^{\frac{19}{3}-\frac{25}{3q_5}}}\|\rho \operatorname{div}u\|_{L^{q_5}}\Theta_k^{1-\frac1{q_5}} + \frac{C^k}{M^{\frac{10}{3}-\frac{25}{3q_6}}}\|\operatorname{div}u\|_{L^{q_6}}\Theta_k^{1-\frac1{q_6}}\\
&\leq C^k \max\left\{\frac{C^k}{R^2 M^{\frac{19}{3}-\frac{25}{3q_1}}}\|\rho\|_{L^{q_1}}, \frac{ C^k }{R^2 M^{\frac{10}{3}}}, \frac{C^k}{R M^{\frac{19}{3}-\frac{25}{3q_2}}}\|\rho u\|_{L^{q_2}},\right.\\
&\left. \frac{ C^k }{R M^{\frac{10}{3}-\frac{25}{3q_3}}}\|u\|_{L^{q_3}}, \frac{C^k}{R^2 M^{\frac{19}{3}}}, \frac{C^k}{M^{\frac{19}{3}-\frac{25}{3q_4}}}\||D u|^2, |\operatorname{div}u|^2\|_{L^{q_4}}, \right.\\
&\left.\frac{C^k}{M^{\frac{22}{3}-\frac{25}{3q_4}}}\||D u|^2, |\operatorname{div}u|^2\|_{L^{q_4}}, \frac{C^k}{M^{\frac{19}{3}-\frac{25}{3q_5}}}\|\rho \operatorname{div}u\|_{L^{q_5}}, \frac{C^k}{M^{\frac{10}{3}-\frac{25}{3q_6}}}\|\operatorname{div}u\|_{L^{q_6}}\right\}(\Theta_k+\Theta_k^{1-\frac{1}q}),
\end{align*}
where $q=\min\limits_{1\leq i \leq 6}q_i$.
Thanks to the assumption on $\rho,\,u$ and $\nabla u$, one can get $q>\frac52$ and 
\begin{align*}
   \Theta_{k+1}^\frac{3}{5} \leq \frac{C^k}{C(M)}(\Theta_k+\Theta_k^{1-\frac{1}q}),
\end{align*}
where $C(M) \to \infty$ as $M \to \infty$.

Since $\theta\in L^\infty((t_0-R^2,t_0),L^5( B(x_0,R))) \cap L^2((t_0-R^2,t_0),\dot{H}^1( B(x_0,R))),$ and $\theta^\frac52\in L^2((t_0-R^2,t_0),\dot{H}^1( B(x_0,R)))$, one can get $\theta\in L^{\frac{25}{3}}((t_0-R^2,t_0)\times B(x_0,R))$, which implies 
$\Theta_k\leq C$, then we have
\begin{align}\label{mx:thetak}
   \Theta_{k+1}\leq \frac{C^k}{C(M)} \Theta_k^\beta \text{ \ \ for \ } \beta>1.
\end{align}

Since $\theta\in L^{\frac{25}{3}}((t_0-R^2,t_0)\times B(x_0,R))$, one can get that there exists $ M$ large enough such that $\|(\theta-\frac{ M}{2})_+\|_{L^\frac{25}{3}(Q_0)}$ is small enough, i.e. $\Theta_0$ is small enough, then Lemma~\ref{lem:itera1} implies $\lim\limits_{k\to+\infty}\Theta_k=0,$ and $\theta\leq  M$ in $(t_0-\frac{R}{4},t_0)\times B(x_0,\frac{R}{2})$.
\end{proof}

By applying this result locally on compact subsets of $Q$ and using a
finite covering argument, we obtain the following local boundedness
result, which will be used in the oscillation-reduction argument.
The finite covering is needed because Theorem~\ref{thm:max} provides
an almost-everywhere bound on each individual cylinder: although the
bound can be chosen uniformly with respect to the center, the
corresponding exceptional null set may depend on the center, and an
uncountable union of null sets may not be null.

\begin{coro}\label{cor:local-upper-bound}
Let $(\rho,u,\theta)$ be an admissible renormalized weak solution to
\eqref{nsf_mass}--\eqref{eq:theta} \textup{(}in the sense of
Definition~\ref{def:weakrenorm1}\textup{)}  in
$Q=(0,T)\times\Omega$, and assume that
\begin{align*}
\rho |\operatorname{div}u|
\in L^p_{\mathrm{loc}}((0,T]\times \Omega)
\qquad\text{for some }p>\frac52.
\end{align*}
Then
$\theta\in L^\infty_{\mathrm{loc}}((0,T]\times \Omega).$
More precisely, for every compact set $\mathcal K\Subset (0,T]\times \Omega$, there
exists a constant $M_{\mathcal K}>0$ such that
\begin{align*}
\theta
\leq M_{\mathcal K}\,\,\text{a.e.\,\,in}\,\,\mathcal K.
\end{align*}
\end{coro}

\begin{proof}[Proof of Corollary \ref{cor:local-upper-bound}]
Let $\mathcal K\subset (0,T]\times\Omega$ be compact and separated from
the initial and spatial boundaries. For every
$z=(t,x)\in\mathcal K$, we construct a relative neighborhood $U_z$
as follows.

If $t<T$, choose $r_z>0$ sufficiently small and define
$\widehat z=(t+r_z^2/8,x)$, so that
$Q_{r_z}(\widehat z)\Subset Q$. Since
\begin{align*}
Q_{r_z/2}(\widehat z)
=
\left(t-\frac{r_z^2}{8},t+\frac{r_z^2}{8}\right)
\times B\left(x,\frac{r_z}{2}\right),
\end{align*}
we have $z\in Q_{r_z/2}(\widehat z)$. We therefore set
$U_z:=Q_{r_z/2}(\widehat z)$. By Theorem~\ref{thm:max}, there exists
a constant $M_z>0$ such that
$\sup\limits_{U_z}\theta\leq M_z$.

If $t=T$, choose $r_z>0$ sufficiently small and define
$U_z:=\left(T-r_z^2/4,T\right]\times B(x,r_z/2)$. Then $U_z$ is open
in the relative topology of $(0,T]\times\Omega$, contains $z$, and
$U_z\cap Q=Q_{r_z/2}(T,x)$. Applying Theorem~\ref{thm:max} with center
$(T,x)$ and radius $r_z$, we obtain a constant $M_z>0$ such that
$\sup\limits_{U_z\cap Q}\theta\leq M_z$.

The family $\{U_z\}_{z\in\mathcal K}$ is an open cover of
$\mathcal K$ in the relative topology of $(0,T]\times\Omega$.
By compactness, there exist $z_1,\ldots,z_N\in\mathcal K$ such that
$\mathcal K\subset\bigcup_{j=1}^N U_{z_j}$. Setting
$M_{\mathcal K}:=\max\limits_{1\leq j\leq N}M_{z_j}$, we conclude that $\sup\limits_{\mathcal K\cap Q}\theta
\leq M_{\mathcal K}.$
\end{proof}

Next, we state a global counterpart of Theorem~\ref{thm:max} on the torus $\TT^d$.
\begin{coro}\label{coro:global-ub}
Let $(\rho,u,\theta)$ be an admissible renormalized weak solution to
\eqref{nsf_mass}--\eqref{eq:theta} \textup{(}in the sense of
Definition~\ref{def:weakrenorm1}\textup{)} on $(0,T)\times\TT^d$, assume that $\rho|\operatorname{div}u|
\in L^p((0,T)\times \TT^d)$ for some $p>\frac52$ and 
$\theta_0\in L^\infty(\TT^d),$
then $\theta\in L^\infty([0,T]\times\TT^d)$.
\end{coro}
We omit the proof; the modification of the iteration near the initial surface is
outlined in the following remark.
\begin{remark}
The argument behind Theorem~\ref{thm:max} is purely local in space--time. On the torus
$\Omega=\TT^d$, with $\theta_0\in L^\infty(\Omega)$, the De~Giorgi iteration can in
addition be carried out in parabolic cylinders touching the initial surface $t=0$,
namely
\[
Q_R^+=(0,R^2)\times B_R(x_0),
\]
by integrating the renormalized energy inequality from $t=0$ rather than averaging
over an intermediate time level. The resulting estimate carries the initial contribution
\[
\int_{B_R}\big(\rho_0+1\big)(\theta_0-C_k)_+^{2}\,dx,
\]
where, since $\theta_0$ is bounded, the radiative part $\eta_b(\theta_0)\sim\theta_0^5$
of the renormalized energy has been absorbed into the displayed $(\theta_0-C_k)_+^2$
term up to the factor $1$. This vanishes once the truncation levels $C_k$ exceed
$\|\theta_0\|_{L^\infty}$. Since on $\TT^d$ there is no spatial boundary
and the quantities entering the iteration are uniformly bounded, the threshold $M$ can
be chosen independently of the center $x_0$. This uniformity upgrades the local
$L^\infty$ bound to the global estimate
\[
\theta\in L^\infty((0,T)\times\Omega).
\]
\end{remark}

\section{H\"older regularity of the temperature}\label{sec:holder}

The proof of Theorem~\ref{thm:main} follows the De~Giorgi strategy: starting from the
local $L^\infty$ bound of Section~\ref{sec:UB}, a single oscillation-reduction step
shrinks the oscillation of $\theta$ by a fixed factor at each dyadic scale, and
iterating this contraction yields the algebraic decay $\osc_{Q_r}\theta\lesssim r^\alpha$,
i.e.\ the local H\"older continuity.
We follow the nonhomogeneous De Giorgi framework developed by Vasseur~\cite[Section~3.2]{Vasseur16}, while adopting the quantitative parabolic approach of Guerand~\cite[Section~3]{JG20}. 

Our target in this section is to prove that for every $\mathcal K\Subset (0,T]\times\Omega$, there exists $\alpha=\alpha_\mathcal K$ such that $\theta\in C^{\frac{\alpha}{2},\alpha}(\mathcal K)$. For $z_0=(t_0,x_0)\in\mathcal K$, let us define
\begin{equation*}
Q_r(z_0):=(t_0-r^2,t_0)\times B(x_0,r),
\end{equation*}
then we can choose $R_\mathcal K>0$ such that
\begin{align*}
Q_{4R_\mathcal K}(z_0)\Subset (0,T]\times\Omega
\qquad\text{for every }z_0\in\mathcal K.
\end{align*}

We first record the following consequence of Theorem~\ref{thm:max} and Corollary \ref{cor:local-upper-bound}.

\begin{theorem}\label{cor:uni-upb}
Under the above notation, let $(\rho,u,\theta)$ be an admissible renormalized weak solution to
\eqref{nsf_mass}--\eqref{eq:theta} \textup{(}in the sense of
Definition~\ref{def:weakrenorm1}\textup{)} on $Q=(0,T)\times\Omega$, and assume that
\begin{align*}
\rho |\operatorname{div}u|
\in L^p_{\mathrm{loc}}((0,T]\times\Omega))
\qquad\text{for some }p>\frac52.
\end{align*}
Then there exists a constant $N_{\mathcal K}>0$, depending only on $\mathcal K$ such that
\begin{align*}
\theta
\leq N_{\mathcal K}\,\,\text{a.e.  in}\,\,\bigcup_{z_0\in\mathcal K}Q_{2R_{\mathcal K}}(z_0).
\end{align*}
\end{theorem}

\begin{proof}[Proof of Theorem \ref{cor:uni-upb}]
Define $\widetilde{\mathcal K}:=\overline{\bigcup_{z_0\in\mathcal K}Q_{2R_{\mathcal K}}(z_0)}.$
By the compactness of $\mathcal K$ and the choice of
$R_{\mathcal K}$, namely,
\begin{align*}
Q_{4R_{\mathcal K}}(z_0)\Subset (0,T]\times\Omega
\qquad
\text{for every }z_0\in\mathcal K,
\end{align*}
we have $\widetilde{\mathcal K}\Subset (0,T]\times\Omega.$ Applying Corollary~\ref{cor:local-upper-bound} to
$\widetilde{\mathcal K}$, we obtain a constant $M_{\widetilde{\mathcal K}}>0$ such that $\theta
\leq M_{\widetilde{\mathcal K}}$  a.e. in $\widetilde{\mathcal K}$.

Finally, since $R_{\mathcal K}$ is fixed in terms of the distance of
$\mathcal K$ from the parabolic boundary of $Q$, the enlarged compact
set $\widetilde{\mathcal K}$ depends only on $\mathcal K$. We therefore set $N_{\mathcal K}:=M_{\widetilde{\mathcal K}}.$
\end{proof}
Throughout this section, all suprema, infima, and oscillations are
understood in the essential sense. To simplify the notation, we omit
the prefix ``$\operatorname*{ess}$''. 
So that the oscillation of $\theta$ in $Q_{2r}(z_0)$ defined as
\begin{equation}
\osc_{Q_{2r}(z_0)}\theta:=\sup\limits_{Q_{2r}(z_0)}\theta-\inf\limits_{Q_{2r}(z_0)}\theta
\end{equation}
is well defined for $0<r\leq R_\mathcal K$

We rewrite the equation as
\begin{align}\label{eq:ene2}
&\partial_t(c_v\rho\theta+\theta^4)+\text{div}\,\left((c_v\rho\theta+\theta^4)u\right)-\text{div} (\kappa(\theta)\nabla \theta) \\ \nonumber
    &=\mu(\theta)\left|D u-\frac13(\operatorname{div}u)\mathbb I\right|^2
+\eta(\theta)|\operatorname{div}u|^2-(R\rho\theta+\tfrac13\theta^4)\text{div}\,u=:F, 
\end{align}
and for $R>0$ define
\begin{align} \label{def:CFR}
\mathcal F_R
=
\| |Du|^2\|_{L^p(Q_R(z_0))}
+\||\operatorname{div}u|^2\|_{L^p(Q_R(z_0))}
+\|\rho|\operatorname{div}u|\|_{L^p(Q_R(z_0))}
+\||\operatorname{div}u|\|_{L^p(Q_R(z_0))}.
\end{align}

For all $r\leq R_\mathcal K$, under the parabolic scaling
\begin{align*}
t=t_0+r^2s,\qquad x=x_0+ry,
\end{align*}
with
\begin{align}\label{theta_r}
&\theta_r:=\theta(t_0+r^2s,x_0+ry),\,\rho_r:=\rho(t_0+r^2s,x_0+ry),\\
&u_r:=u(t_0+r^2s,x_0+ry),\,F_r:=F(t_0+r^2s,x_0+ry), \nonumber
\end{align}
formally one can rewrite \eqref{eq:ene2} as
\begin{equation}\label{eq:theta-scal}
\partial_s(c_v\rho_r\theta_r+\theta_r^4)+r\text{div}_y\,\left((c_v\rho_r\theta_r+\theta_r^4)u_r\right)-\text{div}_y (\kappa(\theta_r)\nabla_y \theta_r)= r^2F_r.
\end{equation}
For simplicity, we write $ Q_r:=Q_r(0,0).$
\begin{lemma}
 If $(\rho,u,\theta)$ is an admissible renormalized weak solution to \eqref{nsf_mass}-\eqref{eq:theta} with \eqref{eq:theta} equivalently rewritten as \eqref{eq:ene2} (in the sense of Definition \ref{def:weakrenorm1}) in $Q_{4R_\mathcal K}(z_0)$, then for every $0<r\leq R_\mathcal K$, the rescaled triple $(\rho_r,u_r,\theta_r)$ is an admissible weak solution of \eqref{eq:theta-scal}
in $Q_4$ in the sense of Definition \ref{def:weak1}. Moreover, it satisfies the truncated renormalized energy inequality \eqref{eq:trunc} and \eqref{eq:trunc-minus} in $Q_4$.
\end{lemma}
\begin{proof}
Let $\psi\in C^1((-16,0),C_c^1(B_4))$ and define
\begin{align*}
\varphi(t,x)
=
\psi\left(\frac{t-t_0}{r^2},\frac{x-x_0}{r}\right).
\end{align*}
Then $\varphi\in C^1((t_0-16r^2,t_0),C_c^1(B_{4r}(x_0)))\subset C^1((t_0-16R_\mathcal K^2,t_0),C_c^1(B_{4R_\mathcal K}(x_0)))$, and hence it can be used as a test
function in the weak formulation of \eqref{eq:ene2}. Set
\begin{align*}
t=t_0+r^2s,\qquad x=x_0+ry.
\end{align*}
Then
\begin{align*}
\partial_t\varphi(t,x)
=
\frac1{r^2}\partial_s\psi(s,y),
\qquad
\nabla_x\varphi(t,x)
=
\frac1r\nabla_y\psi(s,y),
\qquad
dxdt=r^{d+2}\,dyds,
\end{align*}
and
\begin{align*}
\nabla_x\theta(t,x)
=
\frac1r\nabla_y\theta_r(s,y).
\end{align*}
For all $T\in(t_0-16r^2,t_0)$ and $S\in(-16,0)$ satisfies $T=t_0+r^2S$, by the change of variables in the weak formulation, we obtain
\begin{align*}
0
&= - \left( \int_{B_{4r}(x_0)}(c_v\rho\theta+\theta^4)\varphi(T,\cdot) dx - \int_{B_{4r}(x_0)}(c_v \rho\theta+\theta^4)\varphi(t_0-16r^2,\cdot) dx\right)\\
&+\int_{t_0-16r^2}^T\int_{B_{4r}(x_0)}
\left(
(c_v\rho\theta+\theta^4)(\partial_t\varphi
+
u\cdot\nabla_x\varphi)
+
F\varphi-
\kappa(\theta)
\nabla_x\theta\cdot\nabla_x\varphi\right) dx dt  \\
&= - r^d\left( \int_{B_{4}}(c_v\rho_r\theta_r+\theta_r^4)\psi(S,\cdot) dy - \int_{B_4}(c_v\rho_r\theta_r+\theta_r^4)\psi(-16,\cdot) dy\right)\\
&+
\int_{-16}^S\int_{B_4}
\bigg(
(c_v\rho_r\theta_r+\theta_r^4)\left(\frac1{r^2}\partial_s\psi
+
u_r\cdot\frac1r\nabla_y\psi\right)
+
F_r\psi -
\kappa(\theta_r)
\frac1r\nabla_y\theta_r\cdot\frac1r\nabla_y\psi
\bigg)
r^{d+2}\,dyds  \\
&= - r^d\left( \int_{B_4}(c_v\rho_r\theta_r+\theta_r^4)\psi(S,\cdot) dy - \int_{B_4}(c_v\rho_r\theta_r+\theta_r^4)\psi(-16,\cdot) dy\right)\\
&+
\int_{-16}^S\int_{B_{4}}
\bigg(
(c_v\rho_r\theta_r+\theta_r^4)\left(\partial_s\psi
+
r u_r\cdot\nabla_y\psi\right)
+
r^2 F_r\psi -
\kappa(\theta_r)
\nabla_y\theta_r\cdot\nabla_y\psi
\bigg)
r^{d}\,dyds.
\end{align*}
Dividing by $r^d$, this is precisely the weak formulation of \eqref{eq:theta-scal} in $Q_4$.

Next we prove that it satisfies a truncated renormalized energy inequality in $Q_4$. With a slight abuse of notation, let $\psi\in C_c^1(B_4)$ satisfies $\psi\geq 0$ and define $\varphi(x)
=
\psi\left(\frac{x-x_0}{r}\right).$ Then $\varphi\in C_c^1(B_{4r}(x_0))\subset C_c^1(B_{4R_\mathcal K}(x_0)), \varphi\geq 0$, and hence it can be used as a test
function in \eqref{eq:trunc}, and
\begin{align*}
\nabla_x\varphi(x)
=
\frac1r\nabla_y\psi(y),
\quad
dxdt=r^{d+2}\,dyds, \quad\nabla_x\theta(t,x)
=
\frac1r\nabla_y\theta_r(s,y).
\end{align*}
By the change of variables in the truncated renormalized energy inequality,
a.e. $t_0-16r^2\le t_1<t_2\le t_0$ with $t_1=t_0+r^2s_1, t_2=t_0+r^2s_2$,
\begin{align*}
0\leq& -\Big[\int_{B_{4r}} E_b\,\varphi dx\Big]_{t_1}^{t_2}
-\int_{t_1}^{t_2}\!\!\int_{B_{4r}\cap\{\theta>C\}}\kappa(\theta)\,|\nabla\theta_{C}|^2\,\varphi dx dt
+
\int_{t_1}^{t_2}\!\!\int_{B_{4r}} E_b\,u\cdot\nabla\varphi dx dt
\nonumber\\
&-\int_{t_1}^{t_2}\!\!\int_{B_{4r}}\theta_{C}\,\kappa(\theta)\nabla\theta\cdot\nabla\varphi dx dt
+\int_{t_1}^{t_2}\!\!\int_{B_{4r}} \Big(F \theta_{C}+\big(\eta_b(\theta)-\theta^4\theta_{C}\big)\operatorname{div}u\Big)\,\varphi dx dt\\
&=-\Big[\int_{B_{4}} \widetilde E_b\,\psi r^d dy\Big]_{s_1}^{s_2}
-\int_{s_1}^{s_2}\!\!\int_{B_{4}\cap\{\theta_r>C\}}\kappa(\theta_r)\,\frac{1}{r^2}|\nabla\theta_{rC}|^2\,\psi r^{d+2} dy ds\nonumber\\
&+
\int_{s_1}^{s_2}\!\!\int_{B_{4}} \widetilde E_b\,u_r\cdot\frac1r\nabla_y\psi r^{d+2} dy ds-\int_{s_1}^{s_2}\!\!\int_{B_{4}}\theta_{rC}\,\kappa(\theta_r)\frac1r\nabla_y\theta_r\cdot\frac1r\nabla_y\psi r^{d+2}dy ds\nonumber\\
&+\int_{s_1}^{s_2}\!\!\int_{B_{4}} \Big(F_r \theta_{rC}+\big(\eta_b(\theta_r)-\theta_r^4\theta_{rC}\big)\frac1r\operatorname{div}_yu_r\Big)\,\psi r^{d+2}dy ds,\nonumber
\end{align*}
where $\theta_{rC}=(\theta_r-C)_+, \,\eta_b(\theta_r)=\tfrac45\theta_{rC}^5+3C\theta_{rC}^4+4C^2\theta_{rC}^3+2C^3\theta_{rC}^2,\,\widetilde E_b=\tfrac12c_v\rho_r\,\theta_{rC}^2+\eta_b(\theta_r)\ge0$,
which implies
\begin{align}\label{eq:theta-r-trunc}
&\Big[\int_{B_{4}} \widetilde E_b\,\psi  dy\Big]_{s_1}^{s_2}
+\int_{s_1}^{s_2}\!\!\int_{B_{4}\cap\{\theta_r>C\}}\kappa(\theta_r)\,|\nabla\theta_{rC}|^2\,\psi dy ds\nonumber\\
&\leq 
\int_{s_1}^{s_2}\!\!\int_{B_{4}}r \widetilde E_b\,u_r\cdot\nabla_y\psi  dy ds-\int_{s_1}^{s_2}\!\!\int_{B_{4}}\theta_{rC}\,\kappa(\theta_r)\nabla_y\theta_r\cdot\nabla_y\psi dy ds\\
&+\int_{s_1}^{s_2}\!\!\int_{B_{4}} \Big(r^2F_r \theta_{rC}+\big(\eta_b(\theta_r)-\theta_r^4\theta_{rC}\big)r\operatorname{div}_yu_r\Big)\,\psi dy ds.\nonumber
\end{align}
Similarly, the lower-truncation inequality \eqref{eq:trunc-minus} yields, for every
nonnegative $\psi\in C_c^1(B_4)$ and for a.e. $-16\leq s_1<s_2\leq0$,
\begin{align}\label{eq:theta-r-trunc-minus}
&\left[\int_{B_4}\widetilde E_b^-\psi\,dy\right]_{s_1}^{s_2}
+\int_{s_1}^{s_2}\int_{B_4\cap\{\theta_r<C\}}
\kappa(\theta_r)|\nabla_y\theta_{rC}^-|^2\psi\,dy\,ds
\nonumber\\
&\quad\leq
\int_{s_1}^{s_2}\int_{B_4}
r\widetilde E_b^-u_r\cdot\nabla_y\psi\,dy\,ds
+\int_{s_1}^{s_2}\int_{B_4}
\theta_{rC}^-\kappa(\theta_r)
\nabla_y\theta_r\cdot\nabla_y\psi\,dy\,ds
\nonumber\\
&\qquad
+\int_{s_1}^{s_2}\int_{B_4}
\left(
-r^2F_r\theta_{rC}^-
+
\left(
\eta_b^-(\theta_r)+\theta_r^4\theta_{rC}^-
\right)
r\operatorname{div}_y u_r
\right)\psi\,dy\,ds,
\end{align}
where $\theta_{rC}^-:=(C-\theta_r)_+,\,\eta_b^-(\theta_r):=\mathbf{1}_{\{\theta_{r}<C\}}\int_{\theta_{r}}^{C}4(C-\zeta)\zeta^3\,d\zeta=2C^3(\theta_{rC}^-)^2-4C^2(\theta_{rC}^-)^3+3C(\theta_{rC}^-)^4-\frac45(\theta_{rC}^-)^5,\,\widetilde E_b^-:=\frac12c_v\rho_r(\theta_{rC}^-)^2+\eta_b^-(\theta_r)\geq0,$ which completes the proof of this lemma.
\end{proof}

Note that
\begin{align*}
\|r^2F(t_0+r^2s,x_0+ry)\|_{L^p(Q_1)}
=
r^{2-\frac5p}
\|F\|_{L^p(Q_r(z_0))}
=
r^\sigma \|F\|_{L^p(Q_r(z_0))},
\end{align*}
the natural oscillation estimate is expected to be
\begin{align*}
\operatorname{osc}_{Q_{r/2}(z_0)}\theta
\le
(1-\lambda)\operatorname{osc}_{Q_{2r}(z_0)}\theta
+
C r^\sigma \mathcal F_{2r},
\end{align*}
for some $\lambda\in(0,1)$ and $C>0$ depending on the structural
constants $\underline{\mu},\overline{\mu},
\underline{\eta},\underline{\eta}_1,
\overline{\eta},C_\mu$, the lower bound $\eta$ and $\mathcal K$. 

In what follows, since the structural constants, the lower bound $\eta$ are fixed throughout the argument, their dependence will be suppressed. Since the De Giorgi iteration is performed over varying scales $r$, and the resulting H\"older estimates are considered on compact subsets of $Q$, we only keep track of whether the constants depend on $\mathcal K, R_\mathcal K,r$ and $z_0$.

To implement the De Giorgi iteration, we introduce the normalized oscillation
\begin{align} \label{def:Omegar}
\Omega_r
=
\osc_{Q_r}\theta+K r^\sigma\mathcal F_r,
\end{align}
where $K>0$ is chosen sufficiently large, depending only on $R_\mathcal K$, but independently of $r$ and the center $z_0$.

If $\Omega_{2r}=0$, then $\theta$ is constant in $Q_{2r}$ and there is nothing to prove.
Otherwise, define
\begin{align*}
w(s,y)
=
\frac{2}{\Omega_{2r}}
\left[
\theta_r(s,y)
-
\frac{\sup_{Q_{2r}(z_0)}\theta+\inf_{Q_{2r}(z_0)}\theta}{2}
\right],
\qquad (s,y)\in Q_2.
\end{align*}
Then
\begin{align*}
|w|\le 1 \quad \text{in } Q_2.
\end{align*}
Note (\ref{theta_r}) implies well posedness of definition of $w$ in $Q_2$ for $r\leq R_\mathcal K$. Next, we prove

\begin{theorem}\label{thm:osc}
     For some sufficiently large $K=K(\mathcal K)$, suppose that $(\rho,u,\theta)$ is an admissible renormalized weak solution following Definition \ref{def:weakrenorm1} of
\eqref{nsf_mass}--\eqref{eq:theta} in  $Q=(0,T)\times\Omega$, assume that, for some $p\in(\frac52,\infty)$ and $\eta>0$, the hypotheses \eqref{hyp:main},\eqref{hyp-theta-lb} hold.
 Then for all $r\leq R_\mathcal K$ and $z_0\in\mathcal K$, we have 
 \begin{align*}
\operatorname{osc}_{Q_{1/2}}w
\le
2(1-\lambda_\ast),
\end{align*}
where $\lambda_\ast$ only depends on $\mathcal K$.
\end{theorem}
Define 
\begin{equation} \label{def:Q}
\widetilde Q := (-2,-1)\times B_1.
\end{equation}
In order to prove the Theorem \ref{thm:osc}, we need to have the following proposition.  
\begin{proposition}\label{prop:osc-reduction}
 Under the assumption of Theorem \ref{thm:osc}, if
\begin{align}
|\{w\le0\}\cap \widetilde Q|
\ge
\frac12|\widetilde Q|,
\mbox{ \ then \ }
w\le 1-\lambda_1
\qquad\text{in }Q_{1/2}; \label{prop31:1}
\end{align}
if
\begin{align}
|\{w\le0\}\cap \widetilde Q|
\le
\frac12|\widetilde Q|,
\mbox{ \ then \ }
w\ge -1+\lambda_1
\qquad\text{in }Q_{1/2},
\end{align} \label{prop31:1}
where $\lambda_1>0$ depends only on $\mathcal K$.
\end{proposition}
Proposition \ref{prop:osc-reduction} immediately implies Theorem \ref{thm:osc}. Indeed, one of the two alternatives above must hold. Consequently,
\[
\operatorname{osc}_{Q_{1/2}} w
=
\sup_{Q_{1/2}}w-\inf_{Q_{1/2}}w
\le
(1-\lambda_1)-(-1)
=
2-\lambda_1.
\]
Setting $\lambda_*=\lambda_1/2$, we obtain
\[
\operatorname{osc}_{Q_{1/2}} w
\le
2(1-\lambda_*),
\]
which is precisely the conclusion of Theorem \ref{thm:osc}.

To prove Proposition \ref{prop:osc-reduction}, we would like to follow the classical De Giorgi strategy. After normalizing the temperature, we introduce the auxiliary functions
$w_k=2^k(w-(1-2^{-k}))$ and $v_k=2^k(-w-(1-2^{-k}))$.
In the classical linear setting, such rescaled functions usually satisfy equations of the same type. However, in the present problem the equation is nonlinear, and therefore the two families of auxiliary functions lead to more involved equations, see \eqref{eq:wk}-\eqref{eq:wk-rela} and \eqref{eq:vk}-\eqref{eq:vk-rela}. For this reason, we first extract the following unified abstract equation, which captures the common structure of the equations satisfied by the auxiliary functions: \begin{align}\label{eq:v}
&\partial_t\left(c_v\rho_r v+\lambda_1v+\lambda_2v^2+\lambda_3v^3+\lambda_4v^4\right) \\
&\quad +r\operatorname{div}\left(\left(c_v\rho_r v+\lambda_1v+\lambda_2v^2+\lambda_3v^3+\lambda_4v^4\right)u_r\right)
-\operatorname{div}\left(\kappa(\theta_r)\nabla v\right)=F. \nonumber
\end{align}
Here $r\in(0,R_\mathcal K],\,\rho_r$ and $u_r$ satisfy the rescaled continuity equation, and the diffusion coefficient $\kappa(\theta_r)$ satisfies \eqref{c:kappa}-\eqref{c:kappa:2}. The following De Giorgi lemmas are formulated for this common equation and will subsequently be applied to both families.

\subsection{De Giorgi lemmas and oscillation reduction}
  In this subsection, we first state two De Giorgi-type lemmas whose proofs are postponed to Appendix~\ref{sec:degiorgi-proof}. These lemmas will be applied to suitable affine normalizations of the rescaled temperature. We then use them to prove Proposition~\ref{prop:osc-reduction}, which gives the desired reduction of the oscillation.
  
The first De Giorgi lemma converts a smallness assumption on the positive part of the solution into a pointwise lowering of the maximum in a smaller cylinder. 

\begin{lemma}\label{lem:osc-reduction}
 Let $p>\frac{5}{2}$, and let $v\in L^\infty(-4,0;L^5(B_2))\cap L^2(-4,0;\dot H^1(B_2))$ be a weak solution of \eqref{eq:v} satisfying $v\leq 1$ in $Q_2$.
Let
$P(\zeta):=\lambda_1\zeta+\lambda_2\zeta^2+\lambda_3\zeta^3+\lambda_4\zeta^4$.
Assume that
\begin{enumerate}
    \item for a.e. $(t,x)\in Q_2\cap \{v>0\}$ and $\zeta\in[0,v(t,x)]$,
    \begin{align*}
0<\underline{\lambda}
\leq P'(\zeta)
=\lambda_1+2\lambda_2\zeta+3\lambda_3\zeta^2+4\lambda_4\zeta^3
\leq \overline{\lambda},
\end{align*}
\item fix a level $C\in\RR$ and write $v_{C}:=(v-C)_+$, so that
	$\nabla v_{C}=\nabla v \mathbf 1_{\{v>C\}}$, for every nonnegative $\psi\in C^1_c(B_2)$ and a.e.\ $-4\le s<t\le 0$,
	\begin{align}\label{es:pre-energy}
		&\Big[\int_{B_2}\left(\tfrac12\rho_r\,v_{C}^2+H_C(v)\right)\,\psi\,dx\Big]_{s}^{t}
+\int_s^t\!\!\int_{B_2} \,\kappa(\theta_r)\,|\nabla v_C|^2\,\psi\nonumber\\
		&\leq\;
		\int_s^t\!\!\int_{B_2} \left(\tfrac12\rho_r\,v_{C}^2+H_C(v)\right)\,r\,u_r\cdot\nabla\psi
		\;-\;\int_s^t\!\!\int_{B_2} v_C\,\kappa(\theta_r)\nabla v_C\cdot\nabla\psi\\
		&\,\,+\int_s^t\!\!\int_{B_2} \left(v_C F + \left(H_C(v)- P(v) v_C\right) r\operatorname{div}u_r\right)\psi,\nonumber
	\end{align}
     where
\begin{align*}
    H_C(v):=\int_C^v (\zeta-C)_+ P'(\zeta)\,d \zeta\ge0.
\end{align*}
\end{enumerate}
 Then there exists
$\delta=\delta(\underline{\lambda},\overline{\lambda},\mathcal K)>0$ such that, if
\begin{align*}
\int_{Q_1}|v_+|^2\,dxdt\leq \delta, \,\,\,\text{and}\,\,\,\|F\|_{L^p(Q_1\cap\{v>0\})}\leq 1,
\mbox{ \ \ then \ \ }
v\leq \tfrac{1}{2}
\qquad\text{in \ } Q_{\frac{1}{2}}.
\end{align*}
\end{lemma}
The second ingredient is the corresponding intermediate value lemma. It prevents the solutions from passing directly from the lower phase to the upper phase without occupying a quantitatively nontrivial intermediate region.
\begin{lemma}\label{lem:mid-value}
Let the assumptions of Lemma \ref{lem:osc-reduction} hold and assume further $\|F\|_{L^p(Q_2\cap\{v>0\})}\leq 1$ where $p>\frac52$. Denote
  \begin{align*}
  |A|:=\left|\{v\geq \tfrac12\}\cap Q_1\right|,\, |C|:=\left|\{v\leq 0\}\cap {\widetilde Q}\right|,\,|D|:=\left|\{0<v<\tfrac12\}\cap \left((-2,0)\times B_{\frac54}\right)\right|,
  \end{align*}
  where $\widetilde Q$ is defined in \eqref{def:Q}.
  There exists $\beta>0$ such that if $|A|\geq \delta,\, |C|\geq \frac12 \left|{\widetilde Q}\right|$, then $|D|\geq \beta$, where $\beta$ depends on $\delta, \underline{\lambda},\overline{\lambda}$ and $\mathcal K$.
\end{lemma}
The proof of Lemma~\ref{lem:mid-value} is based on the quantitative approach for parabolic De Giorgi classes developed in Guerand~\cite{JG20}. The classical proof of the parabolic intermediate value lemma often relies on a contradiction argument combined with compactness; see, for instance,~\cite{Vasseur16}. Such an argument requires a suitable compactness mechanism for a sequence of solutions. In our case, however, the density coefficient $\rho_r$ is available only in $L^\infty$, and the required compactness is not directly accessible. Therefore we use instead a quantitative proof based only on the De Giorgi energy inequality and measure estimates. Although this yields a slightly weaker form of the intermediate value lemma, it is sufficient for deriving Proposition~\ref{prop:osc-reduction}.

Together, these two results provide the measure-theoretic mechanism required
for the oscillation reduction argument. We now turn to the proof of Proposition~\ref{prop:osc-reduction}. 
\begin{proof}[Proof of Proposition~\ref{prop:osc-reduction}]
   For the proof of the first part, we consider the sequence \begin{align}\label{id:wk}
 w_k(x,t)&=2^k(w-(1-2^{-k}))\nonumber\\
 &=\frac{2^{k+1}}{\Omega_{2r}}\left(\theta_r
-\frac12
\left[\sup_{Q_{2r}(z_0)}\theta+\inf_{Q_{2r}(z_0)}\theta+\Omega_{2r}(1-2^{-k})
\right]\right) \\
&=:\frac{2^{k+1}}{\Omega_{2r}}\left(\theta_r
-C_{r,k}\right),\nonumber
\end{align}
one can deduce the equation of $w_k$ reads
\begin{align}\label{eq:wk}
  &\partial_t \left(c_v\rho_r w_k + \lambda_1 w_k + \lambda_2 w_k^2 + \lambda_3 w_k^3 +\lambda_4 w_k^4\right) \\
  &\,\,\,+ r \text{div}\left(\left(c_v\rho_r w_k + \lambda_1 w_k + \lambda_2 w_k^2 + \lambda_3 w_k^3 +\lambda_4 w_k^4\right)u_r\right)-\text{div}(\kappa(\theta_r)\nabla w_k)=\lambda_5r^2 \left(F_r-C_{r,k}^4(\text{div}\,u)_r\right),\nonumber
\end{align}
where 
\begin{align}\label{eq:wk-rela}
  \lambda_1=  4 C_{r,k}^3,
  \lambda_2=  6 C_{r,k}^2 \frac{\Omega_{2r}}{2^{k+1}},
  \lambda_3=  4 C_{r,k}\left(\frac{\Omega_{2r}}{2^{k+1}}\right)^2,
  \lambda_4= \left(\frac{\Omega_{2r}}{2^{k+1}}\right)^3,
  \lambda_5= \frac{2^{k+1}}{\Omega_{2r}}.
\end{align}

Since $-1\leq w\leq 1$, one can get $w_k\leq 1$ for all $k$.

For a.e. $(t,x)\in Q_2\cap \{w_k>0\}$ and $\zeta\in[0,w_k(t,x)]$, we have
\begin{align*}
    P'(\zeta)=&\lambda_1+2\lambda_2\zeta+3\lambda_3\zeta^2+4\lambda_4\zeta^3\\
    &=4 C_{r,k}^3 + 12 C_{r,k}^2\frac{\Omega_{2r}}{2^{k+1}}\zeta + 12 C_{r,k}\left(\frac{\Omega_{2r}}{2^{k+1}}\zeta\right)^2 + 4\left(\frac{\Omega_{2r}}{2^{k+1}}\zeta \right)^3\\
    &=4 \left(C_{r,k}+\frac{\Omega_{2r}}{2^{k+1}}\zeta \right)^3,
\end{align*}
since
\begin{align*}
   4C_{r,k}^3\leq  4 \left(C_{r,k}+\frac{\Omega_{2r}}{2^{k+1}}\zeta \right)^3\leq 4 \left(C_{r,k}+\frac{\Omega_{2r}}{2^{k+1}}w_k(t,x) \right)^3=4\theta_r(t,x)^3,
\end{align*}
recalling the definition \eqref{id:wk} of $C_{r,k}$ one can get
\begin{align*}
     \frac12\left(\inf\limits_{Q_{2R_\mathcal K}(z_0)}\theta\right)^3\leq\frac12\left(\sup\limits_{Q_{2r}(z_0)}\theta+\inf\limits_{Q_{2r}(z_0)}\theta\right)^3 \leq 4C_{r,k}^3 \leq P'(\zeta) \leq 4\sup\limits_{Q_{2r}(z_0)}\theta^3\leq 4\sup\limits_{Q_{2R_\mathcal K}(z_0)}\theta^3\leq 4 N_\mathcal K^3,
\end{align*}
where $N_\mathcal K$ is the uniform upper bound of $\theta$ derived in Theorem \ref{cor:uni-upb}. Here we need the fact that $\theta$ has positive lower bound to satisfy the structural condition in Lemma~\ref{lem:osc-reduction}.
Therefore, condition {\rm (i)} holds with
\begin{align*}
\underline{\lambda}=\frac{1}{2}\eta^3,
\qquad
\overline{\lambda}=4N_\mathcal K^3.
\end{align*}
In particular, these constants depend only on $\mathcal K$ and $\eta$, but are independent on $r, z_0$.

Next, we use \eqref{eq:theta-r-trunc} and \eqref{id:wk} to prove condition~{\rm (ii)}. Thanks to \eqref{id:wk}, we have 
\begin{align}\label{id:theta-r}
\theta_r=\frac{\Omega_{2r}}{2^{k+1}}w_k+C_{r,k}.
\end{align}
In particular, in \eqref{eq:theta-r-trunc}, we take $\psi\in C_c^1(B_2)$, one can get
\begin{align}\label{es:theta-r-trunc2}
&\Big[\int_{B_{2}} \widetilde E_b\,\psi  dy\Big]_{s_1}^{s_2}
+\int_{s_1}^{s_2}\!\!\int_{B_{2}\cap\{\theta_r>C\}}\kappa(\theta_r)\,|\nabla\theta_{rC}|^2\,\psi dy ds\nonumber\\
&\leq 
\int_{s_1}^{s_2}\!\!\int_{B_{2}}r \widetilde E_b\,u_r\cdot\nabla_y\psi  dy ds-\int_{s_1}^{s_2}\!\!\int_{B_{2}}\theta_{rC}\,\kappa(\theta_r)\nabla_y\theta_r\cdot\nabla_y\psi dy ds\\
&+\int_{s_1}^{s_2}\!\!\int_{B_{2}} \Big(r^2F_r \theta_{rC}+\big(\eta_b(\theta_r)-\theta_r^4\theta_{rC}\big)r\operatorname{div}_yu_r\Big)\,\psi dy ds,\nonumber
\end{align}
holds true for all $C>0$. In particular, it holds true for $C:= C_{r,k}+\frac{\Omega_{2r}}{2^{k+1}}\widetilde C$ for $\widetilde C>-\frac{2^{k+1}}{\Omega_{2r}}C_{r,k}$.
Thanks to \eqref{id:theta-r}, one can get 
\begin{align*}
\theta_{rC}=&\frac{\Omega_{2r}}{2^{k+1}}(w_k-\widetilde C)_+,\\
\eta_b(\theta_r)=&\tfrac45\theta_{rC}^5+3C\theta_{rC}^4+4C^2\theta_{rC}^3+2C^3\theta_{rC}^2 \\
=&\left(\frac{\Omega_{2r}}{2^{k+1}}\right)^2\left(\tfrac45\left(\frac{\Omega_{2r}}{2^{k+1}}\right)^3(w_k-\widetilde C)_+^5+3\left(C_{r,k}+\frac{\Omega_{2r}}{2^{k+1}}\widetilde C\right)\left(\frac{\Omega_{2r}}{2^{k+1}}\right)^2(w_k-\widetilde C)_+^4\right.\\
    &\left.+4\left(C_{r,k}+\frac{\Omega_{2r}}{2^{k+1}}\widetilde C\right)^2\left(\frac{\Omega_{2r}}{2^{k+1}}\right)(w_k-\widetilde C)_+^3+2\left(C_{r,k}+\frac{\Omega_{2r}}{2^{k+1}}\widetilde C\right)^3(w_k-\widetilde C)_+^2\right)\\
    &=\left(\frac{\Omega_{2r}}{2^{k+1}}\right)^2 H_{\widetilde C}(w_k).
    \end{align*}
Indeed, by the definition of $H_{\widetilde C}$ and the identity $P'(\zeta)
=
4\left(
C_{r,k}
+\frac{\Omega_{2r}}{2^{k+1}}\zeta
\right)^3,$
we have
\begin{align*}
H_{\widetilde C}(w_k)
=
\int_{\widetilde C}^{w_k}
(\zeta-\widetilde C)_+P'(\zeta)\,d\zeta=
4\int_{\widetilde C}^{w_k}
(\zeta-\widetilde C)_+
\left(
C_{r,k}
+\frac{\Omega_{2r}}{2^{k+1}}\zeta
\right)^3
d\zeta.
\end{align*}
If $w_k\leq\widetilde C$, both sides vanish. If
$w_k>\widetilde C$, making the change of variables
$y=\zeta-\widetilde C$, we obtain
\begin{align*}
H_{\widetilde C}(w_k)
&=
4\int_0^{w_k-\widetilde C}
y
\left(
C_{r,k}
+\frac{\Omega_{2r}}{2^{k+1}}\widetilde C
+\frac{\Omega_{2r}}{2^{k+1}}y
\right)^3
dy\\
&=
2\left(
C_{r,k}
+\frac{\Omega_{2r}}{2^{k+1}}\widetilde C
\right)^3
(w_k-\widetilde C)_+^2
+4\left(
C_{r,k}
+\frac{\Omega_{2r}}{2^{k+1}}\widetilde C
\right)^2
\frac{\Omega_{2r}}{2^{k+1}}
(w_k-\widetilde C)_+^3\\
&\quad
+3\left(
C_{r,k}
+\frac{\Omega_{2r}}{2^{k+1}}\widetilde C
\right)
\left(
\frac{\Omega_{2r}}{2^{k+1}}
\right)^2
(w_k-\widetilde C)_+^4
+\frac45
\left(
\frac{\Omega_{2r}}{2^{k+1}}
\right)^3
(w_k-\widetilde C)_+^5,
\end{align*}
therefore $\eta_b(\theta_r)
=
\left(
\frac{\Omega_{2r}}{2^{k+1}}
\right)^2
H_{\widetilde C}(w_k).$ Consequently, one can get 
\begin{align*}
    \widetilde E_b=&\tfrac12c_v\rho_r\,\theta_{rC}^2+\eta_b(\theta_r)=\left(\frac{\Omega_{2r}}{2^{k+1}}\right)^2\left(\tfrac12c_v\rho_r\,(w_k-\widetilde C)_+^2 + H_{\widetilde C}(w_k)\right),\\
 \theta_r^4=&\left(\frac{\Omega_{2r}}{2^{k+1}}w_k+C_{r,k}\right)^4\\
 &=\left(\frac{\Omega_{2r}}{2^{k+1}}w_k\right)^4 + 4C_{r,k}\left(\frac{\Omega_{2r}}{2^{k+1}}w_k\right)^3+ 6C_{r,k}^2\left(\frac{\Omega_{2r}}{2^{k+1}}w_k\right)^2 + 4C_{r,k}^3\left(\frac{\Omega_{2r}}{2^{k+1}}w_k\right) + C_{r,k}^4\\
 &=\frac{\Omega_{2r}}{2^{k+1}}P(w_k) + C_{r,k}^4.
    \end{align*}

Moreover, we have
\begin{align*}
\nabla_y\theta_{rC}
&=
\frac{\Omega_{2r}}{2^{k+1}}
\nabla_y(w_k-\widetilde C)_+,
\qquad
\nabla_y\theta_r
=
\frac{\Omega_{2r}}{2^{k+1}}\nabla_yw_k.
\end{align*}
Therefore,
\begin{align*}
\kappa(\theta_r)|\nabla_y\theta_{rC}|^2
&=
\left(\frac{\Omega_{2r}}{2^{k+1}}\right)^2
\kappa(\theta_r)
|\nabla_y(w_k-\widetilde C)_+|^2,
\end{align*}
and
\begin{align*}
\theta_{rC}\kappa(\theta_r)
\nabla_y\theta_r\cdot\nabla_y\psi
&=
\left(\frac{\Omega_{2r}}{2^{k+1}}\right)^2
(w_k-\widetilde C)_+
\kappa(\theta_r)
\nabla_y(w_k-\widetilde C)_+
\cdot\nabla_y\psi.
\end{align*}
Furthermore, recalling that
$\lambda_5=\frac{2^{k+1}}{\Omega_{2r}},
r\operatorname{div}_yu_r
=
r^2(\operatorname{div}u)_r,
$
we obtain
\begin{align*}
&r^2F_r\theta_{rC}
+
\left(
\eta_b(\theta_r)-\theta_r^4\theta_{rC}
\right)
r\operatorname{div}_yu_r
\\
&=
\frac{\Omega_{2r}}{2^{k+1}}
r^2F_r(w_k-\widetilde C)_+
+
\Bigg[
\left(\frac{\Omega_{2r}}{2^{k+1}}\right)^2
H_{\widetilde C}(w_k)
-
\left(
\frac{\Omega_{2r}}{2^{k+1}}P(w_k)
+C_{r,k}^4
\right)
\frac{\Omega_{2r}}{2^{k+1}}
(w_k-\widetilde C)_+
\Bigg]
r\operatorname{div}_yu_r
\\
&=
\left(\frac{\Omega_{2r}}{2^{k+1}}\right)^2
\Bigg[
\lambda_5r^2
\left(
F_r-C_{r,k}^4(\operatorname{div}u)_r
\right)
(w_k-\widetilde C)_++
\left(
H_{\widetilde C}(w_k)
-P(w_k)(w_k-\widetilde C)_+
\right)
r\operatorname{div}_yu_r
\Bigg].
\end{align*}
Substituting these identities into \eqref{es:theta-r-trunc2}
and dividing by
$\left(\frac{\Omega_{2r}}{2^{k+1}}\right)^2$, we obtain
\begin{align*}
&\left[
\int_{B_2}
\left(
\frac12c_v\rho_r(w_k-\widetilde C)_+^2
+H_{\widetilde C}(w_k)
\right)\psi\,dy
\right]_{s_1}^{s_2}
+\int_{s_1}^{s_2}\int_{B_2}
\kappa(\theta_r)
|\nabla_y(w_k-\widetilde C)_+|^2
\psi\,dy\,ds
\\
&\leq
\int_{s_1}^{s_2}\int_{B_2}
\left(
\frac12c_v\rho_r(w_k-\widetilde C)_+^2
+H_{\widetilde C}(w_k)
\right)
r u_r\cdot\nabla_y\psi\,dy\,ds
\\
&\quad
-\int_{s_1}^{s_2}\int_{B_2}
(w_k-\widetilde C)_+
\kappa(\theta_r)
\nabla_y(w_k-\widetilde C)_+
\cdot\nabla_y\psi\,dy\,ds
\\
&\quad
+\int_{s_1}^{s_2}\int_{B_2}
\Bigg[
\lambda_5r^2
\left(
F_r-C_{r,k}^4(\operatorname{div}u)_r
\right)
(w_k-\widetilde C)_+
+
\left(
H_{\widetilde C}(w_k)
-P(w_k)(w_k-\widetilde C)_+
\right)
r\operatorname{div}_yu_r
\Bigg]\psi\,dy\,ds.
\end{align*}
This is precisely the inequality \eqref{es:pre-energy}
corresponding to equation \eqref{eq:wk}. Hence
condition~{\rm (ii)} is satisfied.
  
  Together with condition~{\rm (i)}, this allows us to apply Lemma~\ref{lem:osc-reduction}, which yields a constant $\delta=\delta(\underline{\lambda},\overline{\lambda},\mathcal K)=\delta(\mathcal K)>0$ such that the asserted conclusion holds.
Let $\delta>0$ be the constant from Lemma \ref{lem:osc-reduction}, and let
$\beta=\beta(\delta,\underline{\lambda},\overline{\lambda},\mathcal K)=\beta(\mathcal K)>0$ be the constant given by Lemma \ref{lem:mid-value}. 
Choose $k_0\in\mathbb N$ such that
\begin{align*}
(k_0+1)\beta \geq \left|(-2,0)\times B_{\frac54}\right|.
\end{align*}

If $w\leq 0$, then $w_k\leq 0$, which implies $\{w\leq 0\}\cap \widetilde Q\subset \{w_k\leq 0\}\cap \widetilde Q$, and 
\begin{align}\label{es:wk-set}
\left|\{w_k\leq 0\}\cap \widetilde Q\right|\geq \left|\{w\leq 0\}\cap  \widetilde Q\right|\geq \frac12 |\widetilde Q|.
\end{align}

For all $k\leq k_0$, note that on the set $Q_2\cap\{w_k>0\},$ we have $C_{r,k}<\theta_r\leq \sup\limits_{Q_{2r}(z_0)}\theta\leq \sup\limits_{Q_{2R_\mathcal K}(z_0)}\theta$, then
\begin{align*}
    &\lambda_5r^2 \|F_r-C_{r,k}^4(\text{div}\,u)_r\|_{L^p(Q_2\cap\{w_k>0\})}\\
    &\leq \lambda_5r^{2-\frac5p}(\|F\|_{L^p(Q_{2r})}+(\sup\limits_{Q_{2R_\mathcal K}}\theta)^4\|\text{div}\,u\|_{L^p(Q_{2r})})\\
     &\leq C_\mathcal K\lambda_5r^{2-\frac5p}\mathcal{F}_{2r}= C_\mathcal K2^{k+1} \frac{r^{2-\frac5p}\mathcal{F}_{2r}}{\Omega_{2r}} \leq C_\mathcal K2^{k_0+1}\frac1K,
\end{align*}
where $\mathcal{F}_{2r}$ is defined in \eqref{def:CFR}
then we can select sufficiently large $K$ such that 
\begin{align*}
\lambda_5r^2 \|F_r-C_{r,k}^4(\text{div}\,u)_r\|_{L^p(Q_2\cap\{w_k>0\})}\leq 1. 
\end{align*}
Since $C_\mathcal K$ and $k_0$ only depend on $\mathcal K$, one can get $K$ only depends on $\mathcal K$.

Note that $w_{k+1}=2 w_k-1$, so that $\{w_{k+1}>0\}=\{w_{k}>\frac12\}$ and $\{w_{k+1}\geq 0\}=\{w_{k}\geq\frac12\}$.

Next we want to prove that there exists $k\leq k_0$ such that $\int_{Q_1}(w_{k+1})_+^2 dx dt\leq\delta$, by using the contradiction argument.
So we assume that for all $k\leq k_0$, we have $\int_{Q_1}(w_{k+1})_+^2 dx dt>\delta$.

Since $(w_{k+1})_+\leq 1$, this implies 
\begin{align*}
    \left|A^w_k\right|:=\left|\{w_k\geq\frac12\}\cap Q_1\right|=\left|\{w_{k+1}\geq 0\}\cap Q_1\right|\geq \int_{Q_1}(w_{k+1})_+^2 dx dt\geq \delta.
\end{align*}
Combining the above inequality with \eqref{es:wk-set}, one can use Lemma \ref{lem:mid-value} for $w_k$ to get
\begin{align} \label{es:Dkw}
    |D^w_k|:=\left|\{0<w_k<\frac12\}\cap ((-2,0)\times B_{\frac54})\right|\geq \beta.
\end{align}
Since $w_{k+1}=2w_k-1$, one can get $\{w_{k+1}\leq 0\}\supset \{w_{k}\leq 0\}\cup \{0<w_{k}<\frac12\}$ which implies
\begin{align*}
  \left\{ \{w_{k+1}\leq 0\}\cap ((-2,0)\times B_{\frac54})\right\}\supset \left\{\{w_{k}\leq 0\}\cap ((-2,0)\times B_{\frac54}) \right\}\cup D_k.
\end{align*}
Then one can get
\begin{align*}
    \left|(-2,0)\times B_{\frac54}\right|\geq&\left|\left\{ \{w_{k_0+1}\leq 0\}\cap ((-2,0)\times B_{\frac54})\right\} \right|\\
    \geq& \left| \left\{\{w_{k_0}\leq 0\}\cap ((-2,0)\times B_{\frac54}) \right\} \right| + |D_{k_0}|\\
    &\geq \left\{\{w_{k_0-1}\leq 0\}\cap ((-2,0)\times B_{\frac54}) \right\} + |D_{k_0-1}| + |D_{k_0}|\\
    &\geq \left\{\{w_{k_0-2}\leq 0\}\cap ((-2,0)\times B_{\frac54}) \right\} + + |D_{k_0-2}| + |D_{k_0-1}| + |D_{k_0}|\\
    &\geq...\geq \left\{\{w_0\leq 0\}\cap ((-2,0)\times B_{\frac54}) \right\} + \sum\limits_{i=0}^{k_0} |D_i|\\
    &= \left\{\{w\leq 0\}\cap ((-2,0)\times B_{\frac54}) \right\} + \sum\limits_{i=0}^{k_0} |D_i|\\
    &\geq \left\{\{w\leq 0\}\cap (Q_1\cup \widetilde Q) \right\} + \sum\limits_{i=0}^{k_0} |D_i|\,\geq \,\frac12 \left|\widetilde Q\right| + (k_0+1)\beta,
\end{align*}
where in the last passage we have used \eqref{es:Dkw}. As $(k_0+1)\beta\geq  \left|(-2,0)\times B_{\frac54}\right|$, one can get a contradiction, so there exists $k\leq k_0$ such that $\int_{Q_1}(w_{k+1})_+^2 dx dt\leq\delta$. By using Lemma~\ref{lem:osc-reduction}, one can get
\begin{align*}
    w_{k+1}=2^{k+1}(w-(1-2^{-k-1}))\leq \frac12 \quad {\rm in} \; Q_{\frac12},
\end{align*}
which implies $w\leq 1-2^{-k-2}\leq 1-2^{-k_0-2}$ in $Q_{\frac12}$, then we choose $\lambda_1=2^{-k_0-2}$ and complete the proof of the first part.\\
The proof of the second part is similar.
We consider the sequence \begin{align}\label{id:vk}
 v_k(x,t)&=2^k(-w-(1-2^{-k}))\nonumber\\
 &=\frac{2^{k+1}}{\Omega_{2r}}\left(\frac12
\left[\sup_{Q_{2r}(z_0)}\theta+\inf_{Q_{2r}(z_0)}\theta - \Omega_{2r}(1-2^{-k})
\right]-\theta_r
\right) \\
&=:\frac{2^{k+1}}{\Omega_{2r}}\left(D_{r,k}-\theta_r
\right).\nonumber
\end{align}
First one can deduce the equation of $v_k$ reads
\begin{align}\label{eq:vk}
  &\partial_t \left(c_v\rho_r v_k + \lambda_1 v_k + \lambda_2 v_k^2 + \lambda_3 v_k^3 +\lambda_4 v_k^4\right) \\
  &\,\,\,+ r \text{div}\left(\left(c_v \rho_r v_k + \lambda_1 v_k + \lambda_2 v_k^2 + \lambda_3 v_k^3 +\lambda_4 v_k^4\right)u_r\right)-\text{div}(\kappa(\theta_r)\nabla v_k)=\lambda_5r^2 \left(F_r-D_{r,k}^4(\text{div}\,u)_r\right),\nonumber
\end{align}
where 
\begin{align}\label{eq:vk-rela}
  \lambda_1=  4 D_{r,k}^3,
  \lambda_2= - 6 D_{r,k}^2 \frac{\Omega_{2r}}{2^{k+1}},
  \lambda_3=  4 D_{r,k}\left(\frac{\Omega_{2r}}{2^{k+1}}\right)^2,
  \lambda_4= -\left(\frac{\Omega_{2r}}{2^{k+1}}\right)^3,
  \lambda_5= -\frac{2^{k+1}}{\Omega_{2r}}.
\end{align}
And we can we can check that $\underline{\lambda}, \overline{\lambda}$ are independent on $k$ and $K$.

For a.e. $(t,x)\in Q_2\cap \{v_k>0\}$ and $\zeta\in[0,v_k(t,x)]$, we have
\begin{align*}
    P'(\zeta)=&\lambda_1+2\lambda_2\zeta+3\lambda_3\zeta^2+4\lambda_4\zeta^3\\
    &=4 D_{r,k}^3 - 12 D_{r,k}^2\frac{\Omega_{2r}}{2^{k+1}}\zeta + 12 D_{r,k}\left(\frac{\Omega_{2r}}{2^{k+1}}\zeta\right)^2 - 4\left(\frac{\Omega_{2r}}{2^{k+1}}\zeta \right)^3\\
    &=4 \left(D_{r,k}-\frac{\Omega_{2r}}{2^{k+1}}\zeta \right)^3,
\end{align*}
since
\begin{align*}
   4D_{r,k}^3\geq  4 \left(D_{r,k}-\frac{\Omega_{2r}}{2^{k+1}}\zeta \right)^3\geq 4 \left(D_{r,k}-\frac{\Omega_{2r}}{2^{k+1}}v_k(t,x) \right)^3=4\theta_r(t,x)^3,
\end{align*}
recalling the definition \eqref{id:vk} of $D_{r,k}$ one can get
\begin{align*}
   4\inf\limits_{Q_{2R_\mathcal K}(z_0)}\theta^3 \leq 4\inf\limits_{Q_{2r}(z_0)}\theta^3 \leq P'(\zeta) \leq 4D_{r,k}^3 \leq\frac12\left(\sup\limits_{Q_{2r}(z_0)}\theta+\inf\limits_{Q_{2r}(z_0)}\theta\right)^3\leq 4 \sup\limits_{Q_{2R_\mathcal K}(z_0)}\theta^3.
\end{align*}
 
Therefore, condition {\rm (i)} holds with
\begin{align*}
\underline{\lambda}=4\eta^3,
\qquad
\overline{\lambda}=4N_\mathcal K^3.
\end{align*}
In particular, these constants depend only on $\mathcal K$ and $\eta$, but independent on $r$ or $z_0$.

Next, we use \eqref{eq:theta-r-trunc-minus} and \eqref{id:vk} to prove condition~{\rm (ii)}. Thanks to \eqref{id:vk}, we have 
\begin{align}
\theta_r=D_{r,k}-\frac{\Omega_{2r}}{2^{k+1}}v_k.
\end{align}

In particular, in \eqref{eq:theta-r-trunc-minus}, we take $\psi\in C_c^1(B_2)$, one can get 
\begin{align}\label{es:theta-r-trunc-minus2}
&\left[\int_{B_2}\widetilde E_b^-\psi\,dy\right]_{s_1}^{s_2}
+\int_{s_1}^{s_2}\int_{B_2\cap\{\theta_r<C\}}
\kappa(\theta_r)|\nabla_y\theta_{rC}^-|^2\psi\,dy\,ds
\nonumber\\
&\quad\leq
\int_{s_1}^{s_2}\int_{B_2}
r\widetilde E_b^-u_r\cdot\nabla_y\psi\,dy\,ds
+\int_{s_1}^{s_2}\int_{B_2}
\theta_{rC}^-\kappa(\theta_r)
\nabla_y\theta_r\cdot\nabla_y\psi\,dy\,ds
\nonumber\\
&\qquad
+\int_{s_1}^{s_2}\int_{B_2}
\left(
-r^2F_r\theta_{rC}^-
+
\left(
\eta_b^-(\theta_r)+\theta_r^4\theta_{rC}^-
\right)
r\operatorname{div}_y u_r
\right)\psi\,dy\,ds,
\end{align}
holds true for all $C>0$. In particular, it holds true for $C:= D_{r,k}-\frac{\Omega_{2r}}{2^{k+1}}\widetilde C$ for $\widetilde C<\frac{2^{k+1}}{\Omega_{2r}}D_{r,k}$.
Thanks to \eqref{id:theta-r}, one can get 
\begin{align*}
\theta_{rC}^-
&=
(C-\theta_r)_+\\
&=
\left(
D_{r,k}
-\frac{\Omega_{2r}}{2^{k+1}}\widetilde C
-D_{r,k}
+\frac{\Omega_{2r}}{2^{k+1}}v_k
\right)_+
=
\frac{\Omega_{2r}}{2^{k+1}}
(v_k-\widetilde C)_+.
\end{align*}
Moreover,
\begin{align*}
\eta_b^-(\theta_r)
&=
2C^3(\theta_{rC}^-)^2
-4C^2(\theta_{rC}^-)^3
+3C(\theta_{rC}^-)^4
-\frac45(\theta_{rC}^-)^5\\
&=
\left(\frac{\Omega_{2r}}{2^{k+1}}\right)^2
\Bigg(
2\left(
D_{r,k}
-\frac{\Omega_{2r}}{2^{k+1}}\widetilde C
\right)^3
(v_k-\widetilde C)_+^2 
-4\left(
D_{r,k}
-\frac{\Omega_{2r}}{2^{k+1}}\widetilde C
\right)^2
\frac{\Omega_{2r}}{2^{k+1}}
(v_k-\widetilde C)_+^3\\
&\qquad
+3\left(
D_{r,k}
-\frac{\Omega_{2r}}{2^{k+1}}\widetilde C
\right)
\left(\frac{\Omega_{2r}}{2^{k+1}}\right)^2
(v_k-\widetilde C)_+^4
-\frac45
\left(\frac{\Omega_{2r}}{2^{k+1}}\right)^3
(v_k-\widetilde C)_+^5
\Bigg)\\
&=
\left(\frac{\Omega_{2r}}{2^{k+1}}\right)^2
H_{\widetilde C}(v_k).
\end{align*}
Indeed, since $P'(\zeta)
=
4\left(
D_{r,k}
-\frac{\Omega_{2r}}{2^{k+1}}\zeta
\right)^3,$ the expression in parentheses is precisely
\begin{align*}
H_{\widetilde C}(v_k)
:=
\int_{\widetilde C}^{v_k}
(\zeta-\widetilde C)_+P'(\zeta)\,d\zeta.
\end{align*}

Consequently, one can get
\begin{align*}
\widetilde E_b^-
=
\frac12c_v\rho_r(\theta_{rC}^-)^2
+\eta_b^-(\theta_r)
=
\left(\frac{\Omega_{2r}}{2^{k+1}}\right)^2
\left(
\frac12c_v\rho_r(v_k-\widetilde C)_+^2
+H_{\widetilde C}(v_k)
\right).
\end{align*}

Furthermore,
\begin{align*}
\theta_r^4
&=
\left(
D_{r,k}
-\frac{\Omega_{2r}}{2^{k+1}}v_k
\right)^4\\
&=
D_{r,k}^4
-4D_{r,k}^3
\left(
\frac{\Omega_{2r}}{2^{k+1}}v_k
\right)
+6D_{r,k}^2
\left(
\frac{\Omega_{2r}}{2^{k+1}}v_k
\right)^2
-4D_{r,k}
\left(
\frac{\Omega_{2r}}{2^{k+1}}v_k
\right)^3
+
\left(
\frac{\Omega_{2r}}{2^{k+1}}v_k
\right)^4\\
&=
D_{r,k}^4
-\frac{\Omega_{2r}}{2^{k+1}}P(v_k).
\end{align*}

The diffusion cross term satisfies
\begin{align*}
\theta_{rC}^-\kappa(\theta_r)
\nabla_y\theta_r\cdot\nabla_y\psi
&=
-\left(\frac{\Omega_{2r}}{2^{k+1}}\right)^2
(v_k-\widetilde C)_+
\kappa(\theta_r)
\nabla_y(v_k-\widetilde C)_+
\cdot\nabla_y\psi.
\end{align*}
Finally, using
$\lambda_5
=
-\frac{2^{k+1}}{\Omega_{2r}},
r\operatorname{div}_yu_r
=
r^2(\operatorname{div}u)_r,$
we obtain
\begin{align*}
&-r^2F_r\theta_{rC}^-
+
\left(
\eta_b^-(\theta_r)
+\theta_r^4\theta_{rC}^-
\right)
r\operatorname{div}_yu_r\\
&=
\left(\frac{\Omega_{2r}}{2^{k+1}}\right)^2
\Bigg[
\lambda_5r^2
\left(
F_r-D_{r,k}^4(\operatorname{div}u)_r
\right)
(v_k-\widetilde C)_++
\left(
H_{\widetilde C}(v_k)
-P(v_k)(v_k-\widetilde C)_+
\right)
r\operatorname{div}_yu_r
\Bigg].
\end{align*}
  Substituting these identities into \eqref{es:theta-r-trunc2} and dividing by $\left(\frac{\Omega_{2r}}{2^{k+1}}\right)^2$, we obtain the desired inequality \eqref{es:pre-energy} for \eqref{eq:vk}. Hence condition~{\rm (ii)} is satisfied. 
  
  Together with condition~{\rm (i)}, this allows us to apply Lemma~\ref{lem:osc-reduction}, which yields a constant $\delta=\delta(\underline{\lambda},\overline{\lambda},\mathcal K)=\delta(\mathcal K)>0$ such that the asserted conclusion holds.
Let $\delta>0$ be the constant from Lemma \ref{lem:osc-reduction}, and let
$\beta=\beta(\delta,\underline{\lambda},\overline{\lambda},\mathcal K)=\beta(\mathcal K)>0$ be the constant given by Lemma \ref{lem:mid-value}.

If $w\geq 0$, then $v_k\leq 0$, which implies $\{w\geq 0\}\cap \widetilde Q\subset \{v_k\leq 0\}\cap \widetilde Q$, and 
\begin{align}\label{es:wk-set2}
\left|\{v_k\leq 0\}\cap \widetilde Q\right|\geq \left|\{w\geq 0\}\cap  \widetilde Q\right|\geq \frac12 |\widetilde Q|.
\end{align}

Similarly as above, for all $k\leq k_0$, we can select sufficiently large $K$ only depends on $\mathcal K$ such that 
\begin{align*}
\lambda_5r^2 \|F_r-D_{r,k}^4(\text{div}\,u)_r\|_{L^p(Q_2)}\leq 1. 
\end{align*}

Note that $v_{k+1}=2 v_k-1$, so that $\{v_{k+1}>0\}=\{v_{k}>\frac12\}$ and $\{v_{k+1}\geq 0\}=\{v_{k}\geq\frac12\}$.

Next we want to prove that for the $\delta$ defined in Lemma \ref{lem:osc-reduction}, there exists $k\leq k_0$ such that $\int_{Q_1}(v_{k+1})_+^2 dx dt\leq\delta$, by using the contradiction argument.
So we assume that for all $k\leq k_0$, we have $\int_{Q_1}(v_{k+1})_+^2 dx dt>\delta$.

Since $(v_{k+1})_+\leq 1$, this implies 
\begin{align*}
    \left|A^v_k\right|:=\left|\{v_k\geq\frac12\}\cap Q_1\right|=\left|\{v_{k+1}\geq 0\}\cap Q_1\right|\geq \int_{Q_1}(v_{k+1})_+^2 dx dt\geq \delta.
\end{align*}
Together with \eqref{es:wk-set2}, one can therefore use Lemma \ref{lem:mid-value} for $v_k$ to get
\begin{align*}
    |D^v_k|:=\left|\{0<v_k<\frac12\}\cap ((-2,0)\times B_{\frac54})\right|\geq \beta.
\end{align*}
Since $v_{k+1}=2v_k-1$, one can get $\{v_{k+1}\leq 0\}\supset \{v_{k}\leq 0\}\cup \{0<v_{k}<\frac12\}$ which implies
\begin{align*}
  \left\{ \{v_{k+1}\leq 0\}\cap ((-2,0)\times B_{\frac54})\right\}\supset \left\{\{v_{k}\leq 0\}\cap ((-2,0)\times B_{\frac54}) \right\}\cup D^v_k.
\end{align*}
Then one can get
\begin{align*}
    \left|(-2,0)\times B_{\frac54}\right|\geq&\left|\left\{ \{v_{k_0+1}\leq 0\}\cap ((-2,0)\times B_{\frac54})\right\} \right|\\
    \geq& \left| \left\{\{v_{k_0}\leq 0\}\cap ((-2,0)\times B_{\frac54}) \right\} \right| + |D_{k_0}|\\
    &\geq \left\{\{v_{k_0-1}\leq 0\}\cap ((-2,0)\times B_{\frac54}) \right\} + |D_{k_0-1}| + |D_{k_0}|\\
    &\geq \left\{\{v_{k_0-2}\leq 0\}\cap ((-2,0)\times B_{\frac54}) \right\}  + |D_{k_0-2}| + |D_{k_0-1}| + |D_{k_0}|\\
    &\geq...\geq \left\{\{v_0\leq 0\}\cap ((-2,0)\times B_{\frac54}) \right\} + \sum\limits_{i=0}^{k_0} |D_i|\\
     &= \left\{\{w\geq 0\}\cap ((-2,0)\times B_{\frac54}) \right\} + \sum\limits_{i=0}^{k_0} |D_i|\\
    &\geq \left\{\{w\geq 0\}\cap (Q_1\cup \widetilde Q) \right\} + \sum\limits_{i=0}^{k_0} |D_i|\,\geq\, \frac12 \left|\widetilde Q\right| + (k_0+1)\beta,
\end{align*}
since $(k_0+1)\beta\geq \left|(-2,0)\times B_{\frac54}\right|$, one can get a contradiction, so there exists $k\leq k_0$ such that $\int_{Q_1}(v_{k+1})_+^2 dx dt\leq\delta$. By using Lemma~\ref{lem:osc-reduction}, one can get
\begin{align*}
    v_{k+1}=2^{k+1}(-w-(1-2^{-k-1}))\leq \frac12
\end{align*}
in $Q_{\frac12}$, which implies $w\geq -1+2^{-k-2}\geq -1+2^{-k_0-2}$ in $Q_{\frac12}$, then we choose $\lambda_1=2^{-k_0-2}$ and completes the proof of the second part.
\end{proof}


\subsection{Proof of Theorem \ref{thm:main}} 
Now we turn to the proof of the main result.
\begin{proof}[Proof of Theorem \ref{thm:main}]
 For every compact set
$\mathcal K\Subset (0,T]\times\Omega,$ we choose $R_{\mathcal K}>0$
such that $Q_{4R_{\mathcal K}}(z_0)\Subset (0,T]\times\Omega$ for every $z_0\in\mathcal K$. We prove that $\theta\in C^{\frac{\alpha}{2},\alpha}(\mathcal K)$, we divide the proof into two steps.

\smallskip
\textbf{Step 1:} We claim that for every $z_0\in \mathcal K$ there exist constants $C_\mathcal K>0$ and $0<\alpha<\min\left\{\sigma,-\log_4(1-\lambda_*)\right\}$, independent of $z_0$, such that
\begin{align*}
\operatorname{osc}_{Q_{4^{-k}(2R_\mathcal K)}(z_0)}\theta
\leq C_\mathcal K 4^{-\alpha k}
\end{align*}
for every $k\in\mathbb N$.
Here $\lambda_*$ depend only on $\mathcal K$ and the structural constants.

 Indeed, thanks to Theorem \ref{thm:osc}, one can get that there exists $\lambda_\ast$ such that   \begin{align*}
\operatorname{osc}_{Q_{r/2}}\theta
\le
(1-\lambda_\ast)\Omega_{2r}
&\le
(1-\lambda_\ast)\operatorname{osc}_{Q_{2r}}\theta
+
C_\mathcal K r^\sigma\mathcal F_{2r}\\
&\le
(1-\lambda_\ast)\operatorname{osc}_{Q_{2r}}\theta+C_\mathcal K r^\sigma
\end{align*}
for all $r\leq R_\mathcal K$. Then letting $r=4^{-k}(2R_\mathcal K)$ for $k\in\NN$, we have
\begin{align*}
&\operatorname{osc}_{Q_{4^{-k}(2R_\mathcal K)}(z_0)}\theta \leq (1-\lambda_\ast)\operatorname{osc}_{Q_{4^{-k+1}(2R_\mathcal K)}(z_0)}\theta  + C_\mathcal K 2^{(1-2k)\sigma}\\
    &\leq (1-\lambda_\ast)^2\operatorname{osc}_{Q_{4^{-k+2}(2R_\mathcal K)}(z_0)}\theta + C_\mathcal K(1-\lambda_\ast) 2^{(3-2k)\sigma} + C_\mathcal K 2^{(1-2k)\sigma}\\
     &\leq (1-\lambda_\ast)^3\operatorname{osc}_{Q_{4^{-k+3}(2R_\mathcal K)}(z_0)}\theta + C_\mathcal K(1-\lambda_\ast)^2 2^{(5-2k)\sigma} + C_\mathcal K(1-\lambda_\ast) 2^{(3-2k)\sigma} + C_\mathcal K 2^{(1-2k)\sigma}\\
     &\leq ... \leq (1-\lambda_\ast)^k\operatorname{osc}_{Q_{2R_\mathcal K}(z_0)}\theta + C_\mathcal K\sum\limits_{i=0}^k(1-\lambda_\ast)^i 2^{-\sigma(2k-2i-1)}\\
      &=(1-\lambda_\ast)^k\operatorname{osc}_{Q_{2R_\mathcal K}(z_0)}\theta + C_\mathcal K2^\sigma\sum\limits_{i=0}^k(1-\lambda_\ast)^i 4^{-\sigma(k-i)}.
\end{align*}
Since $\alpha<\min\{\sigma,-\log_4(1-\lambda_\ast)\}$, we have $1-\lambda_\ast< 4^{-\alpha}$, which implies  
\begin{align*}
   \operatorname{osc}_{Q_{4^{-k}(2R_\mathcal K)}(z_0)}\theta \leq 4^{-\alpha k}\operatorname{osc}_{Q_{2R_\mathcal K}(z_0)}\theta + C_\mathcal K2^\sigma 4^{-\alpha k}\sum\limits_{i=0}^k 4^{-(\sigma-\alpha)(k-i)}
    \leq C_\mathcal K 4^{-\alpha k}.
\end{align*}

\textbf{Step 2:} There exists a constant $C_{\mathcal K}>0$ such that
\begin{align*}
|\theta(t,x)-\theta(s,y)|
\le
C_\mathcal K\Bigl(
|x-y|^\alpha
+
|t-s|^{\alpha/2}
\Bigr)
\end{align*}
for all $(t,x),(s,y)\in \mathcal K$.

Choose a compact set $\mathcal K_1$ and a relatively open set
$\mathcal U\subset(0,T]\times\Omega$ such that
$\mathcal K\subset\mathcal U\Subset\mathcal K_1
\Subset(0,T]\times\Omega$.
Applying Step~1 with $\mathcal K_1$ in place of $\mathcal K$ and
using the monotonicity of the oscillation, after enlarging the
constant if necessary, we find $r_{\mathcal K}>0$ and
$C_{\mathcal K}>0$ such that
\begin{align*}
\operatorname{osc}_{Q_r(z)}\theta
\leq C_{\mathcal K}r^\alpha
\end{align*}
for every $z\in\mathcal K_1$ and every
$0<r\leq r_{\mathcal K}$. We choose $r_{\mathcal K}$ sufficiently
small so that all the centers constructed below belong to
$\mathcal K_1$.

Let $\mathcal L$ denote the set of Lebesgue points of $\theta$.
We use the same notation $\theta$ for its precise representative on
$\mathcal L$. Let $X=(t,x)$ and $Y=(s,y)$ belong to
$\mathcal U\cap\mathcal L$, and assume without loss of generality
that $t\geq s$. Set
$d:=\max\{|x-y|,|t-s|^{1/2}\}$.

If $0<d\leq r_{\mathcal K}/2$, define
$\xi:=(x+y)/2$, $\tau:=\min\{t+d^2,T\}$, and
$\widehat Z:=(\tau,\xi)$. By the choice of $r_{\mathcal K}$, we have
$\widehat Z\in\mathcal K_1$. Moreover,
$X,Y\in Q_{2d}(\widehat Z)$. Since $X$ and $Y$ are Lebesgue points,
their precise values lie between the essential infimum and the
essential supremum of $\theta$ on $Q_{2d}(\widehat Z)$. Hence
\begin{align*}
|\theta(X)-\theta(Y)|
\leq
\operatorname{osc}_{Q_{2d}(\widehat Z)}\theta \leq
C_{\mathcal K}(2d)^\alpha \leq
C_{\mathcal K}
\left(
|x-y|^\alpha+|t-s|^{\alpha/2}
\right).
\end{align*}

If $d>r_{\mathcal K}/2$, the same estimate follows, after enlarging
$C_{\mathcal K}$, from the uniform boundedness of $\theta$ on
$\mathcal K_1$. Indeed, since
$d^\alpha\geq(r_{\mathcal K}/2)^\alpha$, the uniform bound can be
absorbed into $C_{\mathcal K}d^\alpha$. Therefore, the above
H\"older estimate holds for all
$X,Y\in\mathcal U\cap\mathcal L$.

Since $\mathcal L$ has full measure in $(0,T)\times\Omega$, the set
$\mathcal U\cap\mathcal L$ is dense in $\mathcal U$ with respect to
the relative topology of $(0,T]\times\Omega$. The above estimate
therefore yields a unique H\"older-continuous extension to
$\overline{\mathcal U}$, which is still denoted by $\theta$. In
particular, $\theta\in C^{\alpha/2,\alpha}(\mathcal K).$

If $\mathcal K$ intersects the terminal surface $t=T$, the values of
$\theta$ there are defined by continuous extension from $t<T$. Since $\mathcal K\Subset (0,T]\times\Omega$ is arbitrary, one can get $\theta\in C_{loc}^{\frac{\alpha}{2},\alpha}(  [\tau,T]\times \Omega)$ for any $\tau>0$, which completes the proof of Theorem \ref{thm:main}.  
\end{proof}


\section{An application of towards Heat conducting Potential Flow} \label{sec:appl}


 We now apply Theorem~\ref{thm:main} to the heat-conducting potential-flow system introduced in \eqref{nsf_toy} with the pressure law \eqref{pres:toy}. Our goal is to establish the global-in-time existence and uniqueness of regular solutions to this system.
We prove Theorem~\ref{thm:global_toy}, building on the qualitative properties established in Theorem~\ref{thm:main}. To this end, we follow the strategy outlined after the statement of Theorem~\ref{thm:global_toy}. The first step is to establish a local-in-time existence result. Although this is based on rather classical techniques, we devote Subsection~\ref{ssec:local-toy}, and in particular Proposition~\ref{prop:local}, to its proof. In the subsequent subsection, we derive the global-in-time a priori estimates and complete the proof of the theorem.
\subsection{Global in time estimates}
We test the momentum equation $\Delta u = \nabla p$ by $u$. Integrating by parts over the torus $\mathbb{T}^d$ eliminates boundary terms by periodicity, yielding the mechanical identity:
\begin{equation}\label{eq:mech_work}
-\int_{\mathbb{T}^d} |\nabla u|^2 dx = -\int_{\mathbb{T}^d} p \,\text{div}\,u \, dx \implies \int_{\mathbb{T}^d} |\text{div}\,u|^2 dx = \int_{\mathbb{T}^d} p \,\text{div}\,u \, dx.
\end{equation}

Next, multiplying the continuity equation by $\frac{\gamma}{\gamma-1}\rho^{\gamma-1}$ and integrating over $\mathbb{T}^d$ gives the potential energy evolution:
\begin{equation}\label{eq:pot_energy}
\frac{d}{dt} \int_{\mathbb{T}^d} \frac{1}{\gamma-1}\rho^\gamma dx = -\int_{\mathbb{T}^d} \rho^\gamma \,\text{div}\,u \, dx.
\end{equation}

Integrating the energy equation over $\mathbb{T}^d$, all pure divergence flux terms vanish due to the periodic boundary conditions:
\begin{equation}\label{eq:int_energy}
\frac{d}{dt} \int_{\mathbb{T}^d} (c_v\rho\theta + \theta^4) dx + \int_{\mathbb{T}^d} \left(\rho\theta + \frac{1}{3}\theta^4\right)\text{div}\,u \, dx = \int_{\mathbb{T}^d} |\text{div}\,u|^2 dx.
\end{equation}

Summing \eqref{eq:pot_energy} and \eqref{eq:int_energy}, and substituting the definition of the pressure law $p(\rho, \theta) = \rho^\gamma + \rho \theta + \frac13 \theta^4$, we obtain:
\begin{equation}
\frac{d}{dt} \int_{\mathbb{T}^d} \lr{\frac{1}{\gamma -1} \rho^\gamma + c_v\rho \theta + \theta^4} dx = -\int_{\mathbb{T}^d} p(\rho, \theta)\,\text{div}\,u \, dx + \int_{\mathbb{T}^d} |\text{div}\,u|^2 dx.
\end{equation}
Applying the mechanical work balance \eqref{eq:mech_work} yields a net cancellation on the right-hand side, establishing $\displaystyle \frac{d}{dt}E(t) = 0$, and eventually 
\begin{equation}\label{def:E0}
    \int_{\mathbb{T}^d} \lr{\frac{1}{\gamma -1} \rho^\gamma + c_v \rho \theta + \theta^4}(t,\cdot)dx = \int_{\mathbb{T}^d} \lr{\frac{1}{\gamma -1} \rho_0^\gamma + c_v \rho_0 \theta_0 + \theta_0^4} dx:=E_0,
\end{equation}
for a.e. $t \in (0,T)$. It follows that for some constants $a_1, a_2 >0$ there holds
\begin{equation}\label{rel:E0-pre}
    a_1 E_0 \leq \int_{\mathbb{T}^d} p(\rho,\theta) dx \leq a_2 E_0.
\end{equation}
Note that we have the structure of the space domain, it is the torus relies into the following simple form of the effective viscous flux
\begin{equation}
    {\rm div\,}u=p(\rho,\theta) -\{
    p(\rho,\theta)\},
\end{equation}
where, recalling the notation, $\displaystyle \{ p(\rho,\theta)\}= \frac{1}{|\TT^d|} \int_{\TT^d} p(\rho,\theta) dx$.

In what follows, we want to show the point-wise bounds for the temperature and density. It is clear these functions are nonnegative. In this part we assume the solutions are smooth enough. Rigorously, we can analyze on the time interval where the solutions are smooth, and then we prolong these properties as in the standard approach.
\tcrb{}
\subsubsection{Maximum and minimum principle for density and temperature}

\begin{lemma}[$L^\infty$ bound for the density]\label{UB-den}
    Let $\rho_0 \in L^\infty(\TT^d)$ and $E_0$ satisfies \eqref{def:E0}-\eqref{rel:E0-pre}. Then
    \begin{equation}
        \rho(x,t) \leq \rho_{\mathrm{MAX}} \quad \text{with} \quad \rho_{\mathrm{MAX}} = \max\left\{ \|\rho_0\|_{L^\infty}, \sqrt[\gamma]{a_2 E_0} \right\} \text{ for a.e. } (t,x)\in (0,T) \times \mathbb{T}^d. 
    \end{equation}
\end{lemma}

\begin{proof}
    We establish the upper bound on the density using a classical argument.  
    First, recall that the density satisfies the following evolution equation:
    \begin{equation}
        \partial_t \rho + u \nabla \rho + \rho \bigl( p(\rho,\theta) - \{ p(\rho,\theta) \} \bigr) = 0.
    \end{equation}
    The idea is to choose $\rho_{\mathrm{MAX}}$ sufficiently large so that it dominates the pressure term.  
    Specifically, we require
    \begin{equation}
        \rho_{\mathrm{MAX}}^\gamma > \{ p(\rho,\theta) \}, \quad \text{namely} \quad \rho_{\mathrm{MAX}}^\gamma = a_2 E_0.
    \end{equation}
    Now suppose, for contradiction, that $\rho$ exceeds this value at some interior point.  
    At a local maximum of $\rho$, the time derivative $\partial_t \rho$ would necessarily be non-negative.  
    However, the structure of the equation forces $\partial_t \rho$ to be strictly negative at such a point, which is impossible.  
    Hence no such local maximum can exist.

    Since the initial density is bounded by $\|\rho_0\|_\infty$, the above reasoning shows that the density cannot exceed $\rho_{\mathrm{MAX}}$ at any later time.  
    Consequently, we obtain
    \begin{equation}
        \rho(x,t) \leq \rho_{\mathrm{MAX}}.
    \end{equation}
    In the event that the previous choice of $\rho_{\mathrm{MAX}}$ is invalid, we simply take $\rho_{\mathrm{MAX}} = \|\rho_0\|_\infty$, and the desired inequality holds trivially.  
    This completes the proof.
\end{proof}

\renewcommand{\proofname}{Proof}

\begin{lemma}[$L^\infty$ bound for the temperature]\label{UB-temp}
    Let $\theta_0 \in L^{\infty}(\TT^d)$ and $E_0$ satisfies \eqref{def:E0} and \eqref{rel:E0-pre}. Then the temperature stays bounded, i.e., 
    \begin{align}
        \theta \leq \theta_{\mathrm{MAX}} \qquad\text{ with }  \qquad \theta_{\mathrm{MAX}}^4 = \max\left\{ 4a_2E_0,  \|\theta_0\|_{L^\infty}^4 \right\}.
    \end{align}
\end{lemma}

\begin{proof}
    We begin by recalling the evolution equation for the temperature. Using the previously established notation, we have
    \begin{equation*}
        \partial_t(c_v\rho\theta+\theta^4)+\operatorname{div}\,\left((c_v\rho\theta+\theta^4)u\right) + \bigl(\{p(\rho,\theta)\}- \rho^\gamma\bigr)\bigl(p(\rho,\theta) -\{ p(\rho,\theta)\}\bigr) = \operatorname{div} (\kappa(\theta)\nabla \theta).
    \end{equation*}
    The goal is to show that $\theta$ cannot exceed a certain threshold $\theta_{\mathrm{MAX}}$.  
    To this end, we choose $\theta_{\mathrm{MAX}}$ such that the following inequality holds:
    \begin{equation*}
        \frac14 \theta_{\mathrm{MAX}}^4 - \{ p(\rho,\theta)\} \geq 0.
    \end{equation*}
    We then integrate the equation over the set $\{ \theta > \theta_{\mathrm{MAX}} \}$.  
    At this point, we recall a consequence of Sard's theorem: the level set of a smooth function is regular for almost all values of the function.  
    Therefore, we consider a sequence $\theta_{\mathrm{MAX}}^n \to \theta_{\mathrm{MAX}}^+$ to ensure that each level set $\{\theta = \theta_{\mathrm{MAX}}^n\}$ is a regular submanifold.
    A direct computation yields the following identities:
    \begin{align*}
        &\int_{\{\theta>\theta_{\mathrm{MAX}}\}} 
        \partial_t(c_v\rho \theta + \theta^4)\,dx +
        \int_{\{\theta > \theta_{\mathrm{MAX}}\}} 
        \operatorname{div}\bigl((c_v\rho \theta +\theta^4)u \bigr)\,dx \\ 
        &\quad =
        \int_{\{\theta>\theta_{\mathrm{MAX}}\}} 
        \bigl( c_v\rho \partial_t(\theta -\theta_{\mathrm{MAX}})_+ + \partial_t(\theta^4-\theta_{\mathrm{MAX}}^4)_+ \bigr) dx \\
        & \qquad + 
        \int_{\{\theta > \theta_{\mathrm{MAX}}\}} 
        c_v \rho u \nabla (\theta-\theta_{\mathrm{MAX}})_+   dx \\
        &\qquad + \int_{\{\theta>\theta_{\mathrm{MAX}}\}} \operatorname{div}\bigl( (\theta^4-\theta_{\mathrm{MAX}}^4)_+ u \bigr)dx +
        \theta_{\mathrm{MAX}}^4 \operatorname{div} u \, dx.
    \end{align*}

    Using the continuity equation, we obtain the following energy estimate:
    \begin{align}
        &\frac{d}{dt}\int_{\mathbb{T}^d} \bigl( c_v\rho (\theta -\theta_{\mathrm{MAX}})_+ + (\theta^4 -\theta_{\mathrm{MAX}}^4)_+ \bigr)dx \nonumber \\
        &\quad + \int_{\{ \theta > \theta_{\mathrm{MAX}}\}} \bigl( \{p(\rho,\theta)\} + \theta^4_{\mathrm{MAX}} - \rho^\gamma \bigr)\bigl( p(\rho,\theta) -\{ p(\rho,\theta)\} \bigr) dx 
        \nonumber \\
        &= \int_{\partial\{ \theta > \theta_{\mathrm{MAX}} \}} \kappa(\theta) \nabla \theta \cdot n \, d\sigma.\nonumber
    \end{align}

    To conclude the desired bound, we need to ensure that the last term on the left-hand side is non-negative. This leads to the following conditions:
    \begin{equation*}
        \frac14 \theta_{\mathrm{MAX}}^4 - \{ p(\rho,\theta)\} \geq 0 \quad \text{and} \quad
        \theta^4_{\mathrm{MAX}} \geq (a_2-a_1)E_0.
    \end{equation*}
    Thanks to the smooth construction of the solutions, these conditions guarantee the correct sign of the two last terms in the previous expression. Consequently, we deduce that
    \begin{equation*}
        \int_{\mathbb{T}^d} \bigl(c_v \rho (\theta -\theta_{\mathrm{MAX}})_+ + (\theta^4 -\theta_{\mathrm{MAX}}^4)_+ \bigr)dx = 0.
    \end{equation*}
    This integral identity implies that $\theta \leq \theta_{\mathrm{MAX}}$ almost everywhere, provided we also have $\theta_{\mathrm{MAX}} \geq \|\theta_0\|_\infty$ initially.

    Finally, collecting all the requirements, we arrive at the explicit bound
    \begin{equation}
        \theta_{\mathrm{MAX}}^4 = \max\left\{ 4a_2E_0, (a_2-a_1)E_0, \|\theta_0\|_\infty^4\right\}.
    \end{equation}
    Note the second term is always smaller than the first one.
    This completes the proof.
\end{proof}
\begin{remark}
    As a consequence of the above lemmas (Lemma \ref{UB-den} and Lemma \ref{UB-temp}) we have 
    \begin{align*}
        \|\operatorname{div} u\|_{L^\infty_{t,x}} \leq  \|p(\rho,\theta)-\{p(\rho,\theta)\}\|_{L^\infty_{t,x}} \leq C\lr{\| \rho \|_{L^\infty_{t,x}}, \| \theta \|_{L^\infty_{t,x}}}.
    \end{align*}
    Furthermore, we observe that if the initial density $\rho_0$ were strictly bounded away from zero, it would be possible to deduce a strictly positive uniform lower bound for the temperature $\theta$. However, since our framework allows initial vacuum states for density, we only retain the standard Lagrangian lower bound for the density:
    \begin{align*}
        \rho(t,x) \geq \rho_0(x) \exp\lr{-\int_0^t \|\operatorname{div} u(\cdot, s)\|_{L^\infty_x} ds}.
    \end{align*}
\end{remark}

\bigskip
 \begin{lemma}[Strict positivity of the temperature]
	\label{lem:theta_positive}
	Let $(\rho,\theta,u)$ be a smooth solution to \eqref{nsf_toy} on
	$[0,T]\times\TT^d$ and assume that
	$\theta_0(x)\geq \underline{\theta}_0>0.$
	Then
	\[
	\theta(t,x)\geq \underline{\theta}_0 e^{-Ct}
	\qquad \mbox{for all } (t,x)\in [0,T]\times\TT^d,
	\]
	where
	\begin{equation} \label{def:C}
C= \max\left\{
\frac{1}{c_v},\frac{1}{3}
\right\} \|\mathrm{div}\,u\|_{L^\infty((0,T)\times\TT^d)}.
	\end{equation}
\end{lemma}

\begin{proof}
	Using the continuity equation, the temperature equation can be written in the non-conservative form
	\begin{equation}
		(c_v\rho+4\theta^3)\big(\partial_t\theta+u\cdot\nabla\theta\big)
		+\Big(\rho\theta+\frac43\theta^4\Big)\mathrm{div}\,u
		=
		|\mathrm{div}\,u|^2
		+\mathrm{div}\big(\kappa(\theta)\nabla\theta\big).
		\label{eq:theta_noncons}
	\end{equation}
	
	Let $\eta(t)=\underline{\theta}_0 e^{-Ct}, $
	where $C$ will be determined later and define
	\[
	w=(\eta-\theta)_+.
	\]
	Furthermore, introduce
	\[
	Q(\theta,\eta)
	=
	{\bf 1}_{\{\theta<\eta\}}
	\int_\theta^\eta 4s^3(\eta-s)\,ds .
	\]
	A direct computation shows that
	\[
	\partial_\theta Q=-4\theta^3(\eta-\theta)_+,
	\qquad
	\partial_\eta Q=(\eta^4-\theta^4)_+.
	\]
	
	We test \eqref{eq:theta_noncons} by $-w$ and integrate over $\TT^d$. Moreover, we observe
	\[
	-c_v\rho(\partial_t\theta+u\cdot\nabla\theta)w
	=
	c_v\rho(\partial_t+u\cdot\nabla) \frac12(\eta-\theta)_+^2
	-c_v\rho\eta' w,
	\]
	and the continuity equation yields
	\[
	\frac{c_v}2\int_{\TT^d}
	\rho(\partial_t+u\cdot\nabla) (\eta-\theta)_+^2\,dx
	=
	\frac{c_v}2\frac{d}{dt}
	\int_{\TT^d}\rho(\eta-\theta)_+^2 \,dx.
	\]
	For the radiative contribution, the chain rule gives
	\[
	-4\theta^3(\partial_t\theta+u\cdot\nabla\theta)w
	=
	(\partial_t+u\cdot\nabla)Q
	-\eta'(\eta^4-\theta^4)_+.
	\]
	Integrating by parts,
	\[
	\int_{\TT^d}u\cdot\nabla Q\,dx
	=
	-\int_{\TT^d}Q\,\mathrm{div}\,u\,dx
	\geq
	-\|\mathrm{div}\,u\|_\infty
	\int_{\TT^d}Q\,dx .
	\]
	Consequently, with
	\[
	Y(t)
	=
	\int_{\TT^d}
	\Big(
	\frac{c_v}2\rho w^2+Q(\theta,\eta)
	\Big)\,dx,
	\]
	we obtain
	\begin{align}
		\frac{d}{dt}Y(t)
		&-
		\|\mathrm{div}\,u\|_\infty Y(t)
		+
		(-\eta')
		\int_{\TT^d}
		\Big[
		c_v\rho(\eta-\theta)_+
		+
		(\eta^4-\theta^4)_+
		\Big]dx
		\nonumber
		\\
		&\leq
		\int_{\TT^d}
		\Big(\rho\theta+\frac43\theta^4\Big)
		|\mathrm{div}\,u|
		(\eta-\theta)_+\,dx .
		\label{eq:before_absorption}
	\end{align}
	
	The dissipation term
	\[
	-\int_{\TT^d}
	|\mathrm{div}\,u|^2(\eta-\theta)_+\,dx
	\]
	is non-positive and was discarded. Moreover,
	\[
	-\int_{\TT^d}
	\mathrm{div}\big(\kappa(\theta)\nabla\theta\big)(-w)\,dx
	=
	\int_{\{\theta<\eta\}}
	\kappa(\theta)|\nabla\theta|^2\,dx
	\geq0,
	\]
	and can also be discarded.
	
	It remains to estimate the source term. On the set
	$\{\theta<\eta\}$, we have 
\begin{align*}
\rho\theta+\frac{4}{3}\theta^4
=
\theta\left(\rho+\frac{4}{3}\theta^3\right)\leq
\eta\left(\rho+\frac{4}{3}\theta^3\right)
=
\eta\left(
\frac{1}{c_v}c_v\rho+\frac{1}{3}\,4\theta^3
\right)\leq
\max\left\{
\frac{1}{c_v},\frac{1}{3}
\right\}
\eta\left(c_v\rho+4\theta^3\right).
\end{align*}

	Hence
	\[
	\Big(\rho\theta+\frac43\theta^4\Big)
	|\mathrm{div}\,u|
	(\eta-\theta)_+
	\leq
	C\eta
	\Big[
	c_v\rho(\eta-\theta)_+
	+
	4\theta^3(\eta-\theta)_+
	\Big],
	\]
	with $C = \max\left\{
\frac{1}{c_v},\frac{1}{3}
\right\}\|\mathrm{div}\,u\|_{L^\infty((0,T)\times\TT^d)}$.
	Therefore we have
	\begin{align*}
			\frac{d}{dt}Y(t) &-
		\|\mathrm{div}\,u\|_\infty Y(t) + (-\eta') \int_{\TT^d}
		\Big[ c_v\rho(\eta-\theta)_+ + (\eta^4-\theta^4)_+\Big]dx \nonumber\\
		&\leq C\eta \int_{\TT^d}\Big[ c_v\rho(\eta-\theta)_+
	+ 4\theta^3(\eta-\theta)_+ \Big]dx.
	\end{align*}
	Since
	$-\eta'(t)=C\eta(t),$
	the first term is absorbed directly by the first positive term in
	\eqref{eq:before_absorption}. For the second term we use the convexity of
	$s\mapsto s^4$:
	\[
	(\eta^4-\theta^4)_+
	=
	[(\eta-\theta)
	(\eta^3+\eta^2\theta+\eta\theta^2+\theta^3)]_{+}
	\geq
	4\theta^3(\eta-\theta)_+.
	\]
	Thus the second source contribution is absorbed by
	\[
	(-\eta')
	\int_{\TT^d}
	(\eta^4-\theta^4)_+\,dx .
	\]
	
	After these absorptions we arrive at
	\[
	\frac{d}{dt}Y(t)
	\leq
	\|\mathrm{div}\,u\|_\infty Y(t).
	\]
	
	Since $\theta_0\geq \underline{\theta}_0=\eta(0)$, we have
	$Y(0)=0.$
	Therefore, by Grönwall's lemma,
	$Y(t)\equiv 0.$
	Finally, observe that
	\[
	Q(\theta,\eta)>0
	\qquad\text{whenever}\qquad
	0\leq\theta<\eta.
	\]
	In particular,
	$Q(0,\eta)=\frac15\eta^5>0,$
	so the functional detects even vacuum regions. 
    Moreover,
\[
Q(\theta,\eta)=0
\quad\Longleftrightarrow\quad
\theta\ge \eta.
\]
Hence
	\[
	\theta(t,x)\geq \eta(t)
	=
	\underline{\theta}_0 e^{-Ct}
\mbox{ \ \ for all $(t,x)\in[0,T]\times\TT^d$.}
	\]
\end{proof}
Let us conclude this subsection by stating an immediate consequence of the above lemmas, which provides a global estimate for $u$.
\begin{lemma}[Estimate for $u$]\label{lemma:u}
Assume that $\rho,\theta\in L^\infty((0,T)\times\Omega) $.
Then, we have
\[
\nabla u\in L^\infty(0,T;BMO(\Omega)).
\]
In particular, for every $1<q<\infty$,
$u\in L^\infty(0,T;W^{1,q}(\Omega)).$
\end{lemma}

\begin{proof}
Since $\rho,\theta\in L^\infty$, the equation of state yields
$p\in L^\infty((0,T)\times\Omega).$
This yields 
\[
\|\text{\rm div}\,u(t)\|_{L^\infty(\Omega)}
\le C\|p(t)\|_{L^\infty(\Omega)}.
\]
Since $u=\nabla\phi$, from
$\nabla \operatorname{div}u=\nabla p$
we obtain
$\Delta\phi=p-\{p\}.$
Hence
\[
\nabla u=\nabla^2\Delta^{-1}(p-\{p\}).
\]
Since $\nabla^2\Delta^{-1}$ is a Calderón--Zygmund operator on
$\mathbb T^d$, it maps $L^\infty$ continuously into $BMO$. Therefore,
$\|\nabla u(t)\|_{BMO}
\le C\|p(t)\|_{L^\infty}.$
As $p\in L^\infty((0,T)\times\mathbb T^d)$, it follows that
$\nabla u\in L^\infty(0,T;BMO(\mathbb T^d)).$
Since $BMO(\Omega)\hookrightarrow L^q(\Omega)$ for every
$1<q<\infty$ on $\mathbb{T}^d$, it follows that
\[
u\in L^\infty(0,T;W^{1,q}(\Omega)),
\qquad 1<q<\infty.
\]
\end{proof}

\subsection{Regularity of the density and temperature}
As discussed earlier, the key ingredient of the present work is the Hölder regularity obtained through the De Giorgi method. We are now in a position to apply the main oscillation-decay theorem, to obtain the H\"older regularity for the temperature. The below result is the synthesis of Theorem \ref{thm:main} and the local in time existence for system \eqref{nsf_toy} proved in Subsection \ref{ssec:local-toy}.

\begin{lemma}[H\"older regularity for the temperature]\label{lem:theta-holder}
Let $(\rho,u,\theta)$ be a solution of \eqref{nsf_toy}.
Assume that
\[
\rho,\theta\in L^\infty((0,T)\times\mathbb{T}^d),
\qquad 
\nabla u\in L^\infty(0,T;BMO(\mathbb{T}^d)).
\]
Then there exists $\sigma\in(0,1)$, depending only on the above
a priori bounds, such that
\[
\theta\in C^\sigma([0,T]\times\mathbb{T}^d).
\qquad \mbox{ Moreover, \ \ \ \ \  }
\|\theta\|_{C^\sigma([0,T]\times\mathbb{ T}^d)}
\le C,
\]
where $C$ depends only on the corresponding a priori bounds.
\end{lemma}

\begin{proof}
Since $\rho,\theta\in L^\infty$ and
$\nabla u\in L^\infty(0,T;BMO(\mathbb{T}^d))$, we have
\[
u\in L^\infty(0,T;W^{1,q}(\mathbb{T}^d))
\qquad\forall\,1<q<\infty.
\]
Hence all assumptions \eqref{hyp:main} of Theorem~\ref{thm:main} (Equivalently Theorem~\ref{thm:osc}) are satisfied. 
The only issue is the initial interval, and at this part, we use the local in time existence result proved in Proposition \ref{prop:local}.
The
oscillation decay estimate of Theorem~\ref{thm:osc} therefore applies
at every scale, and a standard Campanato iteration yields
\[
\theta\in C^\sigma([0,T]\times\mathbb{T}^d).
\]
\end{proof}

Once we have established the H\"older continuity of the temperature, we proceed to show the same property for the density. This simplifies the analysis and grants the full regularity required for the system.

\begin{lemma}[H\"older regularity for the density]\label{lem:del-holder}
    Let $\theta \in C(0,T;C^\sigma(\TT^d))$,
    then $\rho \in C(0,T;C^\alpha(\TT^d))$
    with $\alpha < \sigma$. Consequently, there exists $\beta >0$ such that $\rho \in C^\beta( [0,T]\times \TT^d))$.
\end{lemma}

\begin{proof}
    The proof is based on the computation of the Slobodeckij norm of the density.  
    We begin with the following evolution equation for the difference of density at two points $x$ and $y$:
    \begin{equation*}
        \partial_t(\rho_x - \rho_y) + 
        \operatorname{div}_x(u_x(\rho_x - \rho_y)) + 
        \operatorname{div}_y(u_y(\rho_x - \rho_y))
        + \operatorname{div}_x u_x \, \rho_y - 
        \operatorname{div}_y u_y \, \rho_x = 0.
    \end{equation*}
    This can be rewritten in the equivalent form
    \begin{align*}
         \partial_t(\rho_x - \rho_y) + 
        \operatorname{div}_x(u_x(\rho_x - \rho_y)) + 
        \operatorname{div}_y(u_y(\rho_x - \rho_y)) \\
        - \operatorname{div}_x u_x \, (\rho_x-\rho_y) - (\operatorname{div}_x u_x -\operatorname{div}_y u_y) \rho_x = 0.
    \end{align*}
    By the definition of the effective viscous flux, this further simplifies to
    \begin{align*}
          \partial_t(\rho_x - \rho_y) + 
        \operatorname{div}_x(u_x(\rho_x - \rho_y)) + 
        \operatorname{div}_y(u_y(\rho_x - \rho_y)) \\
        - (p_x - \{p\})(\rho_x-\rho_y) - (p_x - p_y) \rho_x = 0.
    \end{align*}

    We now test this equation with the function $\dfrac{|\rho_x-\rho_y|^{p-2}(\rho_x-\rho_y)}{|x-y|^{d+s p}}$.  
    After standard manipulations, we obtain the following inequality:
    \begin{align*}
             & \frac1p \frac{d}{dt} \iint_{\TT^{2d}} 
            \frac{|\rho_x -\rho_y|^p}{|x-y|^{d+s p}}\,dx\,dy  \\
           &\leq C_p(\|\operatorname{div} u\|_\infty + \| p\|_\infty) \iint_{\TT^{2d}} 
            \frac{|\rho_x -\rho_y|^p}{|x-y|^{d+s p}}\,dx\,dy \\
            &\quad +C_p \iint_{\TT^{2d}} \frac{(u_x-u_y)\cdot(x-y)}{|x-y|^{2}}  
            \frac{|\rho_x -\rho_y|^p}{|x-y|^{d+s p}}\,dx\,dy \\
           &\quad + C_p \iint_{\TT^{2d}} (\theta_x-\theta_y) 
            \frac{|\rho_x -\rho_y|^{p-1}}{|x-y|^{d+s p}}\,dx\,dy.
    \end{align*}
Next goal is to estimate each term, separately. First, by the $L^\infty$--BMO inequality (see Kozono and Taniuchi \cite[Theorem 1]{Kozono}), we have
    \begin{equation*}
        \frac{(u_x-u_y)\cdot(x-y)}{|x-y|^{2}} \leq \|\nabla u\|_\infty \leq 
        C\bigl(1+\|\nabla u \|_{\mathrm{BMO}}(1+|\ln^+\|\nabla u\|_{C^{\alpha'}}|)\bigr).
    \end{equation*}
    We keep in mind the imbedding  that for $\alpha ' =s -\frac{d}{p}$, we have $W^s_p(\TT^d) \subset C^{\alpha'}(\TT^d)$, and then
    \begin{equation*}
        \|\nabla u\|_{C^{\alpha'}} \leq C\bigl(
        \|\rho\|_{W^{s}_p} + \|\theta\|_{C^\sigma}\bigr).
    \end{equation*}

    The last term requires careful treatment. Since $\theta \in C^\sigma$, we estimate
    \begin{align}
        &\iint_{\TT^{2d}} (\theta_x-\theta_y) 
        \frac{|\rho_x -\rho_y|^{p-1}}{|x-y|^{d+s p}}\,dx\,dy \leq C \iint_{\TT^{2d}} 
        \frac{|\rho_x -\rho_y|^{p-1}}{|x-y|^{d+s p -\sigma}}\,dx\,dy \nonumber
        \\
        &\leq 
        C \left( \iint_{\TT^{2d}} 
        \frac{|\rho_x -\rho_y|^{p}}{|x-y|^{d+sp}}\,dx\,dy \right)^{\frac{p-1}{p}}
        \left( \iint_{\TT^{2d}} 
        \frac{1}{|x-y|^{d+s p-\sigma p}}\,dx\,dy \right)^{1/p} \nonumber \\
        &\leq C\iint_{\TT^{2d}} 
        \frac{|\rho_x -\rho_y|^{p}}{|x-y|^{d+s p}}\,dx\,dy + C.\nonumber
    \end{align}
    The last inequality follows from Young's inequality, and it is valid provided $\sigma > s$.  
    Note that there is no restriction on $p$ other than finiteness. The constants may depend on $p$ in a possibly bad way, but they remain finite throughout the analysis.

    Following \cite[Chapter VII]{AdamsFournier}), we define the Slobodeckij semi-norm $\dot W^s_p(\TT^d)$ :
    \begin{equation}
        X_{s,p} :=\|\rho\|_{\dot W^s_p(\TT^d)}^p =\iint_{\TT^{2d}} 
        \frac{|\rho_x -\rho_y|^{p}}{|x-y|^{d+s p}}\,dx\,dy.
    \end{equation}
    Then, combining the estimates above, we arrive at the differential inequality
    \begin{equation}
        \dot X_{s,p} \leq C_p X_{s,p} \bigl(1+|\ln (1+X_{s,p})|\bigr) + C_p.
    \end{equation}
    This inequality implies that $X_{\alpha,p}$ grows at most doubly exponentially in time. More precisely, we obtain
    \begin{equation}
        X_{s,p} \leq C_p e^{e^{ct}} \qquad \forall \; 1<p<\infty.
    \end{equation}
    Nevertheless, this quantity remains bounded for all finite $t$. Hence we have obtained the desired regularity in the high-order Sobolev–Slobodeckij spaces.  
    Consequently, the density is H\"older continuous in space.

The last part is the regularity in time. Due to the continuity equation, we immediately deduce that $\partial_t \rho \in L^p(0,T;W^{s-1}_p(\TT^d))$. Then we find via the complex interpolation (See, Triebel \cite[Chapter 2.4 ]{Triebel}) that 
\begin{equation}
    \left(L^p(0,T;W^s_p);W^1_p(0,T;W^{s-1}_p)\right)_k \subset W^k_p(0,T;W^{s-k}_p) \subset W^{\alpha'}_p(0,T;C(\TT^d)).
\end{equation}
If $s-k>\frac{d}{p}$, then $\alpha'=s-\frac{d}{p}>k$. Therefore, we can taking $p$ so large that $kp >1$, and then we find $\alpha''> 0$ that $\rho \in C^{\alpha''}(0,T;C(\TT^d))$.
    This completes the proof.
\end{proof}

\subsubsection{Full regularity estimates for temperature}

In what follows we establish the full regularity of temperature in the $W^{2,1}_p(\mathbb{T}^d \times (0,T))$ class for sufficiently large $p$.  
We emphasize that the H\"older regularity depends on the length of the time interval; consequently, the estimates for the highest derivatives will also depend on $T$.
\begin{lemma}[Full parabolic regularity of the temperature]
\label{lem:W21p}
Let $p>d+2$ and Lemma \ref{UB-den}--Lemma \ref{lem:del-holder} holds.
Then $\theta\in W^{2,1}_p(\mathbb T^d\times(0,T))$, and there is a constant
$C(T)<\infty$, depending only on $d,p,\sigma$, the ellipticity constants, the data,
and the H\"older norm $\|\theta\|_{L^\infty_tC^\sigma_x}$, such that
\[
\|\theta\|_{W^{2,1}_p(\mathbb T^d\times(0,T))}\le C(T).
\]
\end{lemma}

\medskip
\noindent
\textit{Proof.} The proof proceeds in the following steps; see Denk \cite{Denk} and Denk, Hieber, and Prüss \cite{DEHIEPR} for details of the technique and further references.

\smallskip 

\noindent
\textbf{Freezing coefficient.} Let $(x_0,t_0) \in \mathbb{T}^d \times [0,T]$ be an arbitrary point.  
We rewrite the heat equation in the following localized form:
\begin{align*}
   & \bigl(c_v\rho(x_0,t_0) + 4\theta^3(x_0,t_0)\bigr)\partial_t\theta
   - \kappa\bigl(\theta(x_0,t_0)\bigr)\Delta \theta 
   \\
   &= -c_v \rho u \cdot \nabla \theta - \operatorname{div}(\theta^4 u) 
   - \bigl(\{p(\rho,\theta)\} - \rho^\gamma\bigr)\bigl(p(\rho,\theta) - \{p(\rho,\theta)\}\bigr) 
   \\
   &\quad + \bigl(c_v(\rho(x_0,t_0) - \rho )+ 4\theta^3(x_0,t_0) - 4\theta^3\bigr)\partial_t\theta
   \\
  & \quad + \operatorname{div}\bigl((\kappa(\theta) - \kappa(\theta(x_0,t_0)))\nabla \theta\bigr).
\end{align*}
The critical quantity that we need to control in the forthcoming estimates is
\begin{equation}
    \|\nabla \theta\|_{2p}^2.
\end{equation}
The challenge is to bound this quantity using the available H\"older regularity of $\theta$.

\medskip

At this stage, we recall that $\theta \in L^\infty(0,T;C^\sigma(\mathbb{T}^d))$ for some $\sigma \in (0,1)$.  
Our objective is to choose $p$ sufficiently large so that, for some exponent $a > \frac12$, the following interpolation inequality holds:
\begin{equation}\label{est:2p}
    \|\nabla \theta\|_{2p} \leq C_p \|\nabla^2 \theta\|_p^{1-a} \|\theta\|^a_{W^\sigma_p}.
\end{equation}
Indeed, by the Sobolev embedding theorem, $W^{d/(2p)}_p \subset L^{2p}$.  
Therefore, the above inequality follows from the more general interpolation estimate
\begin{equation}
    \|\theta\|_{W^{1+\frac{d}{2p}}_p} \leq C_p \|\theta\|_{W^2_p}^{1-a} \|\theta\|^a_{W^\sigma_p}.
\end{equation}
A simple algebraic computation shows that the exponents must satisfy
\begin{equation}
    1 + \frac{d}{2p} = 2(1-a) + a\sigma.
\end{equation}
Solving for $a$, we obtain
\begin{equation}
    a = \frac{1 - \frac{d}{2p}}{2 - \sigma}, \qquad \text{so that} \qquad a > \frac12 \quad \text{provided} \quad \sigma > \frac{d}{p}.
\end{equation}
Thus, by choosing $p$ large enough, we can ensure $\sigma > d/p$, which guarantees $a > 1/2$.  
This is the crucial observation that allows us to close the estimates.

\medskip

\noindent\textbf{Localization procedure.}  The classical approach for the linear problems with variable coefficients require localization procedure.
For each fixed $(x_0,t_0)$, we introduce a smooth cutoff function 
\[
\pi_{x_0,t_0}(x,t) = \pi = 
\begin{cases}
1, & |x-x_0| + (t_0 - t)^2 < \tfrac12 \lambda, \\[4pt]
0, & |x-x_0| + (t_0 - t)^2 > \lambda,
\end{cases}
\]
where $\lambda > 0$ is a small parameter to be chosen later.  
Multiplying the heat equation by $\pi$ and rearranging terms, we obtain
\begin{equation}
    \bigl(c_v\rho(x_0,t_0) + 4\theta^3(x_0,t_0)\bigr)\partial_t(\pi\theta)
    - \kappa\bigl(\theta(x_0,t_0)\bigr) \Delta (\pi\theta) = F,
\end{equation}
where the right-hand side $F$ collects all lower-order and commutator terms.  
A detailed inspection reveals that $F$ has the structure
\begin{equation}
    F \sim  \pi \lambda^\sigma \bigl(|\theta_t| + |\nabla^2 \theta|\bigr) 
        + \pi |\nabla \theta|^2 
        + \lambda^{-2}|\theta| \chi_{\operatorname{supp}\pi} 
        + \lambda^{-1} |\nabla \theta| \chi_{\operatorname{supp}\pi},
\end{equation}
up to inessential multiplicative constants.  
The terms involving $\lambda^\sigma$ arise from the H\"older continuity of the coefficients, while the terms with negative powers of $\lambda$ come from derivatives of the cutoff function.

\medskip

\noindent\textbf{Parabolic estimates.}  
For a fixed cutoff $\pi$ associated with some $(x_0,t_0)$, the standard $L^p$-theory for parabolic equations (see, e.g., Lady\v{z}enskaja, Solonnikov, and Ural'ceva \cite{LSU}, Giga and Sohr \cite{GS}) yields the estimate
\begin{equation}
    \|\pi \theta\|_{W^{2,1}_p} \leq C \bigl( \|F\|_{L^p(\operatorname{supp}\pi)} + \text{initial data contribution} \bigr),
\end{equation}
where $W^{2,1}_p$ denotes the parabolic Sobolev space with two spatial derivatives and one time derivative in $L^p$.  
The constant $C$ depends only on the ellipticity constants and on $p$, but not on $\lambda$ or on $(x_0,t_0)$.

\medskip

\noindent\textbf{Covering argument.}  
We now cover the domain $\mathbb{T}^d \times [0,T]$ by a finite collection of balls of radius $\lambda/2$ centered at points $(x_i,t_i)$.  
Let $N_0$ denote the covering number of such balls (independent from  $\lambda$ and $T$).  
Summing the $p$-th powers of the $W^{2,1}_p$ norms over this cover, we obtain
\begin{multline*}
    \|\theta\|_{W^{2,1}_p}^p \leq  
    \sum_{i} \|\pi_i \theta\|^p_{W^{2,1}_p} \leq 
    C_d N_0 \lambda^\sigma \|\theta\|^p_{W^{2,1}_p} 
    + C_d N_0 \bigl\| |\nabla \theta|^2 \bigr\|^p_{L^p} \\
    + C N_0 \lambda^{-2p} \|\nabla \theta\|_{L^p}^p 
    + C N_0 \lambda^{-p\sigma} \|\nabla \theta\|^p_{L^p} + C.
\end{multline*}
Here, the first term on the right-hand side comes from the $\lambda^\sigma$ part of $F$, the second term from the quadratic gradient term, and the remaining terms from the lower-order contributions.

\medskip

\noindent\textbf{Absorbing the highest-order term.}  
We now choose $\lambda$ so small that
\begin{equation}
    C_d N_0 \lambda^\sigma < \frac12.
\end{equation}
Since $N_0$ is independent form $\lambda$, 
thus,  such a $\lambda$ exists.  
With this choice, the term $C_d N_0 \lambda^\sigma \|\theta\|^p_{W^{2,1}_p}$ can be absorbed into the left-hand side.  
Consequently, $\lambda$ becomes a fixed parameter (depending only on the data and on $T$), and we obtain the simplified estimate
\begin{equation}
    \|\theta\|_{W^{2,1}_p} \leq C \bigl\| |\nabla \theta|^2 \bigr\|_{L^p} + C.
\end{equation}

\medskip

\noindent\textbf{Closing the estimate.}  
The quadratic gradient term is now controlled using the interpolation inequality established earlier (\ref{est:2p}).  
As $p$ is large enough we are ensured that
 $2(1-a) < 1$, and then Young's inequality yields
\[
\|\nabla \theta\|_{L^{2p}}^2 \leq \varepsilon \|\nabla^2 \theta\|_p + C_\varepsilon \|\theta\|_{W^\sigma_p}^{\frac{2a}{2a-1}}.
\]
The term $\varepsilon \|\nabla^2 \theta\|_p$ can be absorbed into the left-hand side of the parabolic estimate (since the $W^{2,1}_p$ norm controls $\|\nabla^2 \theta\|_p$).  
Thus, after absorption, we arrive at
\[
\|\theta\|_{W^{2,1}_p} \leq C(T) < \infty.
\]

\medskip

The constant $C(T)$ is finite for every finite $T$, although it may grow as $T \to \infty$ (for instance, exponentially in $T$).  
Nevertheless, this estimate shows that the solution possesses full parabolic regularity in the $W^{2,1}_p$ class, which is sufficient for all subsequent applications, including uniqueness and further nonlinear estimates.

\subsection{Local in time existence for the full system (\ref{nsf_toy})} \label{ssec:local-toy}
This part is dedicated to the local-in-time existence. The construction is classical,
following the framework of Lady\v{z}enskaja, Solonnikov, and Ural'tseva~\cite{LSU};
Nonetheless, it is important here, since it provides the H\"older continuity of the
temperature on the initial time interval---information that the De~Giorgi techniques
cannot supply, as the latter yield interior regularity only on compact subsets away
from $t=0$. 

The main proposition is stated as follows:
\begin{proposition}[Local-in-time existence and regularity]
	\label{prop:local}
	Let $p>d+2$ and let
	\[
	\rho_0\in W^1_p(\mathbb T^d),\qquad
	\theta_0\in W^{2-\frac2p}_p(\mathbb T^d),\qquad
	\rho_0\ge0,\quad \theta_0\ge\underline\theta_0>0.
	\]
	Then there exists $T_0>0$, depending only on the norms of the initial data, and a
	unique solution $(\rho,u,\theta)$ of \eqref{nsf_toy} on $[0,T]\times\mathbb T^d$ with
	\[
	\rho\in L^\infty(0,T_0;W^1_p)\cap W^1_p(0,T_0;L^p),\qquad
	\theta\in W^{2,1}_p(\mathbb T^d\times(0,T_0)),\qquad
	u\in L^\infty(0,T_0;W^2_p).
	\]
	In particular $\theta\in C^{1,1/2}(\mathbb T^d\times(0,T_0))$, and the temperature is
	strictly positive on $[0,T_0]$.
\end{proposition}

\begin{proof}[Scheme of proof]
The construction is classical; we indicate the structure, the full details being
standard for this parabolic regularity class.

\emph{Fixed-point map.} For $p>d+2$ we work above the critical regularity for
existence and apply the Schauder fixed-point theorem to the map
\[
\Phi:C^{1,1/2}(\mathbb T^d\times(0,T))\longrightarrow C^{1,1/2}(\mathbb T^d\times(0,T)),
\qquad \Phi(\widetilde\theta)=\theta,
\]
where $\theta$ solves the linear parabolic problem
\[
(c_v\rho+4\widetilde\theta^3)\partial_t\theta-\operatorname{div}(\kappa(\widetilde\theta)\nabla\theta)
=-c_v\rho u\cdot\nabla\widetilde\theta-\operatorname{div}(\widetilde\theta^4 u)
-\big(\{p(\rho,\widetilde\theta)\}-\rho^\gamma\big)\big(p(\rho,\widetilde\theta)-\{p(\rho,\widetilde\theta)\}\big),
\]
and $\rho,u$ are the (unique) solutions of the transport--elliptic subsystem
\begin{align}
\label{sys:001}\left\lbrace \begin{array}{rl} \partial_t\rho+\ebdiv (\rho u)&=0\\
\div u &= p(\rho,\widetilde \theta) - \{ p(\rho,\widetilde \theta) \}
\end{array}\right.
\end{align}
associated with the frozen temperature $\widetilde\theta$.

\emph{The density--velocity subsystem.} The existence of a solution $(\rho,u)$ to
\eqref{sys:001} is obtained by the iteration
\begin{align}\label{loc:toy:it1}
    \partial_t\rho_{n+1}+u_n\cdot\nabla\rho_{n+1}
+\rho_{n+1}\big(p(\rho_n,\widetilde\theta)-\{p(\rho_n,\widetilde\theta)\}\big)=0,
\qquad \rho_1\equiv\rho_0 ,
\end{align}
the velocity $u_n$ being recovered from $\rho_n,\widetilde\theta$ through the relation $\displaystyle \operatorname{div}u_n=p(\rho_n,\widetilde\theta)-\{p(\rho_n,\widetilde\theta)\}$
on $\mathbb T^d$. \par 
Differentiating \eqref{loc:toy:it1} in space yields
\begin{align*}
     & \partial_t \nabla \rho_{n+1}+ u_n \cdot \nabla \nabla \rho_{n+1} + \nabla u_n \nabla \rho_{n+1}
    \\
    &\qquad +\nabla \rho_{n+1} (p(\rho_n,\widetilde \theta) - \{ p(\rho_n,\widetilde \theta)\}) + 
    \rho_{n+1} ( \frac{\partial p}{\partial \rho} \nabla \rho_n + \frac{\partial p}{\partial \theta} \nabla \widetilde \theta) =0.
\end{align*}
Now testing the above equation with $\displaystyle|\nabla\rho_{n+1}|^{p-2}\nabla\rho_{n+1}$, we have 

\[
\sup_{[0,T]}\|\rho_{n+1}\|_{W^1_p}
\le \tfrac A2\,e^{CT(\|\rho_0\|_{W^1_p}+\|\widetilde\theta\|_{L^\infty_{t,x}})}\|\rho_0\|_{W^1_p}
+ T\,e^{CT(\|\rho_0\|_{W^1_p}+\|\widetilde\theta\|_{L^\infty_{t,x}})}\big(A\|\rho_0\|_{W^1_p}+C\|\nabla\widetilde\theta\|_{L^\infty_tL^p_x}\big),
\]
where, $A,B>2$ are fixed positive constants.
For $T$ small the prefactor $\displaystyle e^{CT(\|\rho_0\|_{W^1_p}+\|\widetilde\theta\|_{L^\infty_{t,x}})}$ is close to $1$ and the second term is
$O(T)$; hence the uniform bounds
\[
\sup_{[0,T]}\|\rho_n\|_{W^1_p}\le A\|\rho_0\|_{W^1_p},\qquad
\sup_{[0,T]}\|\nabla u_n\|_{W^1_p}\le B\|\rho_0\|_{W^1_p}
\]
propagate along the iteration, the second following from the elliptic relation. The
sequence is Cauchy in $L^2$, giving a unique limit $(\rho,u)$ solving
\eqref{sys:001}; the regularity uses only $\widetilde\theta\in L^\infty(0,T;W^1_p)$.
The bounds imply $\partial_t\rho\in L^\infty_tL^p_x$, hence
$\rho\in C^{\alpha,\alpha}(\mathbb T^d\times(0,T))$.

\emph{The temperature problem and closure.} With $(\rho,u)$ so obtained and
$\widetilde\theta\in C^{1,1/2}$, maximal parabolic regularity for the linear
equation defining $\Phi(\widetilde\theta)=\theta$ gives
\[
\|\theta\|_{W^{2,1}_p(\mathbb T^d\times(0,T))}\le 2C_0\|\theta_0\|_{W^{2-2/p}_p}.
\]
Since $p>d+2$, the embedding $W^{2,1}_p\hookrightarrow C^{1,1/2}$ is continuous and
compact, and on a short time interval it carries a positive power of the length,
\[
\|\theta\|_{C^{1,1/2}(\mathbb T^d\times(0,T))}\le C\,T^{\beta}\,\|\theta\|_{W^{2,1}_p},
\qquad \beta=\beta(d,p)>0 .
\]
Choosing $T$ small enough that
\[
2C_0C\,T^{\beta}\|\theta_0\|_{W^{2-2/p}_p}\le R,
\]
the map $\Phi$ sends the ball $\{\|\widetilde\theta\|_{C^{1,1/2}}\le R\}$ into
itself; we fix such a value and denote it $T_0>0$. Compactness of the embedding then
yields, by the Schauder theorem, a fixed point $\theta=\Phi(\theta)$, which together
with the associated $(\rho,u)$ solves \eqref{nsf_toy} on $[0,T_0]$. Uniqueness
follows from the contraction estimates for the subsystem and the linear temperature
equation. Finally, since the coefficients and source are bounded on $[0,T_0]$, the
minimum principle for the parabolic temperature equation gives the strict lower
bound $\theta\ge\underline\theta_0\,e^{-CT_0}>0$.
\end{proof}

\subsection{Proof of Theorem \ref{thm:global_toy}}
Before assembling the argument, we note how the a priori estimates deliver the full
regularity of $u$ and $\rho$ claimed in the statement.
\begin{itemize}[leftmargin=*]
    \item[-] The local theory (Proposition~\ref{prop:local}) and Lemma~\ref{lem:del-holder} provide the density bound $\rho\in L^\infty(0,T;W^1_p(\TT^d))$. 
    \item[-]Lemma~\ref{lem:W21p} controls
the temperature in $W^{2,1}_p$; 
\item[-]As a consequence, the pressure
$p(\rho,\theta)=\rho^\gamma+\rho\theta+\tfrac13\theta^4$ belongs to
$L^\infty(0,T;W^1_p(\TT^d))$.
\item[-] The elliptic relation
$\nabla u=\nabla^2\Delta^{-1}\big(p-\{p\}\big)$ and the Calder\'on--Zygmund theory on
$\TT^d$ then give $u\in L^\infty(0,T;W^2_p)$. 
\item[-] Feeding this velocity bound into the
continuity equation $\partial_t\rho=-\operatorname{div}(\rho u)$ yields
$\partial_t\rho\in L^\infty(0,T;L^p)$, so that
$\rho\in L^\infty(0,T;W^1_p)\cap W^1_p(0,T;L^p)$. Thus the full regularity of
$(\rho,u,\theta)$ stated in the theorem follows once the H\"older and parabolic bounds
below are in force.
\end{itemize}
We now assemble the global solution. Proposition~\ref{prop:local} provides a unique
local strong solution on a maximal interval $[0,T^*)$, with
$\theta\ge\underline\theta_0 e^{-Ct}>0$. On this interval the estimates are uniform:
the maximum principles give $0\le\rho\le\rho_{\mathrm{MAX}}$ and
$\theta\le\theta_{\mathrm{MAX}}$ (Lemmas~\ref{UB-den}, \ref{UB-temp}), whence
$\nabla u\in L^\infty_tBMO$ and $u\in L^\infty_tW^{1,q}$ for all $q<\infty$
(Lemma~\ref{lemma:u}); the lower bound $\theta\ge\underline\theta_0 e^{-Ct}>0$
(Lemma~\ref{lem:theta_positive}) makes the hypotheses of Theorem~\ref{thm:main}
available, yielding $\theta\in C^\sigma$ (Lemma~\ref{lem:theta-holder}) and, in turn,
$\rho\in C^\beta$ (Lemma~\ref{lem:del-holder}). With H\"older-continuous coefficients,
Lemma~\ref{lem:W21p} upgrades the temperature to
$\|\theta\|_{W^{2,1}_p(\TT^d\times(0,T))}\le C(T)<\infty$, and the first paragraph then
closes the bounds for $u$ and $\rho$. \par 
All these estimates depend only on the data and
remain finite on every finite interval; in particular the temperature stays strictly
positive and no controlling norm blows up as $t\uparrow T^*$. The continuation
criterion for strong solutions therefore forbids $T^*<\infty$, so $T^*=\infty$ and the
solution is global, with the stated bounds on each $[0,T]$. Uniqueness on each interval
is that of Proposition~\ref{prop:local}.
\qed

\appendix
\section{Proof of the Iteration Lemma} \label{sec:itlem}
We will use the following version of a classical iteration lemma. Although variants of this result can be found in many references \cite[Chap. 1]{Di1993}, the proof in the present form is elementary, and we include it here for completeness.
\begin{lemma}\label{lem:itera1}
Let $\{U_k\}_{k\geq 0}$ be a sequence of non-negative real numbers. Assume that there exist constants
$C>0$ and $\beta>1$ such that, for every $k\geq 1$,
$$
U_k \leq C^k U_{k-1}^{\beta}.
$$
If
$U_0 < C^{-\frac{\beta}{(\beta-1)^2}},$
then $U_k \to 0$ as $k\to\infty.$
\end{lemma}

\begin{proof}
By iterating the recurrence relation, we obtain
$$
U_k
\leq
C^{k+(k-1)\beta+(k-2)\beta^2+\cdots+2\beta^{k-2}+\beta^{k-1}}
U_0^{\beta^k}.
$$
Let $S_k=\sum_{j=1}^{k} j\beta^{k-j},$
then $U_k \leq C^{S_k}U_0^{\beta^k}.$
It remains to estimate $S_k$. Since
$$
S_k
=
\beta^{k-1}\sum_{j=1}^{k}j\left(\frac1\beta\right)^{j-1},
$$
and $\beta>1$, we have
$$
\sum_{j=1}^{k}j\left(\frac1\beta\right)^{j-1}
\leq
\sum_{j=1}^{\infty}j\left(\frac1\beta\right)^{j-1}
=
\frac{1}{\left(1-\frac1\beta\right)^2}
=
\frac{\beta^2}{(\beta-1)^2}.
$$
Therefore
$$
S_k
\leq
\frac{\beta^2}{(\beta-1)^2}\beta^{k-1}
=
\frac{\beta}{(\beta-1)^2}\beta^k.
$$
Consequently,
$$
U_k
\leq
C^{\frac{\beta}{(\beta-1)^2}\beta^k}U_0^{\beta^k}
=
\left(
C^{\frac{\beta}{(\beta-1)^2}}U_0
\right)^{\beta^k}.
$$
By the assumption
$U_0<C^{-\frac{\beta}{(\beta-1)^2}},$
we have
$C^{\frac{\beta}{(\beta-1)^2}}U_0<1.$
Hence
$$
\left(
C^{\frac{\beta}{(\beta-1)^2}}U_0
\right)^{\beta^k}
\to 0
\quad \text{as } k\to\infty .
$$
Thus $U_k\to 0$ as $k\to\infty$.
\end{proof}

\section{Proof of the De Giorgi Lemmas}\label{sec:degiorgi-proof}
Let $0<R_1\leq R_2\leq 2$ and $\phi$ be a standard cut-off function such that
\begin{align*}
    0\leq \phi\leq1,\qquad
    \phi=1\quad\text{in }B_{r_1},\qquad
    \operatorname{supp}\phi\subset B_{r_2},
    \qquad
    |\nabla\phi|\leq\frac{C}{R_2-R_1}.
\end{align*}

Before turning to the proof of the De Giorgi lemmas, we first establish an energy inequality by using \eqref{es:pre-energy} and taking $\psi=\phi^2$, which will be essential in the subsequent argument. Then one can get that for all $\tau\in(s,t)$, we have
\begin{align*}
		&\int_{B_2} E_b(\tau)\,\phi^2\,dx
+\int_s^\tau\!\!\int_{B_2} \,\kappa(\theta_r)\,|\nabla v_C|^2\,\phi^2\\
		&\leq\;
		\int_{B_2} E_b(s)\,\phi^2\,dx+\int_s^\tau\!\!\int_{B_2} E_b\,r\,u_r\cdot\nabla(\phi^2)
		\;-\;\int_s^\tau\!\!\int_{B_2} v_C\,\kappa(\theta_r)\nabla v_C\cdot\nabla(\phi^2)\\
		&\,\,+\int_s^\tau\!\!\int_{B_2} \left(v_C F + \left(H_C(v)- P(v) v_C\right) r\operatorname{div}u_r\right)\phi^2.
	\end{align*}
By assumption, for every $\zeta$ in the relevant range,
    $0<\underline\lambda\leq P'(\zeta)\leq \overline\lambda,$
then 
\begin{align}\label{es:Hcv}
    \frac{\underline\lambda}{2}(v-C)_+^2
    \leq
    H_C(v)
    \leq
    \frac{\overline\lambda}{2}(v-C)_+^2,
\end{align}
combine this with $\rho_r\geq 0, 0\leq v_C\leq 1$ implies
\begin{align*}
    \int_{B_2} E_b(\tau)\,\phi^2\,dx\geq \frac{\underline\lambda}{2}\int_{B_{R_1}}(v-C)_+^2(\tau) dx,
\end{align*}
and
\begin{align*}
    \int_{B_2} E_b(s)\,\phi^2\,dx\leq \frac{\overline\lambda}{2}\int_{B_{R_2}}(v-C)_+^2(s) dx + \|\rho_r\|_{L^\infty}\int_{B_{R_2}}(v-C)_+^2(s) dx.
\end{align*}

By using the fact that $0\leq v_C\leq 1$ and \eqref{es:Hcv} again, one can get
\begin{align}\label{es:rho-u}
&\left|\int_s^\tau\!\!\int_{B_2} E_br\,u_r\cdot\nabla(\phi^2)\right|\nonumber\\
&\leq\left|\int_s^\tau\int_{B_{R_2}} r \rho_r u_r v_C^2 \phi\, \nabla\phi dx d\tau\right|+\overline{\lambda}\int_s^t\int_{B_{R_2}} r  u_r v_C^2 \phi\, |\nabla\phi| dx d\tau\nonumber\\
&\leq C\int_s^\tau\int_{B_{R_2}} r^2 (\rho_r^2+1) u_r^2 v_C^2 dx d\tau + C \int_s^\tau\int_{B_{R_2}}  v_C^2 |\nabla\phi|^2 dx d\tau\\
&\leq C r^{2-\frac{5}{p}} \|u^2\|_{L^p(Q_{2r}(z_0))}\left|\{v_C>0\}\cap((s,\tau)\times B_{R_2})\right|^{1-\frac1p} + C 
\frac{1}{(R_2-R_1)^2}
\int_s^\tau\int_{B_{R_2}}v_C^2\,dx\,d\tau.\nonumber
\end{align}

For the diffusion term, using the ellipticity of $\kappa(\theta_r)$ and Young's
inequality, we get
\begin{align*}
\int_s^\tau\!\!\int_{B_2} \kappa(\theta_r)|\nabla v_C|^2\phi^2
\geq
\underline{\kappa}\int_s^\tau\!\!\int_{B_2}  |\nabla v_C|^2\phi^2\,dx
\end{align*}
and
\begin{align*}
&
2\int_s^\tau\!\!\int_{B_2} \kappa(\theta_r)v_C\phi\nabla v_C\cdot\nabla\phi\,dx
\\
&\leq 
\frac{\underline{\kappa}}{2}\int_s^\tau\!\!\int_{B_2} |\nabla v_C|^2\phi^2\,dx d\tau
+ C(\underline{\kappa},\overline{\kappa})
\int_s^\tau\!\!\int_{B_2} v_C^2|\nabla\phi|^2\,dx d\tau\\
&\leq 
\frac{\underline{\kappa}}{2}\int_s^\tau\!\!\int_{B_2} |\nabla v_C|^2\phi^2\,dx d\tau
+ C(\underline{\kappa},\overline{\kappa})
\frac{1}{(R_2-R_1)^2}
\int_s^\tau\int_{B_{R_2}}v_C^2\,dx\,d\tau.
\end{align*}
The right-hand side is estimated by Hölder's inequality:
\begin{align*}
    \int Fv_C\phi^2\,dxdt
    \leq
    C\|F\|_{L^p(Q_{R_2}\cap \{v_C>0\})}
    \left|\{v_C>0\}\right|^{1-\frac1p}.
\end{align*}

Combining these estimates and using the equivalence between $H_C(v)$ and
$v_C^2$, we obtain the following modified De Giorgi type energy estimate:
for every $-4\leq s<t\leq 0$ and every $0<R_1<R_2\leq 2$,
\begin{align}\label{es:energy}
&
\sup_{\tau\in(s,t)}
\int_{B_{R_1}} v_C^2(\tau,x)\,dx
+
\int_s^t\int_{B_{R_1}}|\nabla v_C|^2\,dx\,d\tau\nonumber
\\
&\qquad\leq
\gamma_1
\int_{B_{R_2}} v_C^2(s,x)\,dx
+
\frac{\gamma_2}{(R_2-R_1)^2}
\int_s^t\int_{B_{R_2}}v_C^2\,dx\,d\tau\\
&\qquad\qquad+
\gamma_3\left(r^{2-\frac{5}{p}} (\|u^2\|_{L^p(Q_{2r}(z_0))}+\|F\|_{L^p(Q_{R_2}\cap \{v_C>0\})}\right)
\left|\{v_C>0\}\cap((s,t)\times B_{R_2})\right|^{1-\frac1p} ,\nonumber\\
&\qquad\leq
\gamma_1
\int_{B_{R_2}} v_C^2(s,x)\,dx
+
\frac{\gamma_2}{(R_2-R_1)^2}
\int_s^t\int_{B_{R_2}}v_C^2\,dx\,d\tau\nonumber\\
&\qquad\qquad+
\gamma_3\left(R_\mathcal K^{2-\frac{5}{p}} (\| u^2\|_{L^p(Q_{2R_\mathcal K}(z_0))}+\|F\|_{L^p(Q_{R_2}\cap \{v_C>0\})}\right)
\left|\{v_C>0\}\cap((s,t)\times B_{R_2})\right|^{1-\frac1p} ,\nonumber
\end{align}
where the constants $\gamma_1,\gamma_2,\gamma_3$ depend only on $\underline{\lambda}$ and $\overline{\lambda}$.

\subsection{Proof of Lemma \ref{lem:osc-reduction}}
We now perform the standard De Giorgi iteration. Define
\begin{align*}
    C_k:=\frac12(1-2^{-k}),
    \qquad
    r_k:=\frac12(1+2^{-k}),
\end{align*}
and
\begin{align*}
    Q_k:=(-r_k^2,0)\times B_{r_k},
    \qquad
    V_k:=\int_{Q_k}((v-C_k)_+)^2\,dxdt .
\end{align*}
Choose cut-off functions $\phi_k$ such that
\begin{align*}
    \phi_k=1\quad\text{in }B_{r_{k}},
    \qquad
    \operatorname{supp}\phi_k\subset B_{r_{k-1}},
    \qquad
    |\nabla\phi_k|\leq C2^k .
\end{align*}
Applying the above energy estimate \eqref{es:energy} with $C=C_{k+1}, R_1=r_{k+1},\,R_2=r_k, -r_k^2\leq s\leq -r_{k+1}^2\leq t\leq 0$ gives
\begin{align*}
&
\sup_{\tau\in(-r_{k+1}^2,0)}
\int_{B_{r_{k+1}}} v_{k+1}^2(\tau,x)\,dx
+
\int_{-r_{k+1}^2}^0\int_{B_{r_{k+1}}}|\nabla v_{k+1}|^2\,dx\,d\tau
\\
&\qquad\leq
\gamma_1
\int_{B_{r_k}} v_{k+1}^2(s,x)\,dx
+
\frac{\gamma_2}{(r_k-r_{k+1})^2}
\int_{-r_{k}^2}^0\int_{B_{r_k}}v_{k+1}^2\,dx\,d\tau\\
&\qquad\,+
\gamma_3\left(R_\mathcal K^{2-\frac{5}{p}} (\|u^2\|_{L^p(Q_{2R_\mathcal K}(z_0))}+\|F\|_{L^p(Q_{R_2}\cap \{v_{C_{k+1}}>0\})}\right)
\left|\{v_{k+1}>0\}\cap((-r_k^2,0)\times B_{r_k})\right|^{1-\frac1p},
\end{align*}
then taking the mean value in $s$ between $-r_k^2$ and $-r_{k+1}^2$, we get
\begin{align*}
&
\sup_{\tau\in(-r_{k+1}^2,0)}
\int_{B_{r_{k+1}}} v_{k+1}^2(\tau,x)\,dx
+
\int_{-r_{k+1}^2}^0\int_{B_{r_{k+1}}}|\nabla v_{k+1}|^2\,dx\,d\tau
\\
&\qquad\leq
\gamma_1
\frac{1}{r_k^2-r_{k+1}^2}\int_{-r_k^2}^{r_{k+1}^2}\int_{B_{r_k}} v_{k+1}^2(\tau,x)\,dx d \tau
+
\frac{\gamma_2}{(r_k-r_{k+1})^2}
\int_{-r_{k}^2}^0\int_{B_{r_k}}v_{k+1}^2\,dx\,d\tau\\
&\,\,\,\,\,\qquad
+
\gamma_3\left(R_\mathcal K^{2-\frac{5}{p}} (\|u^2\|_{L^p(Q_{2R_\mathcal K}(z_0))}+\|F\|_{L^p(Q_{R_2}\cap \{v>0\})}\right)
\left|\{v_{k+1}>0\}\cap((-r_k^2,0)\times B_{r_k})\right|^{1-\frac1p},
\end{align*}
where we use the fact that $\{v_{C_{k+1}}>0\}\subset \{v>0\}$,
which implies 
\begin{align*}
&
\sup_{-r_{k+1}^2<t<0}
\int_{B_{r_{k+1}}}v_{k+1}^2(t,x)\,dx
+
\int_{Q_{k+1}}|\nabla v_{k+1}|^2\,dxdt
\\
&\qquad
\leq
C^k
\left(
V_k
+
\left|\{v_{k+1}>0\}\cap Q_k\right|^{1-\frac1p}
\right),
\end{align*}
where the $C$ depends only on $\underline{\lambda},\overline{\lambda} $ and $\mathcal K$.
On the set $\{v_{k+1}>0\}$, we have     $v>C_{k+1},$
and hence
\begin{align*}
    v_k=(v-C_k)_+
    \geq
    C_{k+1}-C_k
    =
    2^{-k-2}.
\end{align*}
Therefore, by Chebyshev's inequality,
\begin{align*}
    \left|\{v_{k+1}>0\}\cap Q_k\right|
    \leq
    C2^{2k}V_k .
\end{align*}
Using the parabolic Sobolev inequality, we get
\begin{align*}
    \|v_{k+1}\|_{L^{\frac{10}{3}}(Q_{k+1})}^2
    \leq
    C^k
    \left(
    V_k
    +
    V_k^{1-\frac1p}
    \right).
\end{align*}
Then one can get
\begin{align*}
V_{k+1}
&=
\int_{ Q_{k+1}} |v_{k+1}|^2\,dx\,dt=
\int_{Q_{k+1}} |v_{k+1}|^2
\mathbf 1_{\{v_{k+1}>0\}}\,dx\,dt  \\
&=
\int_{ Q_{k+1}} |v_{k+1}|^2
\mathbf 1_{\{v_k>2^{-(k+2)}\}}\,dx\,dt  \\
&\leq
\|v_{k+1}\|_{L^{\frac{10}{3}}(Q_{k+1})}^2
\left|
\{v_k^2>2^{-2(k+2)}\}\cap  Q_{k+1}
\right|^{\frac{2}{5}}  \\
&\leq
C^k\left(V_k +V_k^{1-\frac1p}\right)
\left(
2^{2(k+2)}
\int_{ Q_{k+1}} |v_k|^2\,dx\,dt
\right)^{\frac{2}{5}}\leq
C^k \left(V_k^{\frac{7}{5}} + V_k^{\frac{7}{5}-\frac1p}\right) .
\end{align*}
Since $v\in L^\infty(-4,0;L^5(B_2))\subset L^2(-4,0;L^2(B_2))$, one can get for all $k$, $\displaystyle V_k=\int_{Q_k}((v-C_k)_+)^2 dx dx\leq \int_{Q_2}v^2dxdt\leq C$, which implies that $V_{k+1}\leq C^kV_k^{\frac75-\frac1p}$.

By Lemma \ref{lem:itera1}, there exists
$\delta=\delta(\underline\lambda,\overline\lambda, \mathcal K)>0$ such that if
\begin{align*}
    V_0=\int_{Q_1}|v_+|^2\,dxdt\leq\delta,
\end{align*}
then $V_k\to0$ as $k\to\infty.$
Passing to the limit gives
    $\displaystyle \int_{Q_{1/2}}\left(v-\tfrac12\right)_+^2\,dxdt=0.$
Therefore $v\leq\frac12$ in $Q_{1/2}.$
\subsection{Proof of Lemma \ref{lem:mid-value}}
Refer to \cite[Lemma 3.1]{JG20}, we have the following lemma, it is an alternative
version of the De Giorgi isoperimetric inequality \cite[Lemma 10]{Vasseur16}.
\begin{lemma}\label{lem:DG-iso}
Let $u\in H^1(B_L)$. Then for all $(k,l)\in\mathbb R^2$ with $k\leq l$, we have
\begin{align}
(l-k)\left|\{u\leq k\}\cap B_L\right|
\left|\{u\geq l\}\cap B_L\right|
\leq
C
\left|\{k<u<l\}\cap B_L\right|^{\frac12}
\left(
\int_{B_L}\left|\nabla(u-k)_+(x)\right|^2\,dx
\right)^{\frac12},
\end{align}
where $C$ only depends on $L$.
\end{lemma}

The following auxiliary intermediate value lemma, which includes an error term, plays an important role in the main proof of Lemma \ref{lem:mid-value}.
\begin{lemma}\label{lem:pre-mid-value}
Under the assumptions of Lemma \ref{lem:mid-value}, for all $(t_1,t_2,\tau)\in(-2,0)^3$ such that
$-2<t_1<\tau<t_2<0$, we have
\begin{align*}
&\left|{v\geq \tfrac12}\cap((\tau,t_2)\times B_1)\right|
\left|{v\leq 0}\cap((t_1,\tau)\times B_1)\right| \\
&\qquad\leq
C_0\left|{0<v<\tfrac12}\cap((t_1,\tau)\times B_{\frac54})\right|^{\frac12}
+
C_0(t_2-t_1)^{3-\frac1p},
\end{align*}
where $C_0$ depends only on $\underline{\lambda}, \overline{\lambda}$ and $R$.
\end{lemma}
\begin{proof}[Proof of Lemma~\ref{lem:pre-mid-value}]
Referring to \eqref{es:energy} with $C=0, R_1=1, R_2=\frac54$, one can get
 \begin{align*}
\int_{B_1} v_+^2(t,x)\,dx
&\leq
\gamma_1\int_{B_{\frac54}} v_+^2(s,x)\,dx
+
\gamma_2
\int_s^t\int_{B_{\frac54}} v_+^2(\tau,x)\,dx\,d\tau  \\
&\quad
+
\gamma_3\left(R_\mathcal K^{2-\frac{5}{p}} (\|u^2\|_{L^p(Q_{2R_\mathcal K}(z_0))}+\|F\|_{L^p(Q_{\frac54}\cap \{v>0\})}\right)
\left|\{v_+>0\}\cap((s,t)\times B_{\frac54})\right|^{1-\frac1p}\\
&\leq
\gamma_1\int_{B_{\frac54}} v_+^2(s,x)\,dx
+
\gamma_2
\int_s^t\int_{B_{\frac54}} v_+^2(\tau,x)\,dx\,d\tau  
+
C_0 \left|\{v_+>0\}\cap((s,t)\times B_{\frac54})\right|^{1-\frac1p},
\end{align*}
where the constants $\gamma_1,\gamma_2$ depend only on $\underline{\lambda}$ and $\overline{\lambda}$, $C_0$ depends only on $\underline{\lambda}, \overline{\lambda}$ and $\mathcal K$.

First, we bound the left hand side from below:
\begin{align*}
\int_{B_1} v_+^2(t,x)\,dx
&\geq
\int_{\{x\in B_1,\ v(t,x)\geq \frac12\}}
\frac14\,dx  =
\frac14
\left|\left\{v(t,\cdot)\geq \tfrac12\right\}\cap B_1\right|.
\end{align*}

Second, we bound from above each term of the right hand side. The first term gives
\begin{align*}
\int_{B_{\frac54}} v_+^2(s,x)\,dx
&\leq
\int_{\{x\in B_{\frac54},\ 0<v(s,x)<\frac12\}}
v_+^2(s,x)\,dx
+
\int_{\{x\in B_{\frac54},\ v(s,x)\geq\frac12\}}
v_+^2(s,x)\,dx  \\
&\leq
C_0\left(
\left|\left\{0<v(s,\cdot)<\tfrac12\right\}\cap B_{\frac54}\right|
+
\left|\left\{v(s,\cdot)\geq\tfrac12\right\}\cap B_{\frac54}\right|
\right).
\end{align*}

The second term gives
\begin{align*}
\gamma_2
\int_s^t\int_{B_{\frac54}} v_+^2(\tau,x)\,dx\,d\tau
\leq
C_0(t-s).
\end{align*}

And the third term gives
\begin{align*}
C_0
\left|\{v_+>0\}\cap((s,t)\times B_{\frac54})\right|^{1-\frac1p} 
\leq
C_0(t-s)^{1-\frac1p}.
\end{align*}
Collecting the above estimates yields
\begin{align*}
    \left|\left\{v(t,\cdot)\geq \tfrac12\right\}\cap B_1\right|\leq C_0 \left(
\left|\left\{0<v(s,\cdot)<\tfrac12\right\}\cap B_{\frac54}\right|
+
\left|\left\{v(s,\cdot)\geq\tfrac12\right\}\cap B_{\frac54}\right|+(t-s)+(t-s)^{1-\frac1p}\right).
\end{align*}
Multiplying by $\left|\left\{v(s,\cdot)\leq0\right\}\cap B_{\frac54}\right|$ and taking advantages of Lemma~\ref{lem:DG-iso}, one can get
\begin{align*}
    &\left|\left\{v(t,\cdot)\geq \tfrac12\right\}\cap B_1\right|\left|\left\{v(s,\cdot)\leq0\right\}\cap B_{1}\right|\\
    &\leq C_0 \left(
\left|\left\{0<v(s,\cdot)<\tfrac12\right\}\cap B_{\frac54}\right|
+
\left|\left\{0<v(s,\cdot)<\tfrac12\right\}\cap B_{\frac54}\right|^{\frac12}+(t-s)+(t-s)^{1-\frac1p}\right)\\
    &\leq C_0 \left(\left|\left\{0<v(s,\cdot)<\tfrac12\right\}\cap B_{\frac54}\right|^{\frac12}+(t-s)+(t-s)^{1-\frac1p}\right).
\end{align*}
Integrating over $s\in[t_1,\tau]$ and $t\in[\tau,t_2]$ with $-2\leq t_1<\tau<t_2\leq 0$, one can get the desired result.
\end{proof}
Now we can give the proof of Lemma~\ref{lem:mid-value}.
\begin{proof}[Proof of Lemma~\ref{lem:mid-value}]
   Let $n\in\mathbb N\setminus\{0\}$ choose later, for $k\in [0,2n]\cap\mathbb N$, set $t_k:= -2+\frac{k}{n}$.
   Remember that $\widetilde Q=(-2,-1)\times B_1$, then there exists $i\in [1,n]$ such that
   \begin{align*}
       \left|v\leq 0, (t_{i-1},t_i)\times B_1\right|\geq \frac{|v\leq 0,\widetilde Q|}{n},
   \end{align*}
   which implies
   \begin{align}\label{es:set-low-bd}
       \left|v<\tfrac12, (t_{i-1},t_i)\times B_1\right|\geq \frac{|v\leq 0,\widetilde Q|}{2n},
   \end{align}
   and there exists $j\in [n,2n-1]$ such that
    \begin{align*}
       \left|v\geq\tfrac12, (t_{j},t_{j+1})\times B_1\right|\geq \frac{|v\geq \tfrac12, Q_1|}{n}\geq \frac{|v\geq \tfrac12, Q_1|}{2n}.
   \end{align*}
   \textbf{Case 1:} there exists $m\in[i,2n-1]$ such that $m+1$ does not satisfy \eqref{es:set-low-bd}. Letting $m^*$ be the first integer $m$ such that $m+1$ does not satisfy \eqref{es:set-low-bd}, one can get
   \begin{align}\label{es:low-level}
       & \left|v<\tfrac12, (t_{m^*-1},t_{m^*})\times B_1\right|\geq \frac{|v\leq 0,\widetilde Q|}{2n},\\
        &\left|v<\tfrac12, (t_{m^*},t_{m^*+1})\times B_1\right|<\frac{|v\leq 0,\widetilde Q|}{2n},\nonumber
   \end{align}
then one can get
   \begin{align}\label{es:high-level}
   \left|v\geq\tfrac12, (t_{m^*},t_{m^*+1})\times B_1\right|=&\frac{|B_1|}{n}-\left|v<\tfrac12, (t_{m^*},t_{m^*+1})\times B_1\right|\nonumber\\
   &\geq \frac{|B_1|}{n}-\frac{|v\leq 0, \widetilde Q|}{2n}
   \geq \frac{|B_1|}{2n}\geq \frac{|v\geq \tfrac12, Q_1|}{2n}.
   \end{align}\\
   \textbf{Case 2:} for all $m\in[i,2n-1]$, $m+1$ does satisfy \eqref{es:set-low-bd}.
   Then letting $m^*=j$, we have \eqref{es:low-level}$_1$ and \eqref{es:high-level} directly.
   Then in all cases, by using Lemma~\ref{lem:pre-mid-value}, \eqref{es:low-level}$_1$ and \eqref{es:high-level}, one can get
   \begin{align*}
       &\frac{|v\leq 0,\widetilde Q|}{2n}\frac{|v\geq \tfrac12, Q_1|}{2n}\\
       &\leq \left|v<\tfrac12, (t_{m^*-1},t_{m^*})\times B_1\right| \left|v\geq\tfrac12, (t_{m^*},t_{m^*+1})\times B_1\right|\\
      &\leq  \left(\left|v\leq 0, (t_{m^*-1},t_{m^*})\times B_1\right| + \left|0<v<\tfrac12, (t_{m^*-1},t_{m^*})\times B_1\right|\right)\left|v\geq\tfrac12, (t_{m^*},t_{m^*+1})\times B_1\right|\\
      &\leq  C \left(\left|0<v<\tfrac12, (t_{m^*-1},t_{m^*})\times B_{\frac{5}{4}}\right|^{\frac12}+ (\frac{1}{n})^{3-\frac1p}\right)\\
         &\leq  C \left(\left|0<v<\tfrac12, (-2,0)\times B_{\frac{5}{4}}\right|^{\frac12}+ (\frac{1}{n})^{3-\frac1p}\right),
     \end{align*}
     which implies 
        \begin{align*}
       |v\leq 0,\widetilde Q||v\geq \tfrac12, Q_1|\leq  C \left(n^2\left|0<v<\tfrac12, (-2,0)\times B_{\frac{5}{4}}\right|^{\frac12}+ (\frac{1}{n})^{1-\frac1p}\right),
     \end{align*}
     from this inequality, we deduce that $\left|0<v<\tfrac12, (-2,0)\times B_{\frac{5}{4}}\right|>0$. 
If $\left|0<v<\tfrac12, (-2,0)\times B_{\frac{5}{4}}\right|=0$, then 
       $|v\leq 0,\widetilde Q||v\geq \tfrac12, Q_1|\leq  C \left(\frac{1}{n}\right)^{1-\frac1p},\,n$ large enough will give a contradict.
     Then we take $n$ such that $n^2\left|0<v<\tfrac12, (-2,0)\times B_{\frac{5}{4}}\right|^{\frac12}\approx (\frac{1}{n})^{1-\frac1p}$, one can get
      \begin{align*}
       |v\leq 0,\widetilde Q||v\geq \tfrac12, Q_1|\leq  C \left|0<v<\tfrac12, (-2,0)\times B_{\frac{5}{4}}\right|^{\frac{p-1}{6p-2}},
     \end{align*}
where $C$ depends only on $\underline{\lambda}, \overline{\lambda}$ and $\mathcal K$,    which completes the proof of Lemma~\ref{lem:mid-value}.
\end{proof}

\section{Truncated energy inequality}\label{proof:ineqTr}
\begin{proof}[Proof of Corollary~\ref{lem:renormTr}]
First we derive \eqref{eq:trunc}, the goal is to insert $b(\theta)=\tfrac12(\theta-C)_+^2$ into \eqref{eq:renorm}. Here
$b'=\theta_C:=(\theta-C)_+$, $b''=\mathbf 1_{\{\theta>C\}}$ and $(b')^2=2b$, so $b$ is
only $C^{1,1}$ (its second derivative jumps at the corner $\theta=C$) and is not
constant at infinity; on both counts it fails to be admissible in
Definition~\ref{def:weakrenorm1}. We remedy the two defects in turn: first a smooth
approximation $b_K$ of $b$ that rounds the corner, then a truncation $\beta_K=T_K\circ b_K$
that makes the renormalizer constant at infinity.

\emph{Smooth approximation of $b$ (corner).}
For $K>0$ fix $\chi_K\in C^\infty([0,\infty))$ with
\[
0\le\chi_K\le1,\qquad \chi_K\equiv0\ \text{on}\ [0,C],\qquad
\chi_K\equiv1\ \text{on}\ [C+\tfrac1K,\infty),\qquad \chi_K\to\mathbf 1_{\{\theta>C\}}\ \text{a.e.},
\]
and define
\[
b_K'(\theta)=\int_0^\theta\chi_K,\qquad b_K(\theta)=\int_0^\theta b_K',\qquad b_K''=\chi_K .
\]
Then $b_K\in C^2([0,\infty))$ rounds the corner over $[C,C+\tfrac1K]$, and
\[
0\le b_K''\le1,\qquad 0\le b_K'\le\theta_C,\qquad 0\le b_K\le b,\qquad (b_K')^2\le 2b_K,
\]
with $b_K'\to\theta_C$ and $b_K\to b$ a.e.\ as $K\to\infty$.

\emph{Truncation (tail).}
Fix $T\in C^2([0,\infty))$ with
\begin{equation}\label{eq:T}
T(x)=x\ \text{on}\ [0,\tfrac12],\qquad
T\ \text{concave on}\ [\tfrac12,1],\qquad
T(x)=1\ \text{on}\ [1,\infty),\qquad 0\le T'\le1,
\end{equation}
and set $T_K(x):=K\,T(x/K)$, so that $T_K'(x)=T'(x/K)$ and $T_K''(x)=\tfrac1K T''(x/K)$.
Define
\[
\beta_K:=T_K\circ b_K\in C^2([0,\infty)),\qquad
\beta_K'=T_K'(b_K)\,b_K',\qquad
\beta_K''=T_K'(b_K)\,b_K''+T_K''(b_K)\,(b_K')^2 .
\]

\emph{Step 1 ($\beta_K$ is an admissible renormalizer).}
Since $T_K(x)=K$ for $x\ge K$ and $b_K(\theta)\ge K$ once $\theta\ge M_K:=C+\sqrt{2K}+\tfrac1K$,
$\beta_K$ is constant on $[M_K,\infty)$. On $\{\theta<M_K\}$, $b_K'\le\theta_C<\sqrt{2K}$,
so $0\le\beta_K'\le\sqrt{2K}$; and using $b_K''\le1$ and $(b_K')^2\le2b_K$,
\[
|\beta_K''|
\le \|T'\|_\infty\,b_K''+\tfrac1K|T''(b_K/K)|\,(b_K')^2
\le \|T'\|_\infty+2\|T''\|_\infty ,
\]
both finite. Hence $\beta_K$ satisfies the hypotheses of
Definition~\ref{def:weakrenorm1}, and \eqref{eq:renorm} applied to $\beta_K$ gives, for
a.e.\ $\tau<s$,
\begin{align}\label{eq:bK-id}
\Big[\int_\Omega E_{\beta_K}\psi\Big]_\tau^s
&+\int_\tau^s\!\!\int_\Omega \beta_K''(\theta)\kappa(\theta)|\nabla\theta|^2\psi
\le\int_\tau^s\!\!\int_\Omega E_{\beta_K}u\cdot\nabla\psi
-\int_\tau^s\!\!\int_\Omega \beta_K'(\theta)\kappa(\theta)\nabla\theta\cdot\nabla\psi
\nonumber\\
&+\int_\tau^s\!\!\int_\Omega \beta_K'(\theta)\mathcal D\,\psi
-\int_\tau^s\!\!\int_\Omega\Big(\beta_K'\rho\theta+\big(\tfrac43\theta^4\beta_K'-\eta_{\beta_K}\big)\Big)\operatorname{div}u\,\psi,
\end{align}
where $\mathcal D=\mu(\theta)|Du-\tfrac13(\operatorname{div}u)\mathbb I|^2+\eta(\theta)|\operatorname{div}u|^2$.
Fix $\tau,s$ outside the null sets attached to a sequence $K\to\infty$, and pass to the
limit.

\emph{Step 2 (pointwise limits and uniform bounds).}
As $K\to\infty$, $b_K''=\chi_K\to\mathbf 1_{\{\theta>C\}}$, $b_K'\to\theta_C$,
$b_K\to b$ a.e., while $T_K'(b_K)=T'(b_K/K)\to T'(0)=1$ a.e.\ \textup{(}as
$b_K/K\to0$\textup{)}. Hence a.e.
\[
\beta_K'=T_K'(b_K)b_K'\to\theta_C,\qquad
\beta_K\to b,\qquad
\eta_{\beta_K}\to\eta_b,\qquad
E_{\beta_K}=\rho\,\beta_K+\eta_{\beta_K}\to E_b,
\]
with the uniform bounds $\beta_K'\le\theta_C$, $\beta_K\le b$, $\eta_{\beta_K}\le\eta_b$,
$0\le E_{\beta_K}\le E_b$.

\emph{Step 3 (limit of the bracket).}
For a.e.\ fixed $t\in\{\tau,s\}$, $E_{\beta_K}(t)\le E_b(t)\in L^1(\Omega)$ and
$E_{\beta_K}(t)\to E_b(t)$ a.e.; dominated convergence gives
$\displaystyle \int_\Omega E_{\beta_K}(t)\psi\to\int_\Omega E_b(t)\psi$, hence
$\displaystyle \left[\int_\Omega E_{\beta_K}\psi\right]_\tau^s\to\big[\int_\Omega E_b\psi\big]_\tau^s$.

\emph{Step 4 (limit of the conduction term).}
Split $\beta_K''=\underbrace{T_K'(b_K)\,b_K''}_{=:P_K}+\underbrace{T_K''(b_K)(b_K')^2}_{=:R_K}$.
The remainder $R_K$ is supported in $A_K:=\{C+\sqrt K<\theta<M_K\}$ (where $T_K''(b_K)\neq0$),
with $|R_K|\le2\|T''\|_\infty$; since $\kappa(\theta)|\nabla\theta|^2=(1+\theta^3)|\nabla\theta|^2\in L^1(Q)$
and $\mathbf 1_{A_K}\to0$ a.e., dominated convergence gives
\[\int_\tau^s\!\int_\Omega R_K\,\kappa|\nabla\theta|^2\psi\to0.\] For the principal part,
$0\le P_K=T_K'(b_K)\chi_K\le1$ and $P_K\to\mathbf 1_{\{\theta>C\}}$ a.e.; since
$\kappa|\nabla\theta|^2\psi\in L^1(Q)$, dominated convergence gives
\[
\int_\tau^s\!\!\int_\Omega P_K\,\kappa|\nabla\theta|^2\psi
\ \longrightarrow\
\int_\tau^s\!\!\int_{\{\theta>C\}}\kappa(\theta)|\nabla\theta_C|^2\psi .
\]
Hence the conduction term converges to
$\int_\tau^s\!\int_{\{\theta>C\}}\kappa|\nabla\theta_C|^2\psi$.

\emph{Step 5 (limits of the right-hand terms).}
Using $\beta_K'\le\theta_C$, $E_{\beta_K}\le E_b$, $\eta_{\beta_K}\le\eta_b$, the
exponent conditions $\tfrac1p+\tfrac1q<\tfrac25$, $\tfrac1p+\tfrac1r<\tfrac25$, $r>5$,
together with $\theta\in L^{25/3}(Q)$, $\theta^{3/2}\nabla\theta\in L^2(Q)$ and the
growth $\mu(\theta)+\eta(\theta)\lesssim1+\theta$, each integrand is dominated by a
fixed $L^1(Q)$ function:
\[
E_b\,|u||\nabla\psi|,\quad \theta_C\kappa|\nabla\theta||\nabla\psi|,\quad
\theta_C\,\mathcal D\,\psi,\quad
\big(\theta_C\rho\theta+\tfrac43\theta^4\theta_C+\eta_b\big)|\operatorname{div}u|\psi
\ \in L^1(Q).
\]
By dominated convergence, the four right-hand integrals of \eqref{eq:bK-id} converge to
the corresponding integrals with $\beta_K'$ replaced by $\theta_C$ and $\eta_{\beta_K}$
by $\eta_b$.

\emph{Step 6 (conclusion).}
Passing to the limit $K\to\infty$ in \eqref{eq:bK-id} and retaining the nonnegative
coercive term $\displaystyle \int_\tau^s\!\int_{\{\theta>C\}}\kappa|\nabla\theta_C|^2\psi$ on the left
yields \eqref{eq:trunc}.

The proof of \eqref{eq:trunc-minus} is analogous and even simpler. Indeed, since
$b_C^-(\theta)=\frac12(C-\theta)_+^2$ is already constant on $[C,\infty)$,
only the corner at $\theta=C$ needs to be smoothed; applying the same limiting
argument, with $(b_C^-)'=-\theta_C^-$ and
$\nabla\theta_C^-=-\mathbf 1_{\{\theta<C\}}\nabla\theta$, yields
\eqref{eq:trunc-minus}.
\end{proof}

\section*{Acknowledgment}

The work of N.C. is supported by the NAWA ULAM grant BPN/SEL/2025/1/00005/U/00001. P.B.M. and T.P. have been partly supported by the Narodowe
Centrum Nauki (NCN) Grant No. 2022/45/B/ST1/03432 (OPUS).
J.Y. is supported by a scholarship (No. 202506040134) of China Scholarship Council. 
The work was done when J.Y. was a visiting Ph.D. student in University of Warsaw, 
and she would like to express her deep gratitude for the hospitality by University of Warsaw.





	
\end{document}